\documentclass[a4paper,12]{article}
\usepackage{amsmath, amsfonts, amssymb, amsthm, bm, graphics, bbm, color}
\usepackage{graphicx}
\usepackage{subfigure}
\usepackage{mathtools}
\usepackage[table]{xcolor}
\usepackage{url}
\usepackage{mathabx}
\usepackage{cancel}
\usepackage{multicol}
\usepackage{float}
\usepackage{caption}
\usepackage{mathalfa}
\usepackage{hyperref}
\usepackage{afterpage}
\usepackage{multirow}
\usepackage[section]{placeins}

\DeclareMathAlphabet\mathbfcal{OMS}{cmsy}{b}{n}
\newtheorem{theorem}{Theorem}[section]
\newtheorem{lemma}[theorem]{Lemma}

\theoremstyle{definition}

\newtheorem{remark}[theorem]{Remark}

\newcolumntype{P}[1]{>{\centering\arraybackslash}p{#1}}

\makeatletter
\newcommand\makebig[2]{%
  \@xp\newcommand\@xp*\csname#1\endcsname{\bBigg@{#2}}%
  \@xp\newcommand\@xp*\csname#1l\endcsname{\@xp\mathopen\csname#1\endcsname}%
  \@xp\newcommand\@xp*\csname#1r\endcsname{\@xp\mathclose\csname#1\endcsname}%
}
\makeatother

\makebig{biggg} {3.0}
\makebig{Biggg} {3.5}
\makebig{bigggg}{4.0}
\makebig{Bigggg}{4.5}
\makebig{biggggg}{5.0}
\makebig{Biggggg}{5.5}
\makebig{bigggggg}{6.0}
\makebig{Bigggggg}{6.5}

\renewcommand{\vec}[1]{\mbox{\boldmath$#1$}}

\newcommand{\dif}{\mathrm{d}}
\newcommand{\im}{\mathrm{i}}

\newcommand{\bbR}{\mathbb{R}}
\newcommand{\tr}{\mathrm{tr}}

\begin{document}
\title{{How far are two symmetric matrices from commuting?  With an application to object characterisation and identification in metal detection}}
\author{P.D. Ledger$^*$, W.R.B. Lionheart$^\dagger$ and J. Elgy$^\ddagger$,\\
$^*$School of Computing \& Mathematical Sciences, University of Leicester,\\
University Road, Leicester LE1 7RH, U.K.\\
$^\dagger$Department of Mathematics, The University of Manchester,\\
Oxford Road, Manchester, M13 9PL, U.K. \\
$^\ddagger$Department of Electronic, Electrical and Systems Engineering, \\
University of Birmingham, Birmingham, B15 2TT, U.K \\
Corresponding author: pdl11@leicester.ac.uk}
\maketitle

\section*{Abstract}

Examining the extent to which measurements of rotation matrices are close to each other is challenging due measurement noise. To overcome this, data is typically smoothed and  Riemannian and Euclidean metrics  {are} applied. However, if  rotation matrices are not directly measured and are instead formed by eigenvectors of measured symmetric matrices, this can be problematic if the associated eigenvalues are close.
In this work, we propose novel {semi-metrics that can be used to approximate the Riemannian metric for small angles. } Our new results  do not require eigenvector information and are beneficial for measured datasets.   
There are also issues when using comparing rotational  data  arising from computational simulations and  it is important that 
the impact of the approximations on the computed outputs is properly assessed to ensure that the approximations made and the finite precision arithmetic are not unduly polluting the results.
In this work, we examine data arising from object  characterisation in metal detection using the complex symmetric rank two magnetic polarizability tensor (MPT) description, we rigorously analyse the effects of our numerical approximations and apply our new {approximate measures of distance to the commutator of the real and imaginary parts of the MPT} to this application. Our new {approximate measures of distance provide}  additional feature information, which is invariant of the object orientation,  to aid with object identification using machine learning classifiers. We present  Bayesian classification examples to demonstrate the success of our approach.

\noindent {\bf Keywords:} Riemannian metric, semi-metric, special orthogonal group, geodesics, matrix commutator, eddy current; magnetic polarizability tensor; metal detection; Bayesian classification, machine learning.

\noindent {\bf MSC Classification:} 	65N21, 65N30, 35R30, 35B30.

\section{Introduction} 

As described by Moakher~\cite{moakher}, understanding the extent to which two  sets of rotational data, when expressed as rotation matrices, are close to each other has important applications, which include plate tectonics~\cite{fe19678b-d1c8-3aa5-a891-204719f4cecf} and sequence dependent modelling of DNA~\cite{10.1063_1.472373}~\footnote{While three-dimensional rotational data is most common, the approaches also extend to higher dimensions.}. 
In these applications, the experimental data is typically noisy and a common approach is to remove or to reduce the noise by the construction of an appropriate filter. Our interest lies in a different application, but shares with these  the need to assess the extent to which different sets of rotational data are close to each other in an appropriate metric or {measure of distance}.

In metal detection, eddy currents are induced in highly conducting objects by low-frequency time varying magnetic fields and inductive measurements are made. Metal detection has applications including security screening, locating antipersonnel landmines,  identifying foreign objects for food safety, scrap sorting, archaeology and treasure hunting. In common with electromagnetic tomography at radar frequencies~\cite{ammarivolkov2005}, impedance measurements at low frequencies for low-contrast conducting objects~\cite{ammarikangbook}, electrosensing~\cite{ammarishapeelectro} electro and magnetostatics~\cite{LedgerLionheart2016}  and non-electromagnetic applications, including  elasto-dynamics field perturbations~\cite{ammarikangbook}, hydrodynamics and in dilute composites~\cite{gwmilton,ammarikangbook}, the magnetic field perturbation in the eddy current case can be described by an asymptotic expansion as the object size tends to zero.  Here, the leading order term object provides object characterisation information through a polarizability tensor description~\cite{Ammari2014} whose coefficients are a function of the exciting frequency, the object's shape and its materials~\cite{LedgerLionheart2020spect}.  The leading order in the expansion can be simplified so that it can be expressed in terms of a  rank two complex symmetric magnetic polarizability tensor (MPT)~\cite{LedgerLionheart2015}, which reduces the number of coefficients required for the description. Furthermore, the result can be generalised to a complete {asymptotic expansion} and generalised MPTs~\cite{LedgerLionheart2023a,LedgerLionheart2018g}, where the  coefficients of the higher rank tensors are smaller, but which provide further object characterisation information.
Nevertheless, the MPT
offers an economical characterisation of conducting magnetic objects, which can be obtained from measurements of the induced voltage.  In addition, the MPT characterisation can be computed for known objects that one might encounter in a detection scenario  and we
have developed algorithms~\cite{ben2020,Elgy2024_preprint} for efficiently and accurately computing the MPT coefficients using $hp$-finite element analysis~\cite{Demkowicz,Demkowiczvol2} and discretisations based on unstructured meshes with prismatic boundary layers and high order ${\vec H} (\text{curl})$ conforming elements~\cite{SchoberlZaglmayr2005}. Furthermore, in~\cite{Elgy2024_preprint} we have applied (adaptive) proper orthogonal decomposition reduced order models~\cite{hesthaven2016} to efficiently compute the MPT spectral signature (MPT coefficients as a function of frequency) to provide additional characterisation information. This approach has been successfully  applied to complex inhomogeneous objects and excellent agreement with measurements have been obtained~\cite{Elgy2022_preprintB}.
Machine learning classification approaches~\cite{bishopbook} have been applied~\cite{ledgerwilsonlion2022} to identify the class of an object from the MPT spectral signature by considering a dataset of a large number of objects and eigenvalues (or tensor invariants) of the real and imaginary parts of the complex symmetric rank two MPT, sampled at discrete frequencies, as features~\cite{ledgerwilsonamadlion2021}. Our interest in this work focuses on examining whether rotational information from the MPT spectral signature (and other polarizability tensor descriptions) could  add further feature information to assist with the classification.
 
In common with the applications outlined in the opening paragraph, the measured MPT spectral signature data is noisy and smoothing is typically applied to reduce the noise. However, even noiseless simulated MPT spectral signature data may have issues. To obtain the simulated data,  approximations are made   including adding numerical regularisation to circumvent a divergence constraint~\cite{ledgerzaglmayr2010} and the application of a discrete approximation using a finite element discretisation~\cite{Elgy2024_preprint}. In this work, we additionally assess the impact of these approximations on the computed MPT tensor coefficients so as to differentiate between effects due to effects due to our approximations and interesting features about the rotation data.
 
  The Riemanian and the Euclidean metrics  for comparing rotations  rely on availability of the rotational data itself.
However, for situations where the rotational data is not measured, but instead needs to be obtained from measured matrix data, then performing its eigenvalue decomposition would appear the obvious choice to obtain the orthogonal matrix describing the rotations. However, while it is well known that the numerical computation of the eigenvalues of a symmetric matrix using the QR algorithm is a stable process, calculating the eigenvectors to provide rotation data becomes unstable when the eigenvalues of the matrix are close together (see~\cite{vanloon}[Thm 8.12.12, pg 399-400]). Therefore, measuring the extent to which two sets of rotational data are close to each other using traditional metrics and eigenvector information, becomes problematic. In this work, we propose alternative {semi-metrics, which can be used to approximate the Riemannian metric for small angles,} that do not require explicit computation of eigenvector information and apply it to the MPT spectral signature.

 The material proceeds as follows: {In Section~\ref{sect:define} we begin by introducing some notation and providing some definitions. This is followed by Section~\ref{sect:metrics} where we recall traditional metrics for measuring the extent to which two sets of rotational data are close to each other and present alternative distance measures that can be computed without using the eigenvector information.}  We show that the {metrics, semi-metrics and associated approximate distance measures  we propose} are invariant under rotation of the dataset and consider their sensitivity in finite precision arithmetic. {Next, Section~\ref{sect:basicexamples} provides a series of illustrative examples to demonstrate the advantages of our new semi-metrics and associated approximate distance measures.} Section~\ref{sect:mpt} recalls the MPT and presents an analysis of the impact of numerical regularisation and numerical discretisation on the computed MPT coefficients. Then, Section~\ref{sect:results} presents a series of numerical results, which compare traditional metrics to our  {approximate distance measures}  for  the extent to which two sets of rotational data obtained from MPT coefficients are close to each other. We also include an example of object classification, which includes the additional feature information provided by our new distance measures.

\section{Notation and definitions} \label{sect:define}
We list and recall below some definitions and standard properties that we will use throughout this work.
\begin{itemize}
	\item $\bbR^{n\times n}$, the set of $n$ by $n$ real matrices, dimension $n^2$.
	\item $\bbR^{n\times n}_s$ the set of symmetric  $n$ by $n$ real matrices, dimension $n(n+1)/2$.
			\item $\bbR^{n\times n}_a$, the set of skew-symmetric  $n$ by $n$ real matrices, dimension $n(n-1)/2$.
			\item $\lambda_i(A)= (\Lambda_A)_{ii}$, the eigenvalues of $A\in \bbR^{n \times n}$. 
			\item $Q_A$ {an} orthogonal matrix of (right) eigenvectors of $A\in \bbR^{n \times n}$. If $A$ is diagonalisable (e.g. $A \in \bbR_s^{n \times n}$) then $A=Q_A^T \Lambda_A Q_A$, where {the superscript} $T$ denotes the transpose. 
	\item $\rho(A) = \max ( |\lambda_1|, |\lambda_2|, \ldots,|\lambda_n|)$, the spectral radius of $A\in \bbR^{n\times n}$.
	\item $\|A\|= (\tr (A A^T))^{1/2}$, the Frobenius norm of $A\in \bbR^{n \times n}$, which is just the Euclidean norm on $\bbR^{n^2}$ and $\| A \|^2 = \displaystyle \sum_{i=1}^n\sum_{j=1}^n (A)_{ij}^2$.
	\item $\langle A,B\rangle = \tr (AB^T)$ the Frobenius inner product of $A,B \in \bbR^{n \times n}$.
	\item $[A,B] = AB-BA$, the commutator of $A,B \in \bbR^{n \times n}_s$.
	\item $SO(n)$, the set of $n$ by $n$ orthogonal matrices with determinant one. This is a submanifold of $\bbR^{n\times n}$, the component of that  defined by the equations $Q Q^T=I$ containing $I$, with $I \in \bbR^{n \times n}_s$  being the identity matrix and $Q\in SO(n)$ a rotation matrix. This manifold has the same dimension as  $\bbR^{n\times n}_a$ . 
	\item If a linear operator $T:V\to W$ is given by a matrix, say $T(v) = Av$, then its operator norm is $\|T\|_{op} = \sqrt{\max_i \lambda_i (A^TA)}$.
	\item The deviatoric part of $A\in \bbR^{n\times n}_s$ is  $ \text{dev}\,(A) = A- \frac{1}{3} \text{tr}(A)I$.
\end{itemize}       

In addition, we remind readers of the following properties
\begin{itemize}
	\item The cyclic property of trace $\tr(ABC)=\tr(BCA)$ etc for matrices $A,B,C \in \bbR^{n \times n}$.
	\item When $A\in \bbR^{n \times n}$ is diagonalisable (e.g. $A\in \bbR^{n \times n}_s$), $ \| A\| ^2 = \sum\limits_{i=1}^n \lambda_i^2$. 
	\item For $U \in SO(n),\, A\in \bbR^{n\times n} $,    $ \|UA\|=\|AU\|=\|A\|$ (this follows from left and right multiplication by an orthogonal matrix $U$ can be considered as $U\otimes I$ and $I \otimes  U$ in $SO(n^2)$.
\end{itemize}

{
\section{Metrics and semi-metrics for the distance to the commutator of two matrices}\label{sect:metrics}
}
If two symmetric matrices $A,B \in \bbR^{n \times n}_s$ commute  this is equivalent to saying that $[A,B]=0$. It is obvious that {any two} diagonal matrices commute, and that the zero matrix and identity matrix commute with everything. 
{Matrices where the eigenspace of one are all the eigenspaces of the other also commute,}
{but how far are two symmetric matrices from commuting?}
 One possibility is to consider the diagonalisation 
  of $A= Q_A ^T\Lambda_A Q_A$ and $B= Q_B^T \Lambda_B Q_B$, where $Q_A,Q_B\in SO(n)$, {which, over permutations~\footnote{{Of course the columns of $Q_A$ and $Q_B$ are only defined upto a sign, and while orderings of $\Lambda_A$ (eg $\lambda_1(A) > \lambda_2 (A) >\cdots> \lambda_n(A)$) and $\Lambda_B$ fixes the columns of $Q_A$ and $Q_B$, we consider only those signs that result in $\text{det}, Q_A = \text{det} \, Q_B=1 $. Of these possible permutations,  we chose the one leading to the smallest distance measure.}}},  and application  of a  metric on $SO(n)$, leads to a minimum distance $d(Q_A,Q_B)$.
  {We recall that a metric $d$ has the following properties:
 \begin{itemize}
 \item The distance between from a point to itself is zero $d(Q_A,Q_A)=0$.
 \item (Positivity) The distance between two points is always positive, if $Q_A\ne Q_B$ then $d(Q_A,Q_B) > 0$.
 \item(Symmetry) The distance from $Q_A$ to $Q_B$ is the same as from $Q_B$ to $Q_A$, $d(Q_A,Q_B)=d(Q_B,Q_A)$.
 \item The triangle inequality holds $d(Q_A,Q_C) \le d(Q_A,Q_B)+d(Q_B,Q_C)$.
 \end{itemize}
 }
 Some common metrics on $SO(n)$ are briefly recalled in Section~\ref{sect:commonmetrics}.  {Then, in Section~\ref{sect:prodmetrics} product metrics are briefly discussed. This is followed by Section~\ref{sect:femmes}, where we consider semi-metrics and   approximate  measures of distance that do not rely on knowledge of $Q_A$ and $Q_B$.}

\subsection{Common metrics on $SO(n)$} \label{sect:commonmetrics}

 The geodesic distance on $S O(n)$  measures the distance around the curved surface of $SO(n)$. 
 {Let $R=Q_AQ_B^T$ then  Riemannian metric between two rotations $Q_A$ and $Q_B$ is~\cite{moakher}
\begin{equation}
d_R(Q_A,Q_B) := \frac{1}{\sqrt{2}} \| \log R \| = \frac{1}{\sqrt{2}} \| \log (Q_A Q_B^T) \| \nonumber .
\end{equation}
This metric measures the shortest geodesic curve that connects $Q_A$ and $Q_B$~\cite{moakher}. This distance is a {geodesic}, which lies entirely in $SO(3)$. Additionally, {
 $k\in {\mathbb R}^n$ is a unit vector and, in the case of $n=3$, has coefficients that coincide with the alternative notation  ${\vec k}\in {\mathbb R}^3 $  for a vector in an orthonormal coordinate system with basis vectors ${\vec e}_i$, $i=1,\ldots,3$ used in Section~\ref{sect:mpt},} $\theta$ is the angle through which $R$ rotates any non-zero vector not parallel to ${k}$ and
$K \in \bbR^{n\times n}_a$ is the skew-symmetric matrix corresponding to the vector product $K{x}= {k} \times {x}$ for $x \in {\mathbb R}^n$, then
\begin{equation}
d_R(Q_A,Q_B)
 = |\theta|  \frac{1}{\sqrt{2}} \| K \| = |\theta| , \nonumber
\end{equation}
which follows from the Rodrigues angle formula
\begin{align}
R= I + \sin \theta K  +(1-\cos \theta) K^2 = \exp (\theta K) \label{eqn:rodrigues} .
\end{align}
Note that ${k}$, hence $K$, and additionally $\theta$  {are only defined up to a sign.}
It is also easy to see that $K^2$ is the  symmetric matrix $-I+{k}{k}^T$.
 As~\cite{moakher} remarks, the metric $d_R(Q_A,Q_B)$   may not be unique. If $Q_A Q_B^T$ is an involution,  so that $(Q_AQ_B^T)^2 = I$ is a rotation through an angle $\pi$, then $Q_A$ and $Q_B$ can be connected by two curves of equal length. The Riemannian metric is also  bi-invariant in $SO(3)$~\cite{moakher}.
}

 Another candidate that is easier to compute is the ``chordal distance''  {or ``Euclidean distance''} \\$d_F(Q_A,Q_B):=\| Q_A - Q_B\|=\| Q_A Q_B^T - I \|$, which has the advantage that it is a norm associated with an inner product, but measures the straight line distance in $\bbR^{n\times n}$ as a Euclidean space.  To  compare,
 \begin{eqnarray*}
 d_F(Q_A,Q_B)^2  &=& \tr ((Q_A - Q_B)(Q_A - Q_B)^T)  \\
 &=&  \tr ( 2 I - R - R^T ) \\ 
 &=& 6 - 2\tr(R) \\
 &=&  6 -2 (2\cos  \theta +1)\\
 &=&  4(1 - \cos \theta) ,
\end{eqnarray*}
or $d_F(Q_A,Q_B) = 2 \sqrt{1 - \cos \theta}= 2\sqrt{2} \sin (|\theta|/2) $. Of course this is approximately $ \sqrt{2} |\theta|$  {for small $\theta$ so that $d_R(Q_A,Q_B) \approx d_F(Q_A,Q_B) /  \sqrt{2}$. Note that, apart from the end points, {the line joining $Q_A$ and $Q_B$, which is of length $d_F(Q_A,Q_B)$, does not lie in $SO(3)$}, but like $d_R(Q_A,Q_B)$ is bi-invariant in $SO(3)$~\cite{moakher}.  For more on the very interesting and useful topic of measuring distances between rotations see~\cite{moakher} and citations therein.}

{
\subsection{Product metrics} \label{sect:prodmetrics}
In the above we have considered metrics on $SO(n)$, }
{ now according to our sense of permutations above, for matrices  $A\in{\mathbb R}_s^{n\times n}$  and  $B\in{\mathbb R}_s^{n\times n}$,
 with distinct eigenvalues $\lambda_1(A) > \lambda_2 (A) >\cdots> \lambda_n(A)$ and  $\lambda_1(B) > \lambda_2 (B) >\cdots> \lambda_n(B)$,  we choose their unique representations to be the pairs $(\Lambda_A,Q_A)\in{\mathbb R}^{n\times n} \times SO(n)$ and $(\Lambda_B,Q_Q)\in{\mathbb R}^{n\times n} \times SO(n)$ where the signs of the columns of $Q_A$ and $Q_B$ are chosen such that $\text{det}\,Q_A=\text{det}\, Q_B=1$ and that they lead to a minimal distance $d(Q_A,Q_B)$. Then, 
 given any metric $d_1$ on ${\mathbb R}^{n\times n}$ and $d_2$ on $SO(n)$, we have the product metric
\begin{align}
d_{prod}(A,B) := \sqrt{d_1(\Lambda_A,\Lambda_B)^2+ d_2 (Q_A,Q_B)^2},
\end{align}
 for measuring the distance between $A$ and $B$. }

{While not essential, it will help some readers with a geometric frame of mind to describe some geometric considerations in order to  understand the connection between $d_{prod}(A,B)$ and $ d_2 (Q_A,Q_B)$.  Considering $n=3$, we know that the action of $O(3)$ on $\mathbb{R}^{3\times 3}_s$ partitions the space into orbits  of  orthogonally equivalent matrices.  The orbits are in one to one correspondence with the space of unordered triples $\{ \lambda_1, \lambda_2, \lambda_3 \}$ of eigenvalues. 
Denoting $S_3 = \text{Sym}(1,2,3)$ as the group of  order 6 of bijections from $\{1,2,3\}$ to $\{1,2,3\}$, then,
 topologically, this is $\mathbb{R}/ S_3$, equivalence classes of triples of real numbers up to permutation with the quotient topology. This is an orbifold (roughly speaking a topological space that is locally a finite group quotient of Euclidean space), and fails to be a manifold on the sets where some of the $\lambda_i$ are equal. Away from these singular sets, for example on the subset $U \subset  {\mathbb R}^{3 \times 3}_s$ where $\lambda_1 > \lambda_2 > \lambda_3$, have a product structure and any $A\in U$ has a unique representation as $(Q_A, \Lambda_A)$. However, just as polar coordinates break down at the origin, this breaks down as eigenvalues become equal and $Q_A$ is not uniquely defined.  On a subset of $U$, whose closure is contained in the interior, the product metric $d_{prod}(A,B)$ is compatible with the standard topology on  $\mathbb{R}^{3\times 3}_s$. Indeed, we can write down the map $A \rightarrow (Q_A, \Lambda_A)$ (see Appendix~\ref{sect:appendixexplicteig}) and its inverse explicitly and verify that it is a diffiomorphism. Of course as one approaches the boundary where eigenvalues are equal the  derivative of the inverse map is unbounded.}

\subsection{Semi-metrics and approximate measures of distance} \label{sect:femmes}
In this section, we explore {semi-metrics and} alternatives to the metrics discussed in Section~\ref{sect:commonmetrics} that {provide an approximation of the extent to which $A,B \in \bbR_s^{n\times n}$  are non-commuting when the eigenvalues $\Lambda_A$ and $\Lambda_B$ are distinct, but may become close}. The {semi-metrics we will propose are related to the product metrics discussed in Section~\ref{sect:prodmetrics} and we will show that they can be used to approximate the Riemannian metrics $d_R(Q_A,Q_B)$ for small angles. Our semi-metrics} 
 will  use  $\Lambda_A$ and $\Lambda_B$, but will not
require knowledge of $Q_A$ or $Q_B$, which can be challenging to compute accurately when the eigenvalues become close (we go into further details in Section~\ref{sect:sensmetrics} and also provide additional references in this section). 
For example, we could consider {the semi-metric}
\begin{align}
d_C (A,B):= \|[A,B]\| \nonumber,
\end{align}
{
which does not satisfy a triangular inequality,} but there are  several problems with this. It is homogenous order one in $A$ and in $B$, which makes it difficult to compare the extent to which  matrices commute  in different norms. We can of course use
\begin{align}
\frac{\|[A,B]\|}{ \|A\| \cdot\|B\|}, \nonumber
\end{align}
or some other normalisation such as   $\|AB\|$. However, the result still depends on the difference between the eigenvalues, even though we have normalised by the sum of the squares of the eigenvalues as $\|A\|^2=\|\Lambda_A^2\|$. 
Our aim  is to obtain a measure of non-commutingness that gets around this problem. We will come back to looking at $d_C$.

{
\subsubsection{A semi-metric and approximate measure of distance inspired by the Hoffman-Wielandt Inequality}\label{connect:measEwithmetR} 
}
The Hoffman-Wielandt inequality applied to $A,B\in \bbR_s^{n \times n}$ takes the form~\cite{Hoffman-Wielandt-inequality}
\begin{align}
\min_{\sigma \in S_n} \sum_{i=1}^n | \lambda_i(A) - \lambda_{\sigma(i)}  (B) |^2 \le \| A- B\|^2,
\end{align}
where $S_n$ is the permutation group of $\{1,\ldots,n\}$. Using similar reasoning, an upper bound can be established~\cite{Hoffman-Wielandt-inequality} so that 
\begin{align}
\min_{\sigma \in S_n} \sum_{i=1}^n | \lambda_i(A) - \lambda_{\sigma(i)} (B) |^2 \le \| A- B\|^2 \le \max_{\sigma \in S_n} \sum_{i=1}^n | \lambda_i(A) - \lambda_{\sigma(i)} (B) |^2.
\end{align}
If the distinct eigenvalues are ordered as {$\lambda_1(A) > \lambda_2(A) \ldots > \lambda_n(A)$ and  $\lambda_1(B)>  \lambda_2(B) \ldots > \lambda_n(B)$} then,
\begin{align}
\min_{\sigma \in S_n} \sum_{i=1}^n | \lambda_i(A) - \lambda_{\sigma(i)} (B) |^2 = \sum_{i=1}^n | \lambda_i(A) - \lambda_{i}(B) |^2 \le \| A- B\|^2,
\label{eqn:Hoffman-Wielandt-inequality}
\end{align}
and, in particular, choosing $B=A+E$, with $E$ representing a small symmetric perturbation, leads to the well known result on the sensitivity of numerically computed eigenvalues to rounding (or other) errors (Thm 8.1.4~\cite{vanloon}[pp395]).

Since
\begin{align}
\| A- B \| & = \| Q_A^T \Lambda_A Q_A - Q_B^T \Lambda_B Q_B\|  = \| Q_A^T  ( \Lambda_A Q_A - Q_A Q_B^T \Lambda_B Q_B) \| \nonumber \\
& = \| ( \Lambda_A Q_A -  R  \Lambda_B Q_B)Q_A^T  \|  \nonumber\\
&= \|  \Lambda_A  - R \Lambda_B R^T \| ,
\end{align}
 using (\ref{eqn:rodrigues}) and the small angle approximations {$\sin \theta = \theta+O(\theta^3) $ and $1-\cos \theta = \theta^2/2 +O(\theta^4) $ as $\theta \to 0$} it follows that
\begin{align}
\| A- B \|^2 =  \|  \Lambda_A  -  \Lambda_B  \| ^2 +O(\theta^2),
\end{align}
 as $\theta \to 0$.
However, our interest lies in finding  approximations of $d_R(Q_A,Q_B)$, which is the length of the shortest geodesic curve that connects $Q_A$ and $Q_B$. This suggests that  we  consider {the semi-metric $ \displaystyle \min_{\pm} | d_E^{\pm} (A,B)|$, which does not satisfy a triangular inequality, where $\displaystyle  \min_{\pm}$ indicates the minimum of $| d_E^{+} (A,B)|$ and $| d_E^{-} (A,B)|$ with }
 \begin{align}
 d_E^{\pm} (A,B):= \begin{array}{c} \max_{\sigma \in S_n} \\ \min_{\sigma \in S_n} \end{array} \left ( \pm \|A-B\|^2 \mp  \sum_{i=1}^n | \lambda_i(A) - \lambda_{\sigma(i)} (B) |^2  \right ),
 \end{align} 
 rather than  the absolute difference between $ \|A-B\|^2$ and  $\sum_{i=1}^n | \lambda_i(A) - \lambda_{\sigma(i)} (B) |^2  $ maximised or minimised over the permutations $\sigma \in S_n$.
Furthermore, in order to establish a {connection  with} the metrics discussed in Section~\ref{sect:commonmetrics}, we prove an explicit connection between 
{$\displaystyle \min_{\pm} | d_E^{\pm} (A,B)|$}
and $|\theta|=d_R(Q_A,Q_B) $ for small angles in the following Lemma.

{{
\begin{lemma}\label{lemma:dRconnectdE}
Fixing $\Lambda_A$ and $\Lambda_B$ to have distinct eigenvalues, then there is a rotation through $\theta$ that depends on the eigenvectors $Q_A$ and $Q_B$ such that
\begin{align}
 d_R(Q_A,Q_B)^2 =  \theta^2 =\min_\pm \left |  \frac{d_E^{\pm }(A,B)}{ \tr(K^2 \Lambda_B \Lambda_A) - \tr (K \Lambda_A K \Lambda_B)  } \right |+O(\theta^3)  ,\label{eqn:approxdr}
\end{align}
as $\theta\to 0 $.
\end{lemma}}
\begin{proof}
Writing $A= Q_A^T \Lambda_A Q_A$ and $B=Q_B^T \Lambda_B Q_B$, then
\begin{align}
\| A-B \|^2 & = \| \Lambda_A\|^2 + \| \Lambda_B\|^2 - \tr(AB) - \tr(BA), \nonumber \\
\|\Lambda_B-\Lambda_A\|^2 & = \| \Lambda_A\|^2 + \|\Lambda_B\|^2 - 2 \tr ( \Lambda_A \Lambda_B).
\end{align}
Noting that $ \tr(AB) = \tr(BA) = \tr(R^T \Lambda_A R \Lambda_B) $, where $R=  Q_A Q_B^T$,  applying the Rodrigues angle formula (\ref{eqn:rodrigues}), and $K^T=-K$  it follows that
\begin{align}
\tr(AB) = \tr(BA) =  &\tr \left ( \Lambda_A \Lambda_B + \left(  \Lambda_A K \Lambda_B + K^T \Lambda_A \Lambda_B \right) \sin \theta
+ 2 K^2 \Lambda_A \Lambda_B ( 1- \cos \theta) + \right . \nonumber \\
& \left . 
K^T \Lambda_A K \Lambda_B \sin^2 \theta + \left(  K^T \Lambda_A K^2 \Lambda_B + K^2 \Lambda_A K \Lambda_B \right )\sin \theta {(1-\cos \theta)} \right. \nonumber \\
&\left . +
K^2 \Lambda_A K^2 \Lambda_B (1-\cos \theta)^2
\right ) \nonumber \\
=  &\tr \left ( \Lambda_A \Lambda_B + 2 K^2 \Lambda_A \Lambda_B ( 1- \cos \theta) +
K^T \Lambda_A K \Lambda_B \sin^2 \theta +  \right . \nonumber \\
& \left . 
\left(  K^2 \Lambda_A K \Lambda_B - K^2 \Lambda_B K \Lambda_A \right )\sin \theta {(1-\cos \theta)}
+
K^2 \Lambda_A K^2 \Lambda_B (1-\cos \theta)^2
\right ) .
\end{align}
Next, by considering the small angle approximations $\sin \theta = \theta+O(\theta^3)$, $1-\cos \theta  = \theta^2/2+O(\theta^4)$ as $\theta \to 0$, we have
\begin{align}
d_E^{\pm}(A,B)  =& \mp\theta^2 \left ( \tr(K^2 \Lambda_B \Lambda_A) - \tr (K \Lambda_A K \Lambda_B)  \right )\nonumber  \\
& \mp \frac{\theta^3}{2} \left (  \tr( K^2 \Lambda_A  K \Lambda_B ) - \tr( K^2 \Lambda_B K \Lambda_A) \right ) \mp \frac{\theta^4}{4}\tr ( K^2   \Lambda_A K^2 \Lambda_B) +O(\theta^4)
\nonumber \\
=& \mp\theta^2 \left ( \tr(K^2 \Lambda_B \Lambda_A) - \tr (K \Lambda_A K \Lambda_B)  \right ) +O(\theta^3),
 \label{eqn:emeasureexpand1} 
\end{align}
as $\theta \to 0$. {Since we wish to make comparisons with $d_R(Q_A,Q_B)$, which measures the shortest geodesic curve that connects $Q_A$ and $Q_B$, we consider $\min_{\pm}$ which immediately leads to the desired result.}
\end{proof}}

{
The practical approximation of $|\theta|$ using $d_E^{\pm }(A,B)$ is aided by the following result.
\begin{lemma} \label{lemma:dRconnectdEapprox}
The normalising constant in Lemma~\ref{lemma:dRconnectdE} has the form  $\tr(K^2 \Lambda_B \Lambda_A) - \tr (K \Lambda_A K \Lambda_B)  =k^TM(A,B)k$ where
 $M(A,B)$ is diagonal, is non-positive definite and for $n=3$ and distinct eigenvalues has coefficients
\begin{subequations} \label{eqn:Menteries}
\begin{align}
(M(A,B))_{11} = &- (\lambda_{2}(A) - \lambda_{3}(A)) (\lambda_{2}(B) - \lambda_{3}(B)) < 0,  \\
(M(A,B))_{22} = & - (\lambda_{1}(A) - \lambda_{3}(A)) (\lambda_{1}(B) - \lambda_{3}(B)) < 0, \\
(M(A,B))_{33} = & - (\lambda_{1}(A) - \lambda_{2}(A)) (\lambda_{1}(B) - \lambda_{2}(B))  < 0,
\end{align}
\end{subequations}
and the estimate
\begin{align}
\min_i (M(A,B))_{ii}  \le &k^T M(A,B) k \le \max_i (M(A,B))_{ii} \label{eqn:bdconstems} ,
\end{align}
holds where the bounds do not require knowledge of $K$.
\end{lemma}
 \begin{proof}
Without loss of generality, we consider $n=3$ and note that
\begin{align}
\tr(K^2 \Lambda_B \Lambda_A) - \tr (K \Lambda_A K \Lambda_B) = &  k_1^2 \left(  - \lambda_{2}(A) \lambda_{2}(B) - \lambda_3 (A) \lambda_3(B) + \lambda_3(A) \lambda_2(B) + \lambda_2 (A) \lambda_3(B) \right ) \nonumber \\
& + k_2^2 \left(  - \lambda_{1}(A) \lambda_{1}(B) - \lambda_3 (A) \lambda_3(B) + \lambda_3(A) \lambda_1(B) + \lambda_1 (A) \lambda_3(B) \right ) \nonumber \\
& + k_3^2 \left(  - \lambda_{1}(A) \lambda_{1}(B) - \lambda_2 (A) \lambda_2(B) + \lambda_2(A) \lambda_1(B) + \lambda_1 (A) \lambda_2(B) \right ) \nonumber\\
=& k^T M(A,B) k,
\end{align}
where $M(A,B)$ is diagonal with coefficients given in (\ref{eqn:Menteries}) and is non-positive definite, which establishes  the first part of the result. Given the properties of $M$ we have
\begin{align}
\min (M(A,B))_{ii} {k}^Tk  \le & k^T M(A,B) k \le \max (M(A,B))_{ii}  k^Tk, \nonumber \end{align}
and since
$|k|^2 = k^T k =1$,  we immediately obtain (\ref{eqn:bdconstems}), which completes the proof.
\end{proof}}
\begin{remark}
Furthermore, since
\begin{align}
\Lambda( \text{dev}\, (A)) =\Lambda(A -(1/3) \tr(A) I ) =\text{diag}\, ( \lambda_2(A) -\lambda_3(A),
\lambda_1(A)-\lambda_3(A),
\lambda_1(A)-\lambda_2(A)),
\end{align}
then $M = - \Lambda( \text{dev}\, (A))  \Lambda( \text{dev}\, (B))$ and this also means that
\begin{align} 
-  \| A-(1/3)\tr(A) I \|_{op}   \| B-(1/3) \tr(B)I  \|_{op}   \le & k^T M k \le 0 , \nonumber 
\end{align}
since $ \rho (\text{dev}\, (A)) \le  \| A-(1/3) \tr(A) I\|_{op}$,
which we expect to be less sharp than (\ref{eqn:bdconstems}).
\end{remark}

{
{
Using our unique representations $A\in {\mathbb R}_s^{n \times n}$ as  $(\Lambda_A,Q_A)\in{\mathbb R}^{n\times n} \times SO(n)$ and $B\in {\mathbb R}_s^{n \times n}$  as  $(\Lambda_B,Q_Q)\in{\mathbb R}^{n\times n} \times SO(n)$ 
introduced above,} the results of Lemmas~\ref{lemma:dRconnectdE} and~\ref{lemma:dRconnectdEapprox}  can be used to estimate a minimal $|\theta| =d_R(Q_A,Q_B)$ from $d_E^{\pm}(A,B)$, which we denote by $d_{E,\theta}(A,B)$. {Due to the restriction of distinct eigenvalues  $k^T M k< 0$, but if  the algebraic multiplicity of the eigenvalues of $A$ or $B$ is greater than one, then  $k^T M k \le 0$.}

{
\subsubsection{A semi-metric and approximate measure of distance  inspired by the commutator $[A,B]$}
Proceeding in a similar manner as Section~\ref{connect:measEwithmetR}, we seek to identify a connection between the semi-metric $d_C(A,B)$} and $d_R(Q_A,Q_B)$. First, we establish that
$d_C(A,B)$ is related to the angle  $\phi :=\angle(AB,BA) = \cos^{-1} \langle AB, BA\rangle/(\|AB\|\cdot \|BA\|)$
by expanding
\begin{align*}
d_C(A,B)^2 = \| [A,B]\|^2 &= \tr((AB-BA)(AB-BA)^T)\\
&=-\tr (ABAB +BA BA) + \tr ( AB BA + BA AB) \\
&= - \tr( AB+BA)^2 + \tr( AB BA + BA AB +AB BA +BA AB) ,
\end{align*}
and noting $\|AB\|^2 =\tr(BAAB)=\tr(ABBA)=\|BA\|^2$ and  $ \tr(ABAB)= \langle AB, BA\rangle$,  so that
\begin{align*}
d_C(A,B)^2 =& 2 (\| AB \|^2 + \| BA \|^2) - \tr (AB +BA)^2\\
 & = 2 (\| AB \|^2 + \| BA \|^2)  - ( \| AB \|^2 + \| BA \|^2 + \langle AB, BA \rangle + \langle  BA,AB \rangle ) \\ 
& = 2 \| AB \|^2 - 2 \langle AB, BA \rangle,
 \end{align*}
 then, for small $\phi$,
\begin{align}
\frac{d_C(A,B)} {\| AB \|} = \frac{\| [A,B]\|}{\| AB \|}  =  \sqrt{2 ( 1- \cos \phi)} \approx |\phi |.
\end{align}
}

However,  a relationship between  $|\theta| = d_R(Q_A,Q_B)$ and $d_C(A,B)$ is preferred, which is provided by the following result.
\begin{lemma}\label{lemma:dRconnectdC}
{Fixing $\Lambda_A$ and $\Lambda_B$ with distinct eigenvalues, there is a rotation through $\theta$ that depends on the eigenvectors such that
\begin{align}
 d_R(Q_A,Q_B)^2 = \frac{d_C (A,B)^2 }{
 C (K, \Lambda_A, \Lambda_B)
   }+ O(\theta^3) \label{eqn:approxdrdc},
\end{align}
as $\theta \to 0$ where
\begin{align}
 C (K, \Lambda_A, \Lambda_B) := &2   \tr \left(  \Lambda_A^2 K^2 \Lambda_B^2 - \Lambda_A^2 K \Lambda_B^2 K \right )  - \nonumber \\
 & 4 \tr \left ( K \Lambda_A \Lambda_B K \Lambda_A \Lambda_B +  K^2 \Lambda_A^2 \Lambda_B^2 - K\Lambda_A K \Lambda_A \Lambda_B^2 - K \Lambda_B K \Lambda_B \Lambda_A^2 \right ) \nonumber.
\end{align}}
\end{lemma}
\begin{proof}
Using $A=Q_A^T \Lambda_A Q_A$ and $B =Q_B^T \Lambda_B Q_B$ it follows that
\begin{align*}
\| AB \|^2 = & \tr ( \Lambda_A^2 R \Lambda_B^2 R^T) \nonumber \\
=& \tr \left (
\Lambda_A^2 \Lambda_B^2 + (\Lambda_A^2 K \Lambda_B^2 K^T) \sin^2 \theta + (\Lambda_A^2 K^2 \Lambda_B^2 + \Lambda_A^2 \Lambda_B^2 K^2 ) (1-\cos \theta) \right . \nonumber \\
 &\left . +( \Lambda_A^2 K \Lambda_B^2 K^2 + \Lambda_A^2 K^2 \Lambda_B^2 K^T) \sin \theta (1-\cos \theta)  + \Lambda_A^2 K^2 \Lambda_B^2 K^2 (1-\cos \theta)^2
\right ),
 \end{align*}
 which, for small angles, becomes
\begin{align*}
 \| AB \|^2 = & \tr \left ( \Lambda_A^2 \Lambda_B^2 + ( \Lambda_A^2 K^2 \Lambda_B^2 - \Lambda_A^2 K \Lambda_B^2 K ) \theta^2 
 + ( \Lambda_A^2 K \Lambda_B^2 K^2 - \Lambda_A^2 K^2 \Lambda_B^2 K) \frac{\theta^3}{2} \right . \nonumber \\
 & \left . + \Lambda_A^2 K^2 \Lambda_B^2 K^2 \frac{\theta^4}{4} 
 \right)+O(\theta^4),\nonumber \\
= & \tr \left ( \Lambda_A^2 \Lambda_B^2 + ( \Lambda_A^2 K^2 \Lambda_B^2 - \Lambda_A^2 K \Lambda_B^2 K ) \theta^2 \right ) +O(\theta^3),
 \end{align*}
 as $\theta \to 0$. In a similar way,
 \begin{align*}
 \langle AB, BA \rangle & = \tr ( R^T \Lambda_A R \Lambda_B R^T \Lambda_A R \Lambda_B)\\
&  = \tr \left (
 \Lambda_A^2 \Lambda_B^2 +2 (K \Lambda_A \Lambda_B K \Lambda_A \Lambda_B- K\Lambda_A K \Lambda_A \Lambda_B^2 - K \Lambda_B K \Lambda_B \Lambda_A^2) \sin^2 \theta \right .
 \nonumber \\
 & \left . + 4K^2 \Lambda_A^2 \Lambda_B^2 (1-\cos \theta) + \ldots
 \right ) \\
 & =  \tr ( \Lambda_A^2 \Lambda_B^2  + 2 ( K \Lambda_A \Lambda_B K \Lambda_A \Lambda_B- K\Lambda_A K \Lambda_A \Lambda_B^2 - K \Lambda_B K \Lambda_B \Lambda_A^2
 + K^2 \Lambda_A^2 \Lambda_B^2) \theta^2  + O(\theta^3),
 \end{align*}
as $\theta \to 0$.  Thus,
 \begin{align}
d_C(A,B)^2 =  \| [A,B]\|^2 = & 2 \| AB \|^2 - 2 \langle AB, BA \rangle  \nonumber \\
  = & 2   \tr \left(  \Lambda_A^2 K^2 \Lambda_B^2 - \Lambda_A^2 K \Lambda_B^2 K \right ) \theta^2 \nonumber \\
& - 4 \tr \left ( K \Lambda_A \Lambda_B K \Lambda_A \Lambda_B +  K^2 \Lambda_A^2 \Lambda_B^2 - K\Lambda_A K \Lambda_A \Lambda_B^2 - K \Lambda_B K \Lambda_B \Lambda_A^2 \right )\theta^2 +O(\theta^3) \nonumber\\
=&  C (K, \Lambda_A, \Lambda_B) \theta^2 +O(\theta^3),
 \end{align} 
 as $\theta \to 0$, which immediately leads to the desired result. 
 \end{proof}
 
{
The practical approximation of $|\theta|$ using $d_C(A,B)$ is aided by the following result.
 \begin{lemma}
 The normalising constant  $C(K,\Lambda_A,\Lambda_B)$  in Lemma~\ref{lemma:dRconnectdC} has the form $C(K,\Lambda_A,\Lambda_B) = k^T N(A,B) k$ where
 $N(A,B)$ is diagonal, is non-negative definite and for $n=3$ and distinct eigenvalues has coefficients
\begin{subequations} \label{eqn:Nconstest}
\begin{align}
(N(A,B))_{11} = & 2 (\lambda_2(A)-\lambda_3(A))^2 ( \lambda_2(B) - \lambda_3(B))^2 > 0, \\
  (N(A,B))_{22} = & 2 (\lambda_1(A)-\lambda_3(A))^2 ( \lambda_1(B) - \lambda_3(B))^2 > 0, \\
 (N(A,B))_{33} = & 2 (\lambda_1(A)-\lambda_2(A))^2 ( \lambda_1(B) - \lambda_2(B))^2 > 0,
\end{align}
\end{subequations}
and the estimate
\begin{align}
\min (N(A,B))_{ii}  \le &   C (K, \Lambda_A, \Lambda_B) \le \max (N(A,B))_{ii} ,\label{eqn:bdconstemsdc} 
\end{align}
holds, where the bounds do not require knowledge of $K$.
\end{lemma}
 \begin{proof}
 Without loss of generality, an explicit computation using $n=3$ establishes that
 \begin{align}
 C (K, \Lambda_A, \Lambda_B)=k^T N(A,B) k , \nonumber
\end{align}
where $N(A,B)$ is diagonal and non-negative definite with coefficents stated in (\ref{eqn:Nconstest}).  Given the properties of $N$ we have
\begin{align}
\min (N(A,B))_{ii} k^Tk \le &   C (K, \Lambda_A, \Lambda_B) \le \max (N(A,B))_{ii} k^Tk , 
\end{align}
and since $k^Tk=1$, we immediately obtain (\ref{eqn:bdconstemsdc}), which completes the proof.
 \end{proof}
}
{
Using our unique representations $A\in {\mathbb R}_s^{n \times n}$ as  $(\Lambda_A,Q_A)\in{\mathbb R}^{n\times n} \times SO(n)$ and $B\in {\mathbb R}_s^{n \times n}$  as  $(\Lambda_B,Q_Q)\in{\mathbb R}^{n\times n} \times SO(n)$ 
introduced above,} we denote the approximation of $d_R(Q_A,Q_B)$ obtained from $d_C (A,B)$ as  $d_{C,\theta}(A,B)$. 

{
\subsection{Rotational invariance of the metrics, semi-metrics and approximate distance measures }}
We wish to explore whether the metrics and {semi-metrics} proposed above are invariant to rotation of the dataset. This is crucial 
since, in Section~\ref{sect:mpt}, we will explore their application to a complex symmetric rank two tensor characterisation of a conducting magnetic object where the hidden object's orientation is not known a-priori. In this application,  the real and imaginary tensor coefficients will be arranged as real symmetric matrices, $A, B \in \bbR^{3 \times 3}_s$, say. Given that the characterisation of the objects are tensorial, the coefficients of the characterisation will transform in a natural way under rotation of the coordinate system by an orthogonal rotation matrix $Q_{Rot}$. This means that $A$ and $B$ transform to $A' = Q_{Rot} A Q_{Rot}^T$ and $B' = Q_{Rot} B Q_{Rot}^T$, respectively, in the transformed system and, similarly, $Q_{A'}=Q_{Rot} Q_A$ and $Q_{B'}=Q_{Rot} Q_B$ for matrices of eigenvectors with $\lambda_i(A) = \lambda_i (A')$ and $\lambda_i(B)=\lambda_i(B')$ being unchanged under the transformation. We need the metrics and measures to be similarly invariant under object rotation for classification purposes as the hidden object's orientation is unknown.

First, considering $d_R(Q_{A'},Q_{B'})$, then
\begin{align}
d_R(Q_{A'},Q_{B'}) =  \frac{1}{2} \| \log (Q_{A'} Q_{B'} )\| = \frac{1}{2} \| \log ( Q_{Rot} Q_A Q_B^T Q_{Rot}^T )\| \nonumber ,
\end{align}
and properties of the matrix exponential and matrix logarithm tell us that $\exp( S M S^{-1}) = S e^M S^{-1}$ so that $S M S^{-1} = \log ( S e^M S^{-1}) $. Setting,  $e^M= Q_A Q_B^T $ and $S=Q_{Rot}$ so that $M = \log ( Q_A Q_B^T)$ and  $S^{-1} = Q_{Rot}^{-1}$ then
\begin{align}
d_R(Q_{A'},Q_{B'}) & =  \frac{1}{2} \|   Q_{Rot}  \log( Q_A Q_B^T) Q_{Rot}^T \|  \nonumber \\
&=\frac{1}{2} \text{tr} \,( (  Q_{Rot}  \log( Q_A Q_B^T) Q_{Rot}^T) (  Q_{Rot}  \log( Q_A Q_B^T) Q_{Rot}^T)^T )^{1/2} \nonumber \\
& =\frac{1}{2} \text{tr} \,(    \log( Q_A Q_B^T)  ( \log( Q_A Q_B^T)^T ))^{1/2 } \nonumber \\
& = \frac{1}{2} \|    \log( Q_A Q_B^T)  \| = d_R(Q_A,Q_B),
\end{align}
by properties of Frobenius norm and the trace operator. Hence, the metric $d_R(Q_A,Q_B)$ is invariant under rotation of the coordinate system. 

Second, considering  $d_F(Q_{A'},Q_{B'})$, then in a similar way
\begin{align}
d_F(Q_{A'},Q_{B'}) & =  \|  Q_{A'} - Q_{B'} \|^2  \nonumber \\
& = \text{tr} \, ( Q_{Rot} (Q_A-Q_B) (Q_{Rot} (Q_A-Q_B) )^T) \nonumber\\
& = \text{tr} \, ( (Q_A-Q_B) (Q_A-Q_B) ^T) = d_F(Q_A,Q_B),
\end{align}
and so $d_F(Q_A,Q_B)$ is also  invariant under rotation of the coordinate system. 

Third, for $d_E^\pm(A',B')$, we need only consider $\| A' - B'\|^2 $ since we know that $\lambda_i(A') = \lambda_i(A)$ and  $\lambda_i(B') = \lambda_i(B)$. We have that
 \begin{align}
\| A' - B'\|^2  & =  \|  Q_{Rot} (A-B) Q_{Rot}^T \|^2  \nonumber \\
& = \text{tr} \, ( Q_{Rot} (A-B) (Q_{Rot} (A-B) )^T) \nonumber\\
& = \text{tr} \, ( (A-B) (A-B) ^T) = \| A - B\|^2,
\end{align}
and so $d_E^\pm(A,B)$ is also  invariant under rotation of the coordinate system. No further difficulties are encountered  when computing $ d_{E,\theta}(A,B)$ as the normalisation only involves eigenvalues.

Finally, for $d_C(A',B')$, we consider 
\begin{align}
d_C(A',B')= \| [A',B'] \|^2 &= \| A' B'  -B' A' \|^2 \nonumber\\
& = \| Q_{Rot} (AB-BA) Q_{Rot}^T \|^2 \nonumber \\
& = \text{tr}(  Q_{Rot} (AB-BA) Q_{Rot}^T ) ( (  Q_{Rot} (AB-BA) Q_{Rot}^T )^T ) \nonumber \\
& = \| [A,B] \|^2 =d_C(A,B),
\end{align}
and so $d_C(A,B)$ is also  invariant under rotation of the coordinate system. Once again, no further difficulties are encountered  when computing $ d_{C,\theta}(A,B) $ as the normalisation only involves eigenvalues.

As we have seen $R= Q_A Q_B^T = \exp (\theta K)$ and we have already established that $d_R(Q_A,Q_B) = |\theta| = d_R(Q_{A'},Q_{B'})= |\theta'|$ and so it useful to establish the relationship between $K$ and $K'$. Firstly, $R' = Q_{A'}Q_{B'}^T = Q_{Rot} R Q_{Rot}^T$ and so
\begin{align}
\theta' K'= \log R' = & \log ( Q_{Rot} R Q_{Rot}^T)  \nonumber \\
	=&  \log ( Q_{Rot} R Q_{Rot}^T)   \nonumber \\
	=&   Q_{Rot} \log (R) Q_{Rot}^T = \theta Q_{Rot} K Q_{Rot}^T	 .
\end{align}
Since we know that $|\theta| = |\theta'|$ then $|K'| = | Q_{Rot} K Q_{Rot}^T|$ and, hence, $|K|$ transforms like a rank two tensor. Since $K$ has non-zero components related to the unit vector ${ k}$, then $\lambda_i(K) = \{0, \pm \im \}$ and, hence, its eigenvalues do not provide additional information.
{
\subsection{Sensitivity of metrics, semi-metrics and approximate distance measures in finite precision arithmetic}}
\label{sect:sensmetrics}
Given the  finite precision of numerical linear algebra,
we consider the behaviour of the metrics $d_R(Q_A,Q_B)$, $d_F(Q_A,Q_B)$, {the semi-metrics $|d_E^\pm (A,B)|$ and  $ d_C(A,B)$
and the associated approximate measures of distance}  $ d_{E,\theta} (A,B)$ and  $ d_{C,\theta} (A,B)$, under small perturbations in $A$ and $B$ for matrices with coefficients $|(A)_{ij}|<1$ and $|(B)_{ij}|<1$. In particular, we  consider their behaviour  when distance between the eigenvalues of $A$ and the distance between the eigenvalues of  $B$  are small.
To aid with this, we first recall results from Wilkinson~\cite{wilkinson1965}[pg. 66-71]. A small perturbation is applied to  $A\in {\mathbb R}_s^{n \times n}$ resulting in $A_\epsilon= A +\epsilon S_A$, with sufficiently small $\epsilon>0$ and $\| S_A\| =1$, using a perturbation analysis, Wilkinson  obtains
\begin{align}
| \lambda_i (A) - \lambda_i (A_\epsilon)  | \le \epsilon C,
\end{align}
with $C$ independent of $\epsilon$ and that the difference between the corresponding eigenvectors $q_i(A)$  and $q_i(A_\epsilon)$ can be estimated as
\begin{align}
|\Delta q_i(A) |:= | q_i(A) -q_i(A_\epsilon) |  \le & C\epsilon \left | \sum_{j=1, j \ne i}^n \frac{ q_j(A) ^TS q_i(A)}{ (\lambda_i(A) -\lambda_j(A)) q_j (A)^T q_j(A)} q_j(A)  \right | ,\nonumber \\
\le & \epsilon 
C \sum_{j=1, j \ne i}^n \frac{| q_j(A) ^TS_A q_i(A)|}{ |\lambda_i(A) -\lambda_j(A)|} |q_j(A)| . \label{eqn:erroreigvect}
\end{align}
Since the computation of eigenvalues and eigenvectors is obtained  using finite precision arithmetic, these results show that the computed eigenvalues are relatively well behaved, but the accuracy of the computed eigenvectors  deteriorates like $1/\min |\lambda_i(A)-\lambda_j(A)|$  if the eigenvalues are close together.

Indeed, this motivates  the use of the approximate  $ d_{E,\theta}(A,B)$ and  $ d_{C,\theta}(A,B)$, rather than $d_R(Q_A,Q_B)$, $d_F(Q_A,Q_B)$,  as the former can be computed without knowledge of the eigenvectors. As further motivation, we  estimate $|\cos \theta -\cos \theta_\epsilon|$ in terms of the metrics $d_R(Q_A,Q_B)$ and $d_F(Q_A,Q_B)$, which again illustrates that their accuracy deteriorates as the eigenvalues become close.

\begin{lemma} \label{lemma:cosdrdfpert}
Considering the perturbed symmetric matrices $A_\epsilon= A +\epsilon S_A$ and $B_\epsilon= B +\epsilon S_B$  with sufficiently small $\epsilon>0$ and $\| S_A\| =\| S_B\|=1$, then
\begin{subequations}
\begin{align}
|\cos \theta_\epsilon - \cos \theta|=&
|\cos d_R(Q_{A_\epsilon},Q_{B_\epsilon}) - \cos d_R(Q_A,Q_B)| =\nonumber \\
& \le C\epsilon \left ( \sum_{i,j=1, i \ne j}^n \frac{| q_j (A)^T S_A q_i(A)|}{|\lambda_i(A) - \lambda_j(A)|} |q_j(A)| | q_i(B)| + 
\frac{| q_j (B)^T S_B q_i(B)|}{|\lambda_i(B) - \lambda_j(B)|} |q_j(B) | |  q_i(A)|  
\right ) \label{eqn:estimatedr}, \\
|\cos \theta_\epsilon - \cos \theta| = &
\frac{1}{4} |d_F(Q_{A_\epsilon},Q_{B_\epsilon})^2 -  d_F(Q_A,Q_B)^2| \nonumber \\
& \le C\epsilon \left ( \sum_{i,j=1, i \ne j}^n \frac{| q_j (A)^T S_A q_i(A)|}{|\lambda_i(A) - \lambda_j(A)|} |q_j(A)| | q_i(B)| + 
\frac{| q_j (B)^T S_B q_i(B)|}{|\lambda_i(B) - \lambda_j(B)|} |q_j(B)| | q_i(A)|  
\right ) \label{eqn:estimatedF}.
 \end{align}
 \end{subequations}
\end{lemma}
\begin{proof}
To prove (\ref{eqn:estimatedr}) then we introduce the matrix $R_\epsilon$ with enteries
\begin{align}
(R_\epsilon)_{ij} :=  (q_i(A) + \Delta q_i (A))( q_j(B) + \Delta q_j(B)) ,
\end{align}
so that 
\begin{align}
\tr(R_\epsilon) = & \sum_{i=1}^3  (q_i(A) + \Delta q_i (A))^T( q_i(B) + \Delta q_i(B)) \nonumber \\
& = \tr(R) + \tr(\Delta R) + O(\epsilon^2) \label{eqn:expandRepsilon},
\end{align}
where
\begin{align}
\tr(R) &  = \sum_{i=1}^3  q_i(A) ^T q_i(B) , \nonumber \\
\tr(\Delta R) &  = \sum_{i=1}^3  \Delta q_i(A) ^T q_i(B)  + q_i(A) ^T \Delta q_i(B)+  \Delta q_i(A) ^T \Delta q_i(B) . \nonumber 
\end{align}
From (\ref{eqn:rodrigues}) we know that $\cos \theta = (\tr(R) -1)/2$.  
Since $|\cos d_R(Q_{A_\epsilon},Q_{B_\epsilon}) - \cos d_R(Q_A,Q_B)| =|\cos \theta_\epsilon - \cos \theta|= | \tr(R_\epsilon) - \tr(R) |/2 $,   by considering (\ref{eqn:expandRepsilon}) and  applying (\ref{eqn:erroreigvect}) to  $A_\epsilon= A +\epsilon S_A$ and $B_\epsilon= B +\epsilon S_B$  
gives the desired result.

To prove (\ref{eqn:estimatedF}) we recall that $d_F(Q_A,Q_B)^2 = 6 - \tr(R)= 4(1 - \cos \theta)$ and so $d_F(Q_{A_\epsilon},Q_{B_\epsilon})^2 = 6 - \tr(R_\epsilon)= 4(1 - \cos \theta_\epsilon)$. Thus,  $4|\cos \theta_\epsilon - \cos \theta| =  |d_F(Q_{A_\epsilon},Q_{B_\epsilon})^2 -  d_F(Q_A,Q_B)^2|$. Then, considering (\ref{eqn:expandRepsilon}) and  applying (\ref{eqn:erroreigvect}) to  $A_\epsilon= A +\epsilon S_A$ and $B_\epsilon= B +\epsilon S_B$ 
gives the desired result.
\end{proof}

Next, we consider the {approximate measures of distance}   $d_{E,\theta}(A,B)$ and  $d_{C,\theta}(A,B)$ under perturbed matrices. The following result follows immediately from the definitions of $d_E^\pm(A,B)$ and $d_C(A,B)$, which are well behaved for close eigenvalues.
\begin{lemma} \label{lemma:cosdedcfpert}
Considering the symmetric perturbed matrices $A_\epsilon= A +\epsilon S_A$ and $B_\epsilon= B +\epsilon S_B$  with sufficiently small $\epsilon>0$, $\| S_A\| =\| S_B\|=1$, then, for small $\epsilon$,
\begin{align}
|d_E^{\pm}(A,B) - d_E^{\pm}(A_\epsilon,B_\epsilon) | \le & C \epsilon, \nonumber\\
|d_C (A,B) - d_C (A_\epsilon,B_\epsilon) | \le & C \epsilon \nonumber.
\end{align}
\end{lemma}
Insights in to the behaviour of $ d_{E,\theta}(A_\epsilon,B_\epsilon)$ and  $ d_{C,\theta}(A_\epsilon,B_\epsilon)$ can be obtained by considering the behaviour of $M(A_\epsilon,B_\epsilon)$ and $N(A_\epsilon,B_\epsilon)$ in the normalising constants $k^TM(A_\epsilon,B_\epsilon)k$ and $k^TN(A_\epsilon,B_\epsilon)k$. Specifically,
   \begin{align*}
(M(A_\epsilon,B_\epsilon))_{11} = &
- (\lambda_{2}(A) - \lambda_{3}(A)+O(\epsilon)) (\lambda_{2}(B) - \lambda_{3}(B) +O(\epsilon)) \nonumber\\
= & - (\lambda_{2}(A) - \lambda_{3}(A)) (\lambda_{2}(B) - \lambda_{3}(B) ) +O(\epsilon),  \\
(M(A_\epsilon,B_\epsilon))_{22} = & - (\lambda_{1}(A) - \lambda_{3}(A)) (\lambda_{1}(B) - \lambda_{3}(B) )+O(\epsilon) , \\
(M(A_\epsilon,B_\epsilon))_{33} = & - (\lambda_{1}(A) - \lambda_{2}(A) ) (\lambda_{1}(B) - \lambda_{2}(B))+O(\epsilon) ,
\end{align*}
{as $\epsilon \to 0$ and similarly}
\begin{align*}
 (N(A_\epsilon,B_\epsilon))_{11} = & 2 (\lambda_2(A)-\lambda_3(A) )^2 ( \lambda_2(B) - \lambda_3(B))^2+O(\epsilon),\\
 (N(A_\epsilon,B_\epsilon))_{22} = & 2 (\lambda_1(A)-\lambda_3(A))^2 ( \lambda_1(B) - \lambda_3(B))^2 +O(\epsilon), \\
(N(A_\epsilon,B_\epsilon))_{33} = & 2 (\lambda_1(A)-\lambda_2(A))^2 ( \lambda_1(B) - \lambda_2(B))^2 +O(\epsilon),
\end{align*}
{as $\epsilon \to 0$.}
{This means that as the eigenvalues become close and say $(\lambda_2(A)- \lambda_3(A) )(  \lambda_2(B)- \lambda_3(B)) \to 0 $, the diagonal coefficients of $M$ and $N$ are always non-zero and depend on $\epsilon$, which   provides a form of regularisation when using the approximate measures.}

\begin{remark}
In a similar way to the behaviour of $|\Delta q_i(A) |:= | q_i(A) -q_i(A_\epsilon) |$, estimated in (\ref{eqn:erroreigvect}), Lemma~\ref{lemma:cosdrdfpert} indicates that the cosines of the metrics (and hence the metrics) $d_R(A_\epsilon,B_\epsilon)$, $d_F(A_\epsilon,B_\epsilon)$ are subject to potentially large errors if the eigenvalues of $A$ or $B$ become close. On the other hand, the behaviour of $ d_{E,\theta} (A_\epsilon,B_\epsilon)$ and  $d_{C,\theta}(A_\epsilon ,B_\epsilon )$ will be regularised when the eigenvalues become close.
\end{remark}

{Lemmas~\ref{lemma:cosdrdfpert}  and~\ref{lemma:cosdedcfpert}, as well as above discussion, will be directly applicable in the following sections as we will obtain the eigenvalue decomposition of  $A$ and $B$ numerically and then compute $ d_R(A_\epsilon,B_\epsilon)$, $d_F (A_\epsilon,B_\epsilon)$ and  $d_{E,\theta}(A_\epsilon,B_\epsilon)$ and  $ d_{C,\theta} (A_\epsilon ,B_\epsilon )$.  We will then also consider the situation where the coefficients of $A$ and $B$ arise from measurements (and are subject to unavoidable measurement errors) or from numerical simulations (and are subject to round-off error and discretisation errors). }

{
\section{Illustrative examples for symmetric matrices} \label{sect:basicexamples}
}
\begin{figure}
\begin{center}
$\begin{array}{cc}
\includegraphics[width=0.45\textwidth]{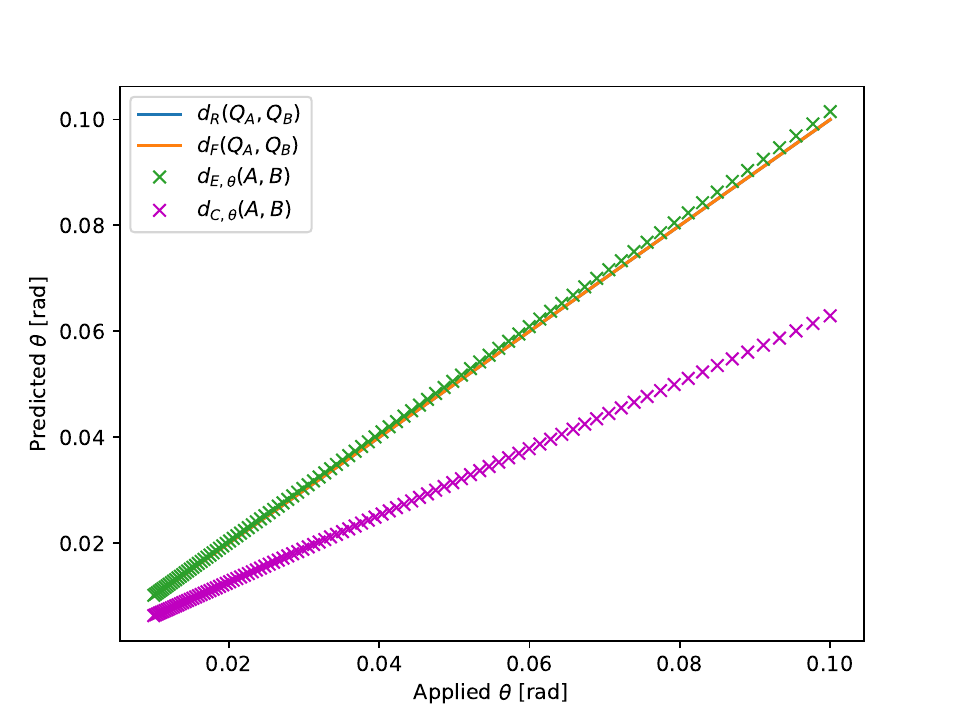} &
\includegraphics[width=0.45\textwidth]{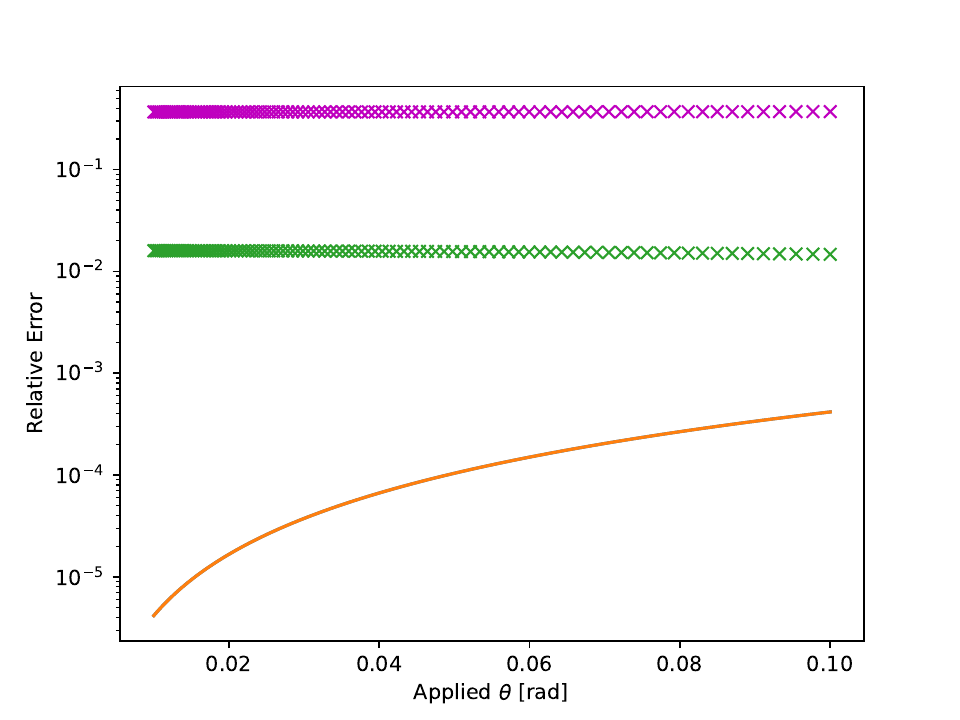} \\
(a) & (b) \\
\end{array}$
\end{center}
\caption{Two symmetric matrices $A$ and $B$ with distinct eigenvalues showing $(a)$ applied and predicted angles and $(b)$ relative error of angle prediction.}
\label{fig:twosymmatrixangles}
\end{figure}

{In this section, we provide examples to illustare the need for alternatives to the classical $d_R(Q_A,Q_B)$ and $d_F(Q_A,Q_B)/\sqrt{2}$ metrics for symmetric matrices $A$ and $B$. In the first example, we illustrate a situation where the existing metrics work well. The matrix $A$ is chosen as 
\begin{align}
A= \left ( \begin{array}{ccc}
3 & 1 & 0.5\\
1 & 4 & 0.25\\
0.5 & 0.25 & 5
\end{array} \right ),
\nonumber
\end{align}
which has eigenvalues and eigenvectors
\begin{align}
\Lambda_A= \left( \begin{array}{ccc}
2.35 & 0 & 0\\
0 &  4.31 & 0\\
0 & 0 &  5.34\end{array} \right ), \qquad
 Q_A= \left ( \begin{array}{rrr}
 0.86&  0.37 &  0.36\\
  -0.50&  0.76&  0.42\\
   -0.11& -0.54&  0.83
 \end{array} \right ),
 \nonumber 
\end{align} 
to 2dp when obtained numerically~\footnote{Unless otherwise mentioned, we use \texttt{numpy.linalg.eigh}  for obtaining the eigenvalue decomposition, which uses the LAPACK routines  \texttt{\_syevd} and \texttt{\_heevd}} and
 $A$ can be observed to have well separated eigenvalues. A rotational axis of $k =(1/\sqrt{3}) (1,1,1)$ is specified,  a  small angle $\theta$ is considered  
and a rotation matrix $R= \exp ( \theta K)$ is generated. Choosing $\Lambda_B = \Lambda_A$ and computing $Q_B =Q_A R^T$ allows $B=Q_B \Lambda_B Q_B^T$ to be found. This process is repeated for $0 < \theta \le 0.1$ rad and the angles predicted by $d_R(Q_A,Q_B)$, $d_F(Q_A,Q_B)/\sqrt{2} $, $d_{E,\theta} (A,B)$ and $d_{C,\theta}(A,B)$ are compared to those prescribed. Note that when computing $d_R(Q_A,Q_B)$ and $d_F(Q_A,Q_B)/\sqrt{2} $ these rely on $Q_A$ and $Q_B$ and, from the definitions of the metrics, we are interested in the {\em smallest distance measure}. 
This means that, as described in Section~\ref{sect:metrics},  among the possible candidates for choosing the ordering of $\Lambda_A$ and $\Lambda_B$, hence, the ordering of the columns $Q_A$ and $Q_B$, we consider the permutation leading to the smallest measure. Additionally, we account for possible sign flips of the directions of the eigenvectors, which still lead to $\text{det}\, Q_A = \text{det}\,Q_B =1$. 
When computing $d_{E,\theta} (A,B)$ and $d_{C,\theta}(A,B)$ we present the results corresponding the lower bounds (LB) provided by the approximation of the constants in Lemmas~\ref{lemma:dRconnectdE} and \ref{lemma:dRconnectdC}.
The results are shown in Figure~\ref{fig:twosymmatrixangles}. As the eigenvalues of $A$ and $B$ are well-separated, both $d_R(Q_A,Q_B)$ and $d_F(Q_A,Q_B)/\sqrt{2} $  provide very accurate predictions of the specified angle and, on the chosen scale, the curves are in-distinguishable. The accuracy of $d_{E,\theta}(A,B)$ is much better than  $d_{C,\theta}(A,B)$ and is of the relative error of the former is approximately $1.5\%$ independent of the angles considered.
}

{Next, to illustrate a situation where $d_R(Q_A,Q_B)$ and $d_F(Q_A,Q_B)$ is problematic, we consider  a matrix $A$ of the form
\begin{align}
A= \left ( \begin{array}{ccc}
1+\epsilon & 0.01 & 0.01\\
0.01 & 1-\epsilon & 0.01\\
0.01 & 0.01 & 2
\end{array} \right ),
\nonumber
\end{align}
which has two close, but distinct, eigenvalues for small $\epsilon$~\cite{vanloon}[pg. 400].  As noted in Section~\ref{sect:sensmetrics}, traditional computational approaches (eg as used by \texttt{numpy.linalg.eigh}) for the numerical computation of $Q_A$, and in particular those columns corresponding to the two closely spaced eigenvalues, are problematic. We choose $B= \text{diag} (1,2,3)$ so that  $\Lambda_B$ contains well separated eigenvalues and, hence, there are no issues in computing the eigenvectors contained in  $Q_B$ using standard approaches. We then make predictions of $\theta = \|\log(Q_A Q_B)\| /\sqrt{2}$ using $d_R(Q_A,Q_B)$, $d_{E,\theta}(A,B)$ 
and $d_{C,\theta}(A,B)$ 
for  $0.001\le \epsilon\le 0.1$ in Figure~\ref{fig:twosymmatrixanglespert}. The results $d_F(Q_A,Q_B)/\sqrt{2}$ are similar to $d_R(Q_A,Q_B)$ and, hence, not shown.
The results highlight the issues in using the metric $d_R(Q_A,Q_B)$  for this situation when  $\epsilon $ is small since, instead of predicting $\theta \approx 0$ (since upto a permutation of columns and possible sign flips $Q_B=I$ and $Q_A\approx I$), the metric is polluted by the inaccurate rotational data in $Q_A$. However, the approximate $d_{E,\theta}(A,B)$ and $d_{C,\theta} (A,B)$ perform much better as they do not rely on eigenvector information and produces a regularised solution for small $\epsilon$. Furthermore, the figure also includes the results obtained using  an explicit solution for $\Lambda_A$~\cite{smitheigen} and a solution for $Q_A$, which is  obtained using cross products~\cite{kopp,3x3techreport} and is summarised in Appendix~\ref{sect:appendixexplicteig}. But, this approach requires conditional statements that depend on the proximity of the eigenvalues to each other  (eg $| \lambda_i - \lambda_j| / | \lambda_i | < \text{tol } $) and can still result in erroneous results if closely spaced eigenvalues are not detected and/or if the approach for obtaining eigenvectors for closely spaced eigenvalues is applied when the eigenvalues are well separated. The effect of using the cross product approach to compute $Q_A$ and,  hence, $d_R(Q_A,Q_B)$, for different proximity conditions are shown. Using the explicit solution for computing $Q_A$ together with $d_R(Q_A,Q_B)$ does recover the expected behaviour, but this comes with the challenge of choosing an appropriate tolerance, which is overcome by instead using $d_{E,\theta}(A,B)$ or $d_{C,\theta}(A,B)$.
}

\begin{figure}
\begin{center}
\includegraphics[width=0.45\textwidth]{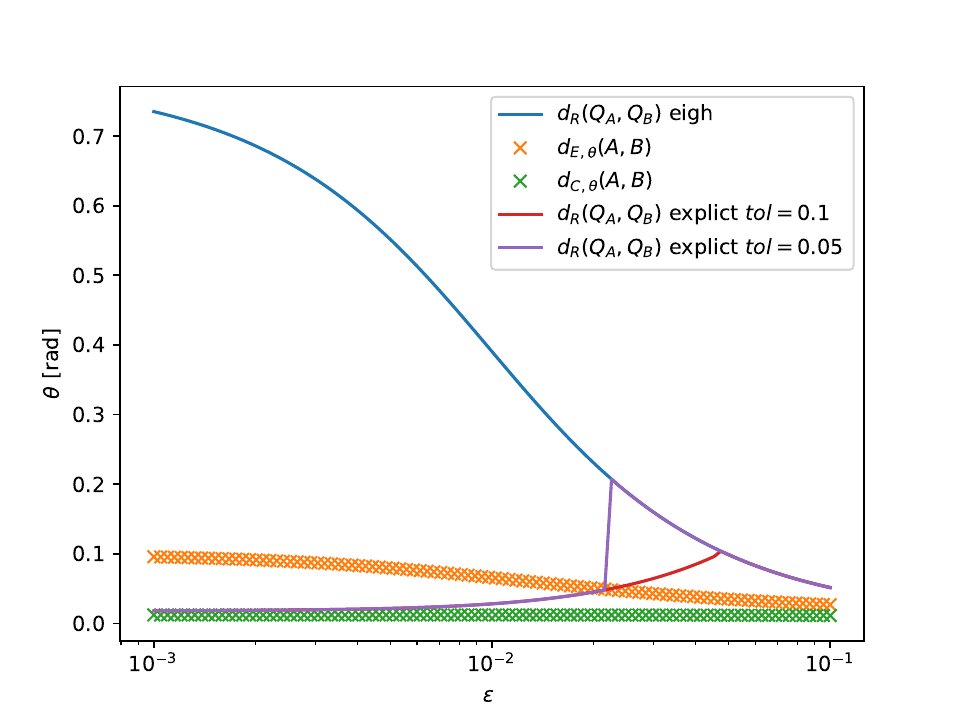} 
\end{center}
\caption{Two symmetric matrices $A$ and $B$ where $A$ has two closely spaced eigenvalues and $B$ has distinct eigenvalues showing the predicted angles as a function of $\epsilon$.}
\label{fig:twosymmatrixanglespert}
\end{figure}

\section{Magnetic Polarizability Tensor (MPT) and errors due to regularisation and discretisation}\label{sect:mpt}

{The purpose of this, and subsequent sections, is to explore how our new semi-metrics, and their associated approximate measures of distance, can aid with the identification and classification of metal objects in metal detection. We begin by briefly introducing 
the rank two complex symmetric magnetic polarizability tensor (MPT), which offers an economical characterisation of conducting magnetic objects. For a hidden object, its coefficients  can be obtained from measurements of the induced voltage. In addition, the MPT coefficients can also be calculated numerically  for a library of objects that one may encounter.  In particular, for an object $B_\alpha = \alpha B +{\vec z}$, where $\alpha \ll1 $ denotes the object's size, $B$ a unit sized object placed at the origin and ${\vec z}$ a translation, 
the complex symmetric MPT 
${\mathcal M} =({\mathcal M})_{ij} {\vec e}_i \otimes {\vec e}_j$, where Einstein index summation convention is implied,
has six  independent coefficients $({\mathcal M})_{ij} $ that depend on $\alpha$, $B$,  the object's magnetic permeability contrast $\mu_r$, its conductivity $\sigma_*$  and also depends on the angular frequency $\omega$ of the exciting magnetic field~\cite{LedgerLionheart2020spect}, but are independent of ${\vec z}$. 
The MPT coefficients, follow from  the additive decomposition~\cite{LedgerLionheart2020spect}
$(\mathcal{M})_{ij}:=(\tilde{\mathcal{R}})_{ij}+\im(\mathcal{I})_{ij}=(\mathcal{N}^0)_{ij}+(\mathcal{R})_{ij}+\im(  \mathcal{I})_{ij}$ where 
\begin{subequations}
\label{eqn:NRI}
\begin{align}
(\mathcal{N}^0( \alpha B,\mu_r) )_{ij}&:=\alpha^3\delta_{ij}\int_{B}(1-\tilde{\mu}_r^{-1})\dif \bm{\xi}+\frac{\alpha^3}{4}\int_{B\cup B^c}\tilde{\mu}_r^{-1}\nabla\times\tilde{\bm{\theta}}_i^{(0)}\cdot\nabla\times\tilde{\bm{\theta}}_j^{(0)}\dif \bm{\xi},\\
(\mathcal{R}(\alpha B, \omega,\sigma_*,\mu_r))_{ij}&:=-\frac{\alpha^3}{4}\int_{B\cup B^c}\tilde{\mu}_r^{-1}\nabla\times\overline{\bm{\theta}_i^{(1)}}\cdot\nabla\times{\bm{\theta}_j^{(1)}}\dif \bm{\xi}, \label{eqn:Rtensor}\\
(\mathcal{I}(\alpha B, \omega,\sigma_*,\mu_r))_{ij}&:=\frac{\alpha^3}{4}\int_B\nu\Big(\overline{\bm{\theta}_i^{(1)}+\tilde{\bm{\theta}}_i^{(0)}+\bm{e}_i\times\bm{\xi}}\Big)\cdot\Big({\bm{\theta}_j^{(1)}+\tilde{\bm{\theta}}_j^{(0)}+\bm{e}_j\times\bm{\xi}}\Big)\dif \bm{\xi}, \label{eqn:Itensor}
\end{align}
\end{subequations}
are each the coefficients of real symmetric rank two tensors and $\delta_{ij}$ is the Kronecker delta. In the above,  the overbar denotes the complex conjugate,  $\tilde{\mu}_r ({\vec \xi}) =\mu_r$ in $B$ and $\tilde{\mu}_r({\vec \xi}) =1$ otherwise, where ${\vec \xi}$ is measured from the origin, which lies inside $B$. We note that an object's MPT coefficients as a function of $\omega$ is called its MPT spectral signature.}

{ To compute (\ref{eqn:NRI}), we need to approximate the  solution of the transmission problem for ${\bm \theta}_i^{(1)}\in {\mathbb C}^3$ 
\begin{subequations}\label{eqn:theta1trans}
\begin{align}
\nabla \times \mu_r^{-1} \nabla \times {\vec \theta}_i^{(1)} -\im \nu {\vec \theta}_i^{(1)} &= \im \nu {\vec \theta}_i^{(0)} && \text{in $B$}, \\
\nabla \times\nabla \times {\vec \theta}_i^{(1)}  &=  {\vec 0} && \text{in $B^c:= {\mathbb R}^3 \setminus \overline{B}$} ,\\
\nabla \cdot {\vec \theta}_i^{(1)} & = 0 && \text{in $B^c$}, \\
[{\vec n} \times {\vec \theta}_i^{(1)}]_\Gamma ={\vec 0}, \ [{\vec n} \times \tilde{\mu}_r^{-1} \nabla \times {\vec \theta}_i^{(1)}]_\Gamma &  ={\vec 0} && \text{on $\Gamma:=\partial B$},  \\
{\vec \theta}_i^{(1)} & = O ( |{\vec \xi}|^{-1} ) && \text{as $| {\vec \xi} | \to \infty$},
\end{align}
\end{subequations}
where $\nu:= \alpha^2 \sigma_* \mu_0 \omega$ ,  $[\cdot ]_\Gamma$ denotes the jump over $\Gamma$, ${\vec n}$ is the unit outward normal.
Furthermore, we also need to obtain $ {\vec \theta}_i^{(0)}\in {\mathbb R}^3$, which is obtained by the solution of a related transmission problem with $\nu=0$ in $B$ and $\tilde{\vec \theta}_i^{(0)} := {\vec \theta}_i^{(0)} - {\vec e}_i \times {\vec \xi} \in {\mathbb R}^3$. }
{The numerical solution of (\ref{eqn:theta1trans}), which is discussed in Section~\ref{sect:numericalcomp}, becomes increasingly challenging as  $\omega$ becomes large (up to the limit of the eddy current model) due to large $\mu_r$ and high $\sigma_*$ and leads to a rapid decay of  eddy currents just inside the conductor. In the physical object $B_\alpha$ these are characterised by  the skin depth
\begin{align}
\delta := \sqrt{2/(\omega \sigma_* \mu_0 \mu_r)},
\end{align}
 which measures the depth to which surface eddy currents decay to $1/e$ of their surface value, and, in the context of the non-dimensional object $B$ used in the computations, 
\begin{align}
\tau :=\sqrt{2/(\mu_r\nu)} = \delta/\alpha , 
\end{align}
 which is a corresponding non-dimensional skin depth for the eddy currents ${\vec J}_i^e := \im \omega \sigma_* {\vec \theta}_i^{(1)}$~\cite{Elgy2024_preprint}.}

{Our previous studies have used invariants of the MPT spectral signature for object classification~\cite{ledgerwilsonamadlion2021,ledgerwilsonlion2022}. In this work, we are particularly interested in  the extent to which $\mathcal{N}^0, \mathcal{R}$ and $ \mathcal{I}$  commute as a function of $\omega$ (and hence $\nu$) and whether this can be provide useful additional information for object classification. We arrange the coefficients of the symmetric rank-2 tensors as real symmetric $3 \times 3 $ matrices and use $\mathcal{N}^0\in{\mathbb R}_s^{3 \times 3}, \mathcal{R}\in{\mathbb R}_s^{3 \times 3}$ and $ \mathcal{I}\in{\mathbb R}_s^{3 \times 3}$ to also used to denote the corresponding matrices.  We introduce the parameter-dependent commutators}
\begin{subequations} \label{eqn:randispectral}
\begin{align}
{\mathcal Z}(\nu) := & [ {\mathcal R} (\nu), {\mathcal I} (\nu) ], \nonumber \\
{\mathcal Z}^{(0)} (\nu) := & [ {\mathcal N} ^{(0)}, {\mathcal I} (\nu) ], \nonumber \\
\tilde{\mathcal Z}(\nu) := & [ \tilde{\mathcal R} (\nu), {\mathcal I} (\nu) ]  = 
 [ {\mathcal R} (\nu), {\mathcal I} (\nu)  ]+ [ {\mathcal N} ^{(0)}, {\mathcal I} (\nu) ]
= {\mathcal Z}( \nu )  + {\mathcal Z}^{(0)} (\nu) ,
\end{align}
\end{subequations}
which are skew symmetric and vanish if the 
 {eigenspace of one are all the eigenspaces of the other.
 We note that if the rotational and reflectional symmetries of an object are such that  $ {\mathcal N}^0$, $ {\mathcal R}$ and $ {\mathcal I}$ have at most three independent coefficients each then the commutators vanish independent of $\nu$ and so we will focus on objects that do not have this property.}

{ We recall below Lemma 8.11 from~\cite{LedgerLionheart2020spect} that provides bounds on the growth of $|{\mathcal Z}(\nu)|$, which provides some useful insights in to its behaviour.}
\begin{lemma}[Lemma 8.11 from~\cite{LedgerLionheart2020spect}] 
\label{lemma:oldcomresultupdate}
The off-diagonal elements of
$[  {\mathcal R} (\nu) , {\mathcal I}(\nu) ]$, 
$ [ {\mathcal R} (\nu_1) ,  {\mathcal R} (\nu_2) ]$ and $[  {\mathcal I} (\nu_1),  {\mathcal I} (\nu_2) ]$
 for $0 < \nu_1,\nu_2,\nu< \infty$ and $\nu_1 \ne \nu_2$ can be estimated as follows
\begin{subequations}
\begin{align}
\left |  ({\mathcal Z})_{ij} \right | = \left | ([  {\mathcal R} (\nu) , {\mathcal I}(\nu)])_{ij}\right | &
   \le C \alpha^6 \nu ,  \label{eqn:proofcom1}\\
\left |
( [ {\mathcal R} (\nu_1),  {\mathcal R} (\nu_2)])_{ij}
 \right | & \le C \alpha^6  ,  \\
\left | ( 
[  {\mathcal I} (\nu_1),  {\mathcal I} (\nu_2) ])_{ij}
 \right | & \le C \alpha^6  \nu_1 \nu_2,  
\end{align}
\end{subequations}
where $ i \ne j$ , $C>0$ is independent of $\nu, \nu_1, \nu_2$ and $\alpha$.
\end{lemma}
Note that an alternative proof is possible, which follows from (48a) and (48b) in~\cite{LedgerLionheart2020spect}.
This result can be extended to the cases of $({\mathcal Z}^{(0)} (\nu))_{ij}$ and
 $(\tilde{\mathcal Z})_{ij}$ in a similar way and the result is presented below:
\begin{lemma} \label{lemma:commutatoeztild}
The off-diagonal elements of the commutators of ${\mathcal N}^{(0)} $ and ${\mathcal I} (\nu) $  for $0 < \nu < \infty$ 
 can be estimated as
\begin{align}
\left |  ( {\mathcal Z}^{(0)})_{ij} \right | = \left | 
( [  \tilde{\mathcal N}^{(0)}  , {\mathcal I}(\nu)])_{ij}
 \right | & \le C \alpha^6 \nu ,  
 \end{align}
 and
 between $\tilde{\mathcal R} (\nu)$ and ${\mathcal I} (\nu) $
  as \begin{align}
\left |  (\tilde{\mathcal Z})_{ij} \right | = \left | 
( [  \tilde{\mathcal R} (\nu) , {\mathcal I}(\nu)])_{ij}
 \right | & \le C \alpha^6 \nu ,  
 \end{align}
  where $ i \ne j$ , $C>0$ is independent of $\nu$ and $\alpha$.
\end{lemma}
These results reveal that the commutators do not necessarily vanish and indicate their rate dependence with $\nu$, but not their magnitude. {Before we present an application of the previously presented metrics  and approximate measures of distance arising from our new semi-metrics},  it is crucial to understand the potential impact of the approximations made in the numerical computation and to assess whether any  predictions of non-commuting  behaviour in  the numerical predictions is due to the approximations made or is indeed a reflection of the behaviour predicted by Lemma~\ref{lemma:commutatoeztild}. 
{To do this, we first provide an outline of the procedure we use to compute the MPT coefficients to high accuracy. Then,  a--prori error estimates for the errors in the approximated tensor coefficients, and their effect on the commutators, are established:} First, for the errors associated with regularisation and, then, secondly, combining with the effects associated with the discrete approximation.  An appendix (Appendix~\ref{sect:appendix}) is included, which presents a constant free  a-posterori error indicator of the error associated with regularisation in practical computations.
As will be observed in our subsequent numerical examples, it will be very important to distinguish between the effects due to numerical regularisation and discretisation discussed in the following sections and the effects associated with the commutator, predicted by Lemmas~\ref{lemma:oldcomresultupdate} and~\ref{lemma:commutatoeztild}.

\subsection{Outline of numerical computation} \label{sect:numericalcomp}
In practice, to compute MPT coefficients, a finite domain $\Omega$ with convex truncation boundary located sufficiently far from the object $B$ is introduced and the condition ${\bm n} \times {\bm \theta}_i^{(1)} = {\bm 0}$ imposed as an approximation to the static far field decay of the field. This leads to
 the transmission problem for ${\bm \theta}_i^{(1)}\in {\mathbb C}^3$ stated in (\ref{eqn:theta1trans}) being replaced by 
\begin{subequations}\label{eqn:theta1transfull}
\begin{align}
\nabla \times {\mu}_r^{-1} \nabla \times {\bm \theta}_i^{(1)} - \im \nu {\bm \theta}_i^{(1)}  & =  \im \nu {\bm \theta}_i^{(0)} && \text{in $B$},\\
\nabla \times  \nabla \times {\bm \theta}_i^{(1)}   & =  {\bm 0}&& \text{in $\Omega \setminus \overline{B}$},\\
\nabla \cdot {\bm \theta}_i^{(1)} & = 0 && \text{in $\Omega \setminus \overline{B}$},\\
[ {\bm n} \times {\bm \theta}_i^{(1)}] & = {\bm 0} && \text{on $\Gamma$}, \\
[ {\bm n} \times \mu_r^{-1} \nabla \times {\bm \theta}_i^{(1)}] & = {\bm 0} && \text{on $\Gamma$}, \\
{\bm n} \times {\bm \theta}_i^{(1)}   & = {\bm 0}  && \text{on $\partial \Omega$}.
\end{align}
\end{subequations}
In order to compute the MPT spectral signature is obtained by sweeping through different values of $\omega$, (and, hence, $\nu$) solving this problem and then computing (\ref{eqn:NRI}) in each case. {The weak solution of (\ref{eqn:theta1transfull}) is: Find ${\bm \theta}_i^{(1) }\in X$ such that
\begin{align}
\int_\Omega \tilde{\mu}_r^{-1} \nabla \times {\bm \theta}_i^{(1) } \cdot \nabla \times \overline{\bm w} \dif {\vec \xi} - \im  \int_B \nu {\bm \theta}_i^{(1) } \cdot  \overline{\bm w} \dif {\vec \xi}  = \im \int_B \nu {\bm \theta}_i^{(0)} \cdot  \overline{\bm w} \dif {\vec \xi},
\end{align}
for all ${\bm w} \in X$ where $X = \{ {\bm u} \in {\bm H}(\hbox{curl}, \Omega) : {\bm n} \times {\bm u} = {\bm 0},  \nabla \cdot {\bm \theta}_i^{(1)} =0 \text{  in $\Omega \setminus \overline{B}$} \}$  and ${\bm \theta}_i^{(0)} =\tilde{\bm \theta}_i^{(0)} + {\bm e}_i \times {\bm \xi}$ is a source term obtained by weakly solving the transmission problem for $\tilde{\bm \theta}_i^{(0)} \in X$~\cite{ben2020} in a similar way to the above.}
 In our practical computations,  the Coulomb gauge $\nabla \cdot {\bm \theta}_i^{(1)} =0$ in $\Omega \setminus \overline{B}$ is replaced by regularisation  using a small regularisation parameter $\varepsilon \ll 1$~\cite{zaglmayrphd,ledgerzaglmayr2010} leading to the weak variational problem: find ${\bm \theta}_i^{(1), \varepsilon}\in X^\varepsilon $ such that
\begin{align}
\int_\Omega \tilde{\mu}_r^{-1} \nabla \times {\bm \theta}_i^{(1), \varepsilon} \cdot \nabla \times \overline{\bm w} \dif {\vec \xi} - \im  \int_B \nu {\bm \theta}_i^{(1), \varepsilon} \cdot  \overline{\bm w} \dif {\vec \xi} + \varepsilon \int_{\Omega \setminus \overline{B}}  {\bm \theta}_i^{(1), \varepsilon} \cdot  \overline{\bm w} \dif {\vec \xi}  = \im \int_B \nu {\bm \theta}_i^{(0),\varepsilon} \cdot  \overline{\bm w} \dif {\vec \xi},
\end{align}
for all ${\bm w} \in X^\varepsilon$, is sought where $X^\varepsilon = \{ {\bm u} \in {\bm H}(\hbox{curl}, \Omega) : {\bm n} \times {\bm u} = {\bm 0}\}$  and ${\bm \theta}_i^{(0),\varepsilon} =\tilde{\bm \theta}_i^{(0),\varepsilon} + {\bm e}_i \times {\bm \xi}$ is a source term with $\tilde{\bm \theta}_i^{(0),\varepsilon}\in X^\varepsilon $ obtained in a similar way. However, as ${\bm \theta}_i^{(0),\varepsilon}$ is independent of $\nu$, it is assumed to be a known source in the following. Then, discrete approximations $ {\bm \theta}_i^{(1),hp}\in X^{hp} \cap X^\varepsilon $  to the weak solutions ${\bm \theta}_i^{(1), \varepsilon}\in X^\varepsilon$ are sought  using: find  $ {\bm \theta}_i^{(1),hp} \in X^{hp} \cap X^\varepsilon$ such that
\begin{align}
\int_\Omega \tilde{\mu}_r^{-1} \nabla \times {\bm \theta}_i^{(1), hp} \cdot \nabla \times \overline{\bm w}^{hp} \dif {\vec \xi} - \im  \int_B \nu {\bm \theta}_i^{(1), hp} \cdot  \overline{\bm w}^{hp} \dif {\vec \xi} + \varepsilon \int_{\Omega \setminus \overline{B}}  {\bm \theta}_i^{(1), hp} \cdot  \overline{\bm w}^{hp} \dif {\vec \xi}  = \im \int_B \nu {\bm \theta}_i^{(0),hp} \cdot  \overline{\bm w}^{hp} \dif {\vec \xi},
\end{align}
for all ${\bm w}^{hp} \in X^{hp} \cap X^\varepsilon$, where $X^{hp} \subset {\vec H}(\hbox{curl})$ is a set of ${\vec H}(\hbox{curl})$ conforming finite element basis functions~\cite{zaglmayrphd}. The discrete approximations are obtained by forming a non-overlapping partition of $\Omega$ using unstructured tetrahedral meshes with the addition of thin prismatic layers to handle thin skin-depth effects~\cite{Elgy2024_preprint}. The discretisation allows for refinement of the mesh ($h$-refinement) and increasing the order of the elements ($p$-refinement). We employ the \texttt{NGSolve} finite element library~\cite{NGSolve} and its integrated \texttt{Netgen} mesh generator~\cite{netgendet} in our \texttt{MPT-Calculator}~\footnote{The open source  \texttt{MPT-Calculator} software can be accessed at \texttt{https://github.com/MPT-Calculator/MPT-Calculator} and version 1.51 with commit number 94a825c
 was employed in this work.}  software for this purpose.

\subsection{Errors due to regularisation in the continuous problem} \label{sect:errorscontproblem}
An a-prori error estimate obtained for ${\bm A}$-based formulation of the regularised magenetostaic and eddy current problems with respect to the un-regularised weak solution is stated in~\cite{ledgerzaglmayr2010} and this can immediately be applied to the approximation of the weak solutions $\tilde{\bm \theta}_i^{(0)} \in X$ by $\tilde{\bm \theta}_i^{(0), \varepsilon} \in X^\varepsilon$ 
and
 ${\bm \theta}_i^{(1)} \in X$ by ${\bm \theta}_i^{(1), \varepsilon}  \in X^\varepsilon$ leading to
\begin{align}
\| \tilde{\bm \theta}_i^{(0), \varepsilon} - \tilde{\bm \theta}_i^{(0)}  \|_{{\bm H}(\hbox{curl})} \le C \varepsilon  ,\qquad 
\| {\bm \theta}_i^{(1), \varepsilon} - {\bm \theta}_i^{(1)}  \|_{{\bm H}(\hbox{curl})} \le C \varepsilon \nu .
\end{align}
These  estimates can be used to estimate the error associated with using regularised solutions to obtain the tensor coefficients
\begin{subequations}
\begin{align}
(\mathcal{N}^{0,\varepsilon} )_{ij}&:=\alpha^3\delta_{ij}\int_{B}(1-\tilde{\mu}_r^{-1})\dif \bm{\xi}+\frac{\alpha^3}{4}\int_{\Omega }\tilde{\mu}_r^{-1}\nabla\times\tilde{\bm{\theta}}_i^{(0),\varepsilon }\cdot\nabla\times\tilde{\bm{\theta}}_j^{(0),\varepsilon}\dif \bm{\xi},\\
({\mathcal R}^\varepsilon)_{ij} = &- \frac{\alpha^3}{4}  \int_\Omega \tilde{\mu}_r^{-1} \nabla \times {\bm \theta}_i^{(1), \varepsilon} \cdot \nabla \times \overline{{\bm \theta}_j^{(1), \varepsilon}} \dif {\bm \xi}, \\
({\mathcal I}^\varepsilon)_{ij} = &\frac{\alpha^3}{4}  \int_B \frac{1}{\nu} \nabla \times {\mu}_r^{-1} \nabla \times {\bm \theta}_i^{(1), \varepsilon} \cdot  \nabla \times {\mu}_r^{-1} \nabla \times \overline{{\bm \theta}_j^{(1), \varepsilon}} \dif {\bm \xi},
\end{align}
\end{subequations}
with respect to $(\mathcal{N}^{0} )_{ij}$, $({\mathcal R})_{ij}$ and $({\mathcal I})_{ij}$, respectively, which is achieved through the following estimates:
\begin{lemma} \label{lemma:errorduetoreg}
{Given  weak regularised approximations 
 $\tilde{\bm \theta}_i^{(0), \varepsilon} \in  X^\varepsilon$ to  $\tilde{\bm \theta}_i^{(0)}\in X$ and  ${\bm \theta}_i^{(1), \varepsilon} \in  X^\varepsilon$ to ${\bm \theta}_i^{(1)} \in X$, the associated error in the regularised tensor coefficients on a finite domain $\Omega$ can be estimated as}
\begin{subequations}
\begin{align}
\left | ({\mathcal N}^{(0)})_{ij} - ({\mathcal N}^{(0),\varepsilon})_{ij} \right | \le & C \varepsilon  \alpha^3   \left ( \varepsilon  + \sum_{n=1}^3 \| \nabla \times  \tilde{\bm \theta}_n^{(0), \varepsilon} \|_{L^2(\Omega)}  \right ), \label{errorregno}\\
\left | ({\mathcal R})_{ij} - ({\mathcal R}^\varepsilon)_{ij} \right | \le & C \varepsilon \nu \alpha^3   \left ( \varepsilon \nu + \sum_{n=1}^3 \| \nabla \times  {\bm \theta}_n^{(1), \varepsilon} \|_{L^2(\Omega)}  \right ), \label{errorregr}\\
\left | ({\mathcal I})_{ij} - ({\mathcal I}^\varepsilon)_{ij} \right | \le & C \varepsilon \nu^2 \alpha^3   \left ( 1+ \varepsilon \nu + \sum_{n=1}^3 \|   {\bm \theta}_n^{(1), \varepsilon} \|_{L^2(B)}  \right ).  \label{errorregi}
\end{align}
\end{subequations}
\end{lemma}
\begin{proof}
To show (\ref{errorregno}), it is useful to introduce  $n({\bm u},{\bm v}):=  \frac{\alpha^3}{4}  \int_\Omega \tilde{\mu}_r^{-1} \nabla \times {\bm u} \cdot \nabla \times \overline{{\bm v}} \dif {\bm \xi}$~\footnote{The vector solutions $\tilde{\vec \theta}_j^{(0)}$ are real valued, but including the complex conjugate in the defiition of  $n({\bm u},{\bm v})$ is useful for what follows.}. It then follows that
\begin{align}
 ({\mathcal N}^{(0)})_{ij} - ({\mathcal N}^{(0),\varepsilon})_{ij}  = & n(\tilde{\bm \theta}_i^{(0) } , \tilde{\bm \theta}_j^{(0) }) -  n(\tilde{\bm \theta}_i^{(0), \varepsilon } , \tilde{\bm \theta}_j^{(0), \varepsilon }) \nonumber \\
=& n(\tilde{\bm \theta}_i^{(0) }- \tilde{\bm \theta}_i^{(0), \varepsilon } , \tilde{\bm \theta}_j^{(0) }) + n(  \tilde{\bm \theta}_i^{(0), \varepsilon } , \tilde{\bm \theta}_j^{(0) }) -  n(\tilde{\bm \theta}_i^{(0), \varepsilon } , \tilde{\bm \theta}_j^{(0), \varepsilon })  \nonumber \\
=& n(\tilde{\bm \theta}_i^{(0) }-\tilde {\bm \theta}_i^{(0), \varepsilon } , \tilde{\bm \theta}_j^{(0) }) + n(  \tilde{\bm \theta}_i^{(0), \varepsilon } , \tilde{\bm \theta}_j^{(0) } - \tilde{\bm \theta}_j^{(0), \varepsilon })  \nonumber \\
=& n(\tilde{\bm \theta}_i^{(0) }- \tilde{\bm \theta}_i^{(0), \varepsilon } ,\tilde{\bm \theta}_j^{(0) } -\tilde{\bm \theta}_j^{(0) ,\varepsilon} ) +n(\tilde{\bm \theta}_i^{(0) }- \tilde{\bm \theta}_i^{(0), \varepsilon } , {\bm \theta}_j^{(0) ,\varepsilon} ) \nonumber \\
& + n(  \tilde{\bm \theta}_i^{(0), \varepsilon } , \tilde{\bm \theta}_j^{(0) } - \tilde{\bm \theta}_j^{(0), \varepsilon })  \nonumber .
     \end{align}
Thus, by the Cauchy-Schwartz inequality,
\begin{align}
\left |
({\mathcal N}^{(0)})_{ij} - ({\mathcal N}^{(0),\varepsilon})_{ij} 
\right | \le &  C \alpha^3 \left (  \| \nabla \times (\tilde{\bm \theta}_i^{(0) }- \tilde{\bm \theta}_i^{(0), \varepsilon }) \|_{L^2 (\Omega)}  \| \nabla \times (\tilde{\bm \theta}_j^{(0) }-\tilde{\bm \theta}_j^{(0), \varepsilon }) \|_{L^2 (\Omega)} 
\right . \nonumber \\
 & + \left .
 \| \nabla \times (\tilde{\bm \theta}_i^{(0) }- {\bm \theta}_i^{(0), \varepsilon }) \|_{L^2 (\Omega)}
 \| \nabla \times \tilde{\bm \theta}_j^{(0) , \varepsilon}\|_{L^2 (\Omega)} + \right . \nonumber\\
 & + \left .
 \| \nabla \times (\tilde{\bm \theta}_j^{(0) }- {\bm \theta}_j^{(0), \varepsilon }) \|_{L^2 (\Omega)}
 \| \nabla \times \tilde{\bm \theta}_i^{(0) , \varepsilon}\|_{L^2 (\Omega)}
\right ),
\end{align}     
with $C$ independent of $\nu,\alpha$, but depending on $\mu_r$.    
Since, $\| \nabla \times (\tilde{\bm \theta}_i^{(0) }- \tilde{\bm \theta}_i^{(0), \varepsilon }) \|_{L^2 (\Omega)} \le \| \tilde{\bm \theta}_i^{(0) }- \tilde{\bm \theta}_i^{(0), \varepsilon }  \|_{{\bm H} (\text{curl},\Omega)} \le C \varepsilon $ the result in (\ref{errorregno}) immediately follows.

The result for (\ref{errorregr}) follows analogously by writing $r({\bm u},{\bm v}):=-n({\bm u},{\bm v})=  - \frac{\alpha^3}{4}  \int_\Omega \tilde{\mu}_r^{-1} \nabla \times {\bm u} \cdot \nabla \times \overline{{\bm v}} \dif {\bm \xi}$ so that
\begin{align}
 ({\mathcal R})_{ij} - ({\mathcal R}^\varepsilon)_{ij}  =& r({\bm \theta}_i^{(1) }- {\bm \theta}_i^{(1), \varepsilon } , {\bm \theta}_j^{(1) } -{\bm \theta}_j^{(1) ,\varepsilon} ) +r({\bm \theta}_i^{(1) }- {\bm \theta}_i^{(1), \varepsilon } , {\bm \theta}_j^{(1) ,\varepsilon} ) \nonumber \\
& + r(  {\bm \theta}_i^{(1), \varepsilon } , {\bm \theta}_j^{(1) } - {\bm \theta}_j^{(1), \varepsilon })  \nonumber ,
     \end{align}
and by the Cauchy-Schwartz inequality,
\begin{align}
\left |
({\mathcal R})_{ij} - ({\mathcal R}^\varepsilon)_{ij} 
\right | \le &  C \alpha^3 \left (  \| \nabla \times ({\bm \theta}_i^{(1) }- {\bm \theta}_i^{(1), \varepsilon }) \|_{L^2 (\Omega)}  \| \nabla \times ({\bm \theta}_j^{(1) }- {\bm \theta}_j^{(1), \varepsilon }) \|_{L^2 (\Omega)} 
\right . \nonumber \\
 & + \left .
 \| \nabla \times ({\bm \theta}_i^{(1) }- {\bm \theta}_i^{(1), \varepsilon }) \|_{L^2 (\Omega)}
 \| \nabla \times {\bm \theta}_j^{(1) , \varepsilon}\|_{L^2 (\Omega)} + \right . \nonumber\\
 & + \left .
 \| \nabla \times ({\bm \theta}_j^{(1) }- {\bm \theta}_j^{(1), \varepsilon }) \|_{L^2 (\Omega)}
 \| \nabla \times {\bm \theta}_i^{(1) , \varepsilon}\|_{L^2 (\Omega)}
\right ).
\end{align}         
Then, since, $\| \nabla \times ({\bm \theta}_i^{(1) }- {\bm \theta}_i^{(1), \varepsilon }) \|_{L^2 (\Omega)} \le \| {\bm \theta}_i^{(1) }- {\bm \theta}_i^{(1), \varepsilon }  \|_{{\bm H} (\text{curl},\Omega)} \le C \varepsilon \nu$, the result in (\ref{errorregr}) immediately follows.

To show (\ref{errorregi}), it is useful to note 
\begin{align}
({\mathcal I}^\varepsilon)_{ij} = & \frac{\alpha^3}{4}  \int_B \frac{1}{\nu}  \nabla \times {\mu}_r^{-1} \nabla \times 
{\bm \theta}_i^{(1), \varepsilon } \cdot
 \nabla \times \mu_r^{-1} \nabla \times \overline{ {\bm \theta}_j^{(1), \varepsilon } } \dif {\bm \xi} \nonumber\\
= &  \frac{\alpha^3}{4}   \int_B \nu \left ( {\bm \theta}_i^{(1),\varepsilon } +{\bm  \theta}_i^{(0)}  \right ) \cdot 
\left ( \overline{{\bm \theta}_j^{(1),\varepsilon }} + {\bm \theta}_j^{(0)}  \right ) \dif {\bm \xi} \nonumber,
\end{align}
and
introduce  $i({\bm u},{\bm v}):=\frac{\alpha^3}{4} \int_B \nu {\bm u} \cdot \overline{\bm v} \dif {\bm \xi}$ so that
\begin{align}
 ({\mathcal I})_{ij} - ({\mathcal I}^\varepsilon)_{ij}
 & = i ( {\bm \theta}_i^{(1)} + {\bm \theta}_i^{(0)}
 , {\bm \theta}_j^{(1) } + {\bm \theta}_j^{(0)}  ) -i ( {\bm \theta}_i^{(1), \varepsilon} + {\bm \theta}_i^{(0)}
 , {\bm \theta}_j^{(1) , \varepsilon} + {\bm \theta}_j^{(0)}  ) \nonumber \\
 & = i ({\bm \theta}_i^{(1)}, {\bm \theta}_j^{(1)}) +i( {\bm \theta}_i^{(1)}, {\bm \theta}_j^{(0)} )+i( {\bm \theta}_i^{(0)}, {\bm \theta}_j^{(1)}) \nonumber\\
 & \qquad - \left ( i ({\bm \theta}_i^{(1), \varepsilon}, {\bm \theta}_j^{(1), \varepsilon}) +i ( {\bm \theta}_i^{(1), \varepsilon }, {\bm \theta}_j^{(0)} )+i( {\bm \theta}_i^{(0)}, {\bm \theta}_j^{(1),\varepsilon}) 
  \right ) \nonumber \\
  & = i ({\bm \theta}_i^{(1)} - {\bm \theta}_i^{(1), \varepsilon} , {\bm \theta}_j^{(1)}) +
   i ( {\bm \theta}_i^{(1), \varepsilon} , {\bm \theta}_j^{(1)}) 
  +i( {\bm \theta}_i^{(1)}
 -{\bm \theta}_i^{(1),\varepsilon}
  , {\bm \theta}_j^{(0)} )+i( {\bm \theta}_i^{(0)}, {\bm \theta}_j^{(1)}-{\bm \theta}_j^{(1),\varepsilon }) \nonumber\\
 & \qquad -  i ({\bm \theta}_i^{(1), \varepsilon}, {\bm \theta}_j^{(1), \varepsilon})  
  \nonumber \\
  & = i ({\bm \theta}_i^{(1)} - {\bm \theta}_i^{(1), \varepsilon} , {\bm \theta}_j^{(1)}- {\bm \theta}_j^{(1), \varepsilon}  ) 
  +i( {\bm \theta}_i^{(1)}
 -{\bm \theta}_i^{(1),\varepsilon}
  , {\bm \theta}_j^{(0)} )+i( {\bm \theta}_i^{(0)}, {\bm \theta}_j^{(1)}-{\bm \theta}_j^{(1),\varepsilon }) \nonumber\\
 & \qquad +  i ({\bm \theta}_i^{(1)}-
 {\bm \theta}_i^{(1), \varepsilon}, {\bm \theta}_j^{(1), \varepsilon})  +  i ({\bm \theta}_i^{(1), \varepsilon},
 {\bm \theta}_j^{(1)} - {\bm \theta}_j^{(1), \varepsilon}   ) .
  \nonumber  
\end{align}
Thus, by the Cauchy-Schwartz inequality, and noting that $\  \| {\bm \theta}_j^{(0)} \|_{L^2(B)}  \le C$ independent of $\nu$, then
\begin{align}
\left |  ({\mathcal I})_{ij} - ({\mathcal I}^\varepsilon)_{ij} \right | \le & C \alpha^3 \nu \left (
\| {\bm \theta}_i^{(1)} - {\bm \theta}_i^{(1),  \varepsilon } \|_{L^2(B)} \| {\bm \theta}_j^{(1)} - {\bm \theta}_j^{(1), \varepsilon } \|_{L^2(B)}  +
\| {\bm \theta}_i^{(1)} - {\bm \theta}_i^{(1) ,\varepsilon} \|_{L^2(B) } 
+\| {\bm \theta}_j^{(1)} - {\bm \theta}_j^{(1) ,\varepsilon} \|_{L^2(B) }\right . \nonumber \\
&  \left . \qquad  +  \| {\bm \theta}_i^{(1)} - {\bm \theta}_i^{(1),\varepsilon } \|_{L^2(B)}  \|  {\bm \theta}_j^{(1),\varepsilon } \|_{L^2(B)}
+ \| {\bm \theta}_j^{(1)} - {\bm \theta}_j^{(1),\varepsilon } \|_{L^2(B)}  \|  {\bm \theta}_i^{(1) ,\varepsilon} \|_{L^2(B)} 
 \right ) \nonumber .
\end{align}
Finally, since $\|  ({\bm \theta}_i^{(1) }- {\bm \theta}_i^{(1), \varepsilon }) \|_{L^2 (\Omega)} \le \| {\bm \theta}_i^{(1) }- {\bm \theta}_i^{(1), \varepsilon }  \|_{{\bm H} (\text{curl},\Omega)} \le C \varepsilon \nu$, then the result in (\ref{errorregi}) immediately follows.

\end{proof}

Introducing
\begin{align}
({\mathcal Z}^\varepsilon)_{ij} : = & ([ {\mathcal R}^\varepsilon (\nu)  , {\mathcal I}^\varepsilon (\nu)])_{ij},\nonumber
\end{align}
as the components of the commutator computed with the components of ${\mathcal R}^\varepsilon$ and ${\mathcal I}^\varepsilon$ then
\begin{align}
{\mathcal Z} - {\mathcal Z}^\varepsilon = & ( {\mathcal R} {\mathcal I } - {\mathcal I}{\mathcal R}) - 
( {\mathcal R}^\varepsilon {\mathcal I }^\varepsilon - {\mathcal I}^\varepsilon {\mathcal R}^\varepsilon)\nonumber \\
 = &   ( {\mathcal R} - {\mathcal R}^\varepsilon) {\mathcal I} + {\mathcal R}^\varepsilon ( {\mathcal I}- {\mathcal I}^\varepsilon) - ( {\mathcal I} - {\mathcal I}^\varepsilon) {\mathcal R} + {\mathcal I}^\varepsilon ( {\mathcal R} - {\mathcal R}^\epsilon ) \nonumber,
\end{align}
so that
\begin{align}
\left \| {\mathcal Z} - {\mathcal Z}^\varepsilon  \right \| \le
\|   {\mathcal R} - {\mathcal R}^\varepsilon  \| \|  {\mathcal I}  \| + 
\|  {\mathcal R}^\varepsilon \| \|   {\mathcal I}- {\mathcal I}^\varepsilon \| +
\|   {\mathcal I} - {\mathcal I}^\varepsilon \| \|  {\mathcal R} \| +
\|  {\mathcal I}^\varepsilon \| \|  {\mathcal R} - {\mathcal R}^\epsilon \|. \nonumber
\end{align}
Then, from  Lemma~\ref{lemma:errorduetoreg} and using the spectral representation of ${\vec \theta}_i^{(1)}$ from (42) in~\cite{LedgerLionheart2020spect}, the leading order terms in the error can be estimated as
\begin{align}
\left \| {\mathcal Z} - {\mathcal Z}^\varepsilon  \right \| \le C \alpha^3 \varepsilon \nu \left ( 
\|  {\mathcal I}  \| + 
\nu \|  {\mathcal R}^\varepsilon \|   +
\nu \|  {\mathcal R} \| +
\| {\mathcal I}^\varepsilon \| \right ). \label{eqn:errestZZepsilon}
\end{align}
In a similar way, introducing
\begin{align}
{\mathcal Z}^{(0),\varepsilon} (\nu) := & [ {\mathcal N} ^{(0),\varepsilon}, {\mathcal I}^\varepsilon (\nu) ],\qquad
\tilde{\mathcal Z}^\varepsilon : =  [ \tilde{\mathcal R}^\varepsilon (\nu)  , {\mathcal I}^\varepsilon (\nu)] ,
\end{align}
then
\begin{align}
\left \| {\mathcal Z}^{(0)} - {\mathcal Z}^{(0),\varepsilon}  \right \| \le &
\|   {\mathcal N}^{(0)} - {\mathcal N}^{(0),\varepsilon}  \| \|  {\mathcal I}  \| + 
\|  {\mathcal N}^{(0),\varepsilon} \| \|   {\mathcal I}- {\mathcal I}^\varepsilon \| +
\|   {\mathcal I} - {\mathcal I}^\varepsilon \| \|  {\mathcal N}^{(0)} \| +
\|  {\mathcal I}^\varepsilon \| \|  {\mathcal N}^{(0)} - {\mathcal N}^{(0),\epsilon} \|\nonumber\\
 \le& C \alpha^3 \varepsilon \left (
\|  {\mathcal I}  \| + \|  {\mathcal I}^\varepsilon  \| +\nu^2 \left ( \|  {\mathcal N}^{(0),\varepsilon} \| +  \|  {\mathcal N}^{(0)} \|  \right )
\right ), \label{eqn:errestZ0Z0epsilon}
\end{align}
and
\begin{align}
\left \| \tilde{\mathcal Z} - \tilde{\mathcal Z}^\varepsilon  \right \| \le &
(\|   {\mathcal N}^{(0)} - {\mathcal N}^{(0),\varepsilon}\| +\| {\mathcal R} - {\mathcal R}^\varepsilon \|) \|  {\mathcal I}  \| + 
\|  \tilde{\mathcal R}^\varepsilon \| \|   {\mathcal I}- {\mathcal I}^\varepsilon \| +
\|   {\mathcal I} - {\mathcal I}^\varepsilon \| \|  \tilde{\mathcal R} \|  \nonumber \\
&+
\|  {\mathcal I}^\varepsilon \|(\|  {\mathcal N}^{(0)} - {\mathcal N}^{(0),\epsilon} \|+ \|  {\mathcal R} - {\mathcal R}^\epsilon \|) \nonumber\\
\le & C \alpha^3 \varepsilon \left(\|  {\mathcal I}  \| + 
\| {\mathcal I}^\varepsilon \|  +  \nu \left ( 
\|  {\mathcal I}  \| _F+ 
\nu \|  {\mathcal R}^\varepsilon \|   +
\nu \|  {\mathcal R} \| +
\| {\mathcal I}^\varepsilon \| \right )\right ).\label{eqn:erresttilZtilZepsilon}
\end{align}
From (\ref{eqn:errestZZepsilon}), (\ref{eqn:errestZ0Z0epsilon}) and (\ref{eqn:erresttilZtilZepsilon}) we can observe that the regularisation will perturb the commutator and this perturbation will be of order $\varepsilon$. This effect must be considered carefully when considering the extent to which ${\mathcal N}^{(0)}$, ${\mathcal R}$ and ${\mathcal I}$ commute. In the next section, we consider the additional errors associated with numerical discretisation.

\subsection{Errors due to discretisation}

In the case when the continuous solutions ${\vec \theta}_i^{(1),\varepsilon}$ and ${\vec \theta}_i^{(0),\varepsilon}$ are smooth vector fields (e.g. corresponding to when  $B$ has homogenous materials, $\Gamma = \partial B$ is a  smooth surface and $\omega$ is small, such that the skin depth  $\delta$  is similar to the size of the object $\alpha$  (i.e. $\tau  \approx 1$)), then, based on known results in two-dimensions~\cite{ainsworth2002}, the rates of convergence are conjectured to follow the same pattern to those known for elliptic problems using $H^1$ conforming $hp$ finite element discretisations~\cite{szabo1991finite} in three-dimensions, namely
\begin{align}
\| {\bm \theta}_i^{(1), \varepsilon} - {\bm \theta}_i^{(1), hp} \|_{{\bm H}(\text{curl},\Omega)} \le  k{ N_{Dof}^{ - p/3}},
\end{align}
 {where $N_{Dof}$ denotes the number of degrees of freedom,} for $h$-refinement using order $p$ elements and the estimate
\begin{align}
\| {\bm \theta}_i^{(1), \varepsilon} - {\bm \theta}_i^{(1), hp} \|_{{\bm H}(\text{curl},\Omega)} \le  k\exp( -a N_{Dof}^{\vartheta}), \label{eqn:expconv}
\end{align}
with $\vartheta \ge 1/3$ for $p$-refinement, with constants $k$ and $a$ depending on the problem parameters. However, for higher frequencies, larger conductivities and/or high $\mu_r$,  $\tau $ will becomes small compared to $|B|$ and $p$-refinement provides only algebraic, rather than exponential, convergence. Furthermore, in the case where  $B$ has sharp edges or corners, in addition to  higher frequencies, larger conductivities and/or high $\mu_r$, then $p$-refinement will also become algebraic. 
In such cases, combining $p$--refinement with local $h$-refinement,  in the vicinity of the object boundary by using prismatic layers, can improve the rate of convergence and, if the correct combination of $h$-- and $p$--refinements are made, lead to exponential convergence similar to (\ref{eqn:expconv}) with $\vartheta \approx 1/5$. Examples that demonstrate this in practice for MPT spectral signatue object characterisation can be found in~\cite{Elgy2024_preprint}.

 For simplicity in the following, we assume that $B$  and $\Omega$ are Lipschitz polyhedra so that they can be exactly discretised by flat faced tetrahedral meshes.  Assuming that $h$ and $p$ can be chosen such that (\ref{eqn:expconv}) holds, the following error estimate for the approximated tensor coefficients with respect to the continuous regularised solution can be obtained:

\begin{lemma} \label{lemma:errorhpfem}
{Given  weak $hp$ finite element approximations ${  \tilde{\bm \theta}_i^{(0), hp }} \in X^{hp} \cap  X^\varepsilon $ to $\tilde{\bm \theta}_i^{(0), \varepsilon } \in  X^\varepsilon$ and ${  {\bm \theta}_i^{(1), hp }} \in   X^{hp} \cap  X^\varepsilon$ to ${\bm \theta}_i^{(1), \varepsilon } \in  X^\varepsilon$, the associated error in the discrete approximation of the tensor coefficients with respect to the regularised coefficients can be estimated as}
\begin{align}
\left | ({\mathcal N}^{(0),\varepsilon})_{ij} - ({\mathcal N}^{(0),hp} )_{ij} \right | \le & C    \exp(  -a N_{dof}^\theta)  \left (  \exp(  - a  N_{dof}^\theta)  +  \| \nabla \times  \tilde{\bm \theta}_i^{(0), hp} \|_{L^2(\Omega)} +  \| \nabla \times  \tilde{\bm \theta}_j^{(0), hp} \|_{L^2(\Omega)} \right ), \label{errorreghpn0}
 \end{align}
 with $C$ and $a$ depending on $\varepsilon$ and $\alpha$ and
\begin{subequations}
 \begin{align} 
\left | ({\mathcal R}^\varepsilon)_{ij} - ({\mathcal R}^{hp} )_{ij} \right | \le & C  \exp(  - a  N_{dof}^\theta)  \left (  \exp(  - a  N_{dof}^\theta)  +  \| \nabla \times  {\bm \theta}_i^{(1), hp} \|_{L^2(\Omega)} +  \| \nabla \times  {\bm \theta}_j^{(1), hp} \|_{L^2(\Omega)} \right ), \label{errorreghpr}\\
\left | ({\mathcal I}^\varepsilon )_{ij} - ({\mathcal I}^{hp})_{ij} \right | \le & C  \exp(  - a  N_{dof}^\theta)   \left (
1+  \exp( -  a N_{dof}^\theta)  +
\|   {\bm \theta}_i^{(1), hp } \|_{L^2(B)} +  \|   {\bm \theta}_j^{(1),  hp } \|_{L^2(B)} \right ).  \label{errorreghpi}
\end{align}
\end{subequations}
with $C$ and $a$ additionally depending on $\nu$.
\end{lemma}
\begin{proof}
To show (\ref{errorreghpn0}),   a similar approach to the proof of Lemma~\ref{lemma:errorduetoreg} is followed leading to
\begin{align}
\left |
({\mathcal N}^{(0),\varepsilon})_{ij} - ({\mathcal N}^{(0),hp})_{ij} 
\right | \le &  C \alpha^3 \left (  \| \nabla \times (\tilde{\bm \theta}_i^{(0),\varepsilon }- \tilde{\bm \theta}_i^{(0), hp }) \|_{L^2 (\Omega)}  \| \nabla \times (\tilde{\bm \theta}_j^{(0), \varepsilon }- \tilde{\bm \theta}_j^{(0), hp }) \|_{L^2 (\Omega)} 
\right . \nonumber \\
 & + \left .
 \| \nabla \times (\tilde{\bm \theta}_i^{(0) , \varepsilon}- {\bm \theta}_i^{(0), hp }) \|_{L^2 (\Omega)}
 \| \nabla \times \tilde{\bm \theta}_j^{(0) , hp}\|_{L^2 (\Omega)} + \right . \nonumber\\
 & + \left .
 \| \nabla \times (\tilde{\bm \theta}_j^{(0) ,\varepsilon}- {\bm \theta}_j^{(0), hp }) \|_{L^2 (\Omega)}
 \| \nabla \times \tilde{\bm \theta}_i^{(0) , hp }\|_{L^2 (\Omega)}
\right ) .
\end{align}     
Then, since $\| \nabla \times (\tilde{\bm \theta}_i^{(0), \varepsilon }- \tilde{\bm \theta}_i^{(0),hp }) \|_{L^2 (\Omega)} \le \| \tilde{\bm \theta}_i^{(0) , \varepsilon}- \tilde{\bm \theta}_i^{(0), hp }  \|_{{\bm H} (\text{curl},\Omega)} \le  \frac{k(\varepsilon ,\nu,\alpha ) }{\exp( a(\varepsilon,\nu,\alpha) N_{Dof}^{\theta})}$ the result in (\ref{errorreghpn0}) immediately follows. The result for (\ref{errorreghpr}) is obtained in a similar way.

To show (\ref{errorreghpi}),  then  by following a similar approach to the proof of Lemma~\ref{lemma:errorduetoreg}, it can be shown that
\begin{align}
\left |
({\mathcal I}^\varepsilon)_{ij} - ({\mathcal I}^{hp})_{ij} 
\right | \le &   C \alpha^3 \nu \left (
\| {\bm \theta}_i^{(1) , \varepsilon} - {\bm \theta}_i^{(1) , hp} \|_{L^2(B)} \| {\bm \theta}_j^{(1) , \varepsilon } - {\bm \theta}_j^{(1) , hp } \|_{L^2(B)}  +
\| {\bm \theta}_i^{(1), , \varepsilon} - {\bm \theta}_i^{(1) ,hp} \|_{L^2(B) } \right . \nonumber \\
 &  \qquad +\| {\bm \theta}_j^{(1),  \varepsilon} - {\bm \theta}_j^{(1) ,hp } \|_{L^2(B) }
 +  \| {\bm \theta}_i^{(1),  \varepsilon} - {\bm \theta}_i^{(1),hp } \|_{L^2(B)}  \|  {\bm \theta}_j^{(1),hp } \|_{L^2(B)}\nonumber\\
 & \left . \qquad 
+ \| {\bm \theta}_j^{(1),  \varepsilon} - {\bm \theta}_j^{(1),hp } \|_{L^2(B)}  \|  {\bm \theta}_i^{(1) ,hp} \|_{L^2(B)} 
 \right ) \nonumber .
\end{align}     
and, since $\|  {\bm \theta}_i^{(1), \varepsilon }- {\bm \theta}_i^{(1),hp } \|_{L^2 (\Omega)} \le \| {\bm \theta}_i^{(1) , \varepsilon}- {\bm \theta}_i^{(1), hp }  \|_{{\bm H} (\text{curl},\Omega)} \le  \frac{k(\varepsilon ,\nu,\alpha ) }{\exp( a(\varepsilon,\nu,\alpha) N_{Dof}^{\theta})}$, the result in (\ref{errorreghpi}) immediately follows.

\end{proof}
The estimates in Lemma~\ref{lemma:errorhpfem} indicate that, for the correct combination of $h$-- and $p$--refinements, the discrete regularised tensor coefficients will converge exponentially fast to the regularised continuous counterpart. The following result combines both the effects of regularisation and discretisation.

\begin{theorem} \label{thm:errohpfemwithregtocont}
{Given  weak $hp$ finite element approximations ${  \tilde{\bm \theta}_i^{(0), hp }} \in X^{hp} \cap  X^\varepsilon $ to $\tilde{\bm \theta}_i^{(0), \varepsilon } \in  X^\varepsilon$ and ${  {\bm \theta}_i^{(1), hp }} \in   X^{hp} \cap  X^\varepsilon$ to ${\bm \theta}_i^{(1), \varepsilon } \in  X^\varepsilon$, which in turn approximate   $\tilde{\bm \theta}_i^{(0)}\in X$ and ${\bm \theta}_i^{(1)} \in X$,  respectively,  the error in the discrete approximation of the tensor coefficients can be estimated as}
\begin{subequations}
\begin{align}
\left | ({\mathcal N}^{(0)})_{ij} - ({\mathcal N}^{(0),hp} )_{ij} \right | \le & C_1 \varepsilon \nu \alpha^3   \left ( \varepsilon \nu +  \| \nabla \times  \tilde{\bm \theta}_i^{(0), \varepsilon} \|_{L^2(\Omega)} +  \| \nabla \times  \tilde{\bm \theta}_j^{(0), \varepsilon} \|_{L^2(\Omega)} \right ) \nonumber \\
& +C_2    \exp( - a  N_{dof}^\theta)  \left (  \exp(  -a N_{dof}^\theta)  +  \| \nabla \times  \tilde{\bm \theta}_i^{(0), hp} \|_{L^2(\Omega)} +  \| \nabla \times  \tilde{\bm \theta}_j^{(0), hp} \|_{L^2(\Omega)} \right ), 
\label{errorreghpfulln0} 
\end{align}
\end{subequations}
where $C_2$ and $a$ depend on $\varepsilon$ and $\alpha$ and
\begin{subequations}
\begin{align}
\left | ({\mathcal R})_{ij} - ({\mathcal R}^{hp} )_{ij} \right | \le & C_1 \varepsilon \nu \alpha^3   \left ( \varepsilon \nu +  \| \nabla \times  {\bm \theta}_i^{(1), \varepsilon} \|_{L^2(\Omega)} +  \| \nabla \times  {\bm \theta}_j^{(1), \varepsilon} \|_{L^2(\Omega)} \right )
\nonumber \\
& + C_2   \exp(  -a  N_{dof}^\theta)   \left (  \exp( - a  N_{dof}^\theta)  +  \| \nabla \times  {\bm \theta}_i^{(1), hp} \|_{L^2(\Omega)} +  \| \nabla \times  {\bm \theta}_j^{(1), hp} \|_{L^2(\Omega)} \right ), \label{errorreghpfullr}\\
\left | ({\mathcal I} )_{ij} - ({\mathcal I}^{hp})_{ij} \right | \le &C_1 \varepsilon \nu^2 \alpha^3   \left ( 1+ \varepsilon \nu +  \|   {\bm \theta}_i^{(1), \varepsilon} \|_{L^2(B)} +  \|   {\bm \theta}_j^{(1), \varepsilon} \|_{L^2(B)} \right )\nonumber \\
& + C_2     \exp( - a N_{dof}^\theta)  \left (
1+   \exp(  - a  N_{dof}^\theta)  +
\|   {\bm \theta}_i^{(1), hp } \|_{L^2(B)} +  \|   {\bm \theta}_j^{(1),  hp } \|_{L^2(B)} \right ),  \label{errorreghpfulli}
\end{align}
\end{subequations}
with  $C_2$ and $a$ additionally depending on $\nu$.
\end{theorem}
\begin{proof}
The result in (\ref{errorreghpfulln0}) is obtained by using
\begin{align}
|  ({\mathcal N}^{(0)}
)_{ij} - ({\mathcal N}^{(0),hp})_{ij}| = &
|  ( {\mathcal N}^{(0)} )_{ij} - ({\mathcal N}^{(0),\varepsilon} )_{ij}  + 
({\mathcal N}^{(0),\varepsilon} )_{ij}
-({\mathcal N}^{(0),hp})_{ij} |  \nonumber \\
 \le & |  ({\mathcal N}^{(0)})_{ij} - ({\mathcal N}^{(0),\varepsilon} )_{ij}  | + | 
({\mathcal N}^{(0),\varepsilon} )_{ij}
-({\mathcal N}^{(0),hp})_{ij}|, \end{align}
which follows by the triangular inequality, and applying Lemmas~\ref{lemma:errorduetoreg} and~\ref{lemma:errorhpfem}. The other two cases are obtained analogously.
\end{proof}
As expected, this shows that the error is bounded by the sum of the error associated with the regularisation and the discretisation. For small $\varepsilon$, it will typically be the discretisation error that will dominate, which can be controlled by an appropriate choice of $h$ and $p$ leading to exponential convergence, as reported in the above estimates. These estimates could  be further improved by also considering the effects due to the iterative solution of the linear system in the finite element approximation~\cite{vanloon}, but we will assume that such effects are smaller than the discretisation errors. The use of reduced order models~\cite{hesthaven2016} to compute the spectral signature~\cite{ben2020,Elgy2024_preprint} will also introduce further errors, which can be controlled by tolerances in the model order reduction and careful choice of snapshot solution parameters, but these will not be considered further in this work.

If desired, the results of Theorem~\ref{thm:errohpfemwithregtocont} could be used to obtained a-prori estimates of $\left \| {\mathcal Z} - {\mathcal Z}^{hp}  \right \|$, \\$\left \| {\mathcal Z}^{(0)} - {\mathcal Z}^{(0),hp}  \right \|$ and
$\left \| \tilde{\mathcal Z} - \tilde{\mathcal Z}^{ hp}  \right \|$ where 
\begin{align}
{\mathcal Z}^{hp} (\nu) := & [ {\mathcal N} ^{(0),hp}, {\mathcal I}^{hp} (\nu) ], \qquad 
{\mathcal Z}^{(0),hp} (\nu) :=  [ {\mathcal N} ^{(0),hp}, {\mathcal I}^{hp} (\nu) ],\qquad
\tilde{\mathcal Z}^{hp} : =  [ \tilde{\mathcal R}^{hp} (\nu)  , {\mathcal I}^{hp} (\nu)],
\end{align}
in similar way to the estimates $\left \| {\mathcal Z} - {\mathcal Z}^\varepsilon  \right \|$, $\left \| {\mathcal Z}^{(0)} - {\mathcal Z}^{(0),\varepsilon}  \right \|$ and
$\left \| \tilde{\mathcal Z} - \tilde{\mathcal Z}^\varepsilon  \right \|$ in Section~\ref{sect:errorscontproblem}. In  Appendix~\ref{sect:appendix}, a constant free explicit  error indicator is proposed, which will be used to indicate the errors associated with regularisation in the next section.

\section{Numerical results} \label{sect:results}

In this section, a selection of numerical examples will be considered to demonstrate the results of the previous sections.  Firstly, an MPT characterisation of a simple irregular tetrahedron with six independent coefficients is considered. A discretisation sufficient to capture the converged solution will be used to demonstrate that the presence of non-vanishing commutators is not due to discretisation and regularisation errors. This example will also be used to motivate how the behaviour of the distance measures as a function of frequency is closely related to the behaviour of the eigenvalues of the MPT spectral signature. Next, an example of a toy gun model is considered with a barrel made of a highly magnetic material. Again, the interplay between the distance measures as a function of frequency and the behaviour of the eigenvalues of the MPT spectral signature is  considered. As a final application, we consider how the distance measures can provide additional feature information for a  object  classification problem using two Bayesian machine learning approaches.

\subsection{Homogeneous irregular polyhedron} \label{sect:adjtet}
We consider an irregular polyhedron $B=B_1 \cup B_2$ with $B_1$ being an irregular tetrahedron with vertices ${\bm v}_1  = ( 0,0,0)$, ${\bm v}_2 = ( 7, 0 , 0)$, ${\bm v}_3 = ( 5.5, 4.6,  0)$ and ${\bm v}_4 =(3.3 , 2 , 5)$
and $B_2$ the irregular tetrahedron with vertices ${\bm v}_1 = ( 0 , 0 , 0) $, ${\bm v}_2 = ( 7 , 0 , 0)$ ,$ {\bm v}_3 = ( 5.5 , -3.0,  0)$, $ {\bm v}_4 = ( 3.3 , 2 ,  5)$
with $B_1\cap B_2 \ne \emptyset$ and $| B_1| > | B_2 |$. The polyhedron  is shown in Figure~\ref{fig:twotetra_mesh}. The object is chosen to have  $\alpha =0.001 $m and homogeneous materials $\mu_r = 32$ and $\sigma_*=1\times 10^7$ S/m. 
The object $B$ is placed centrally in the a domain with bounding box  $[-100,100]^3$. The domain $\Omega$ is then discretised by a mesh of 18\,413 unstructured tetrahedra to which $L=3$ layers of prismatic elements following the "geometric increasing" strategy~\cite{Elgy2024_preprint} are included on $\Gamma$  resulting in 6461 prisms. The regularisation parameter is chosen as $\varepsilon = 1 \times 10^{-10}$, which is also used for the subsequent examples unless otherwise stated.

\begin{figure}[H]
\centering
\includegraphics[width=0.25\textwidth]{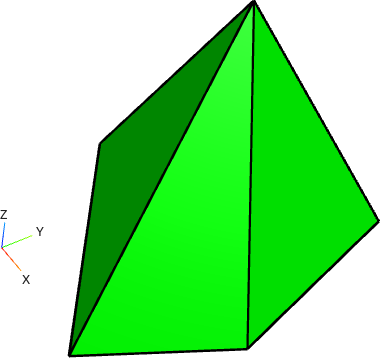}
\caption{Irregular polyhedron: Illustration of the geometry.}
\label{fig:twotetra_mesh}
\end{figure} 

The irregular nature of $B$, and its lack of rotational and reflectional symmetries, means that the exact ${\mathcal M}$ will have six independent coefficients.  In order to obtain converged results for the off-diagonal coefficients of $\tilde{\mathcal{R}}^{hp}$ and ${\mathcal{I}}^{hp}$ for $10^1 \le \omega \le 10^8$ rad/s it was found necessary to use order $p=3$ elements on the aforementioned mesh.

The converged eigenvalue spectral signatures $\lambda_i(\tilde{\mathcal R})$ and $\lambda_i({\mathcal I})$ are shown in Figure~\ref{fig:eigenvalues_two_tetra}, where we observe that the crossing of the eigenvalue curves for $\lambda_i({\mathcal I})$ at approximately $\omega =0.62 \times 10^7$ rad/s. {For simplicity of notation we have dropped the superscript $hp$ here and subsequently on the tensors. Importantly, here and in the following, we obtain the MPT spectral signature at discrete frequencies and  exclude those frequencies where the eigenvalues  of $\tilde{\mathcal R}$,   ${\mathcal R}$ and  ${\mathcal I}$ are repeated, but include situations where they may become close. Thus, the metrics and approximate distance measures from our new semi-metrics in Section~\ref{sect:metrics} can be applied.}

\begin{figure}[H]
\centering
$\begin{array}{c c}
\includegraphics[width=0.5\textwidth]{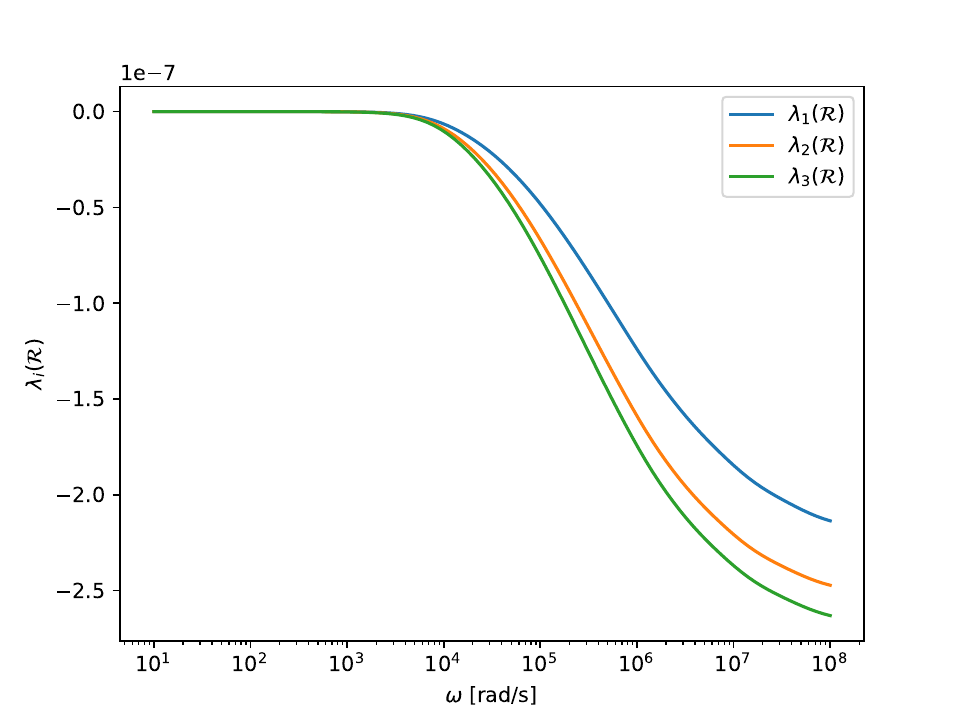} &
\includegraphics[width=0.5\textwidth]{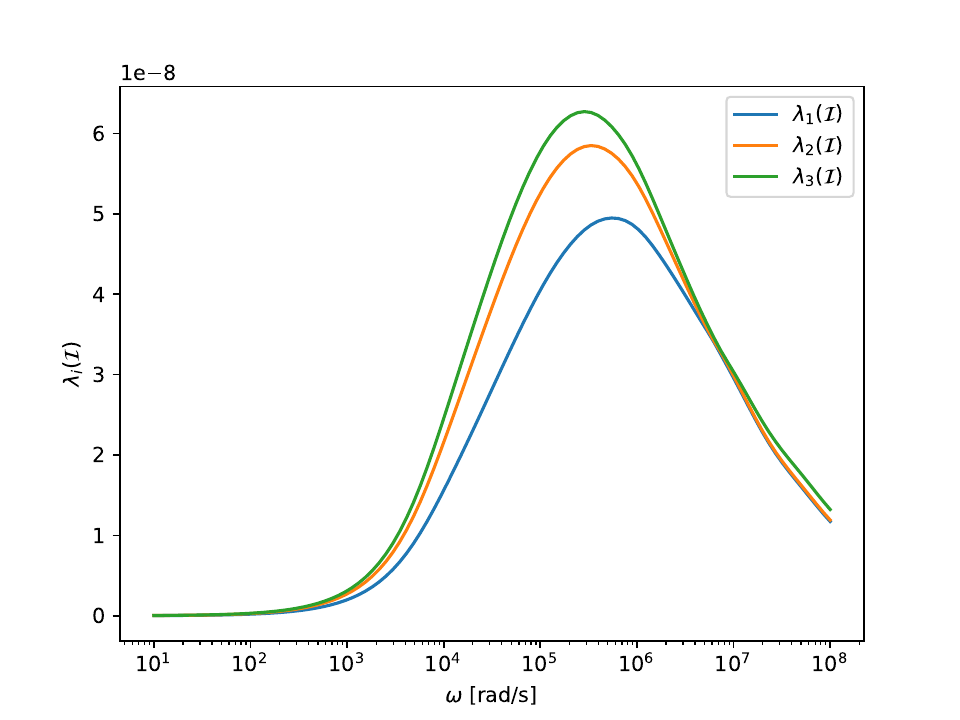}  \\
(a) & (b)
\end{array}$\\
$\begin{array}{c}
\includegraphics[width=0.5\textwidth]{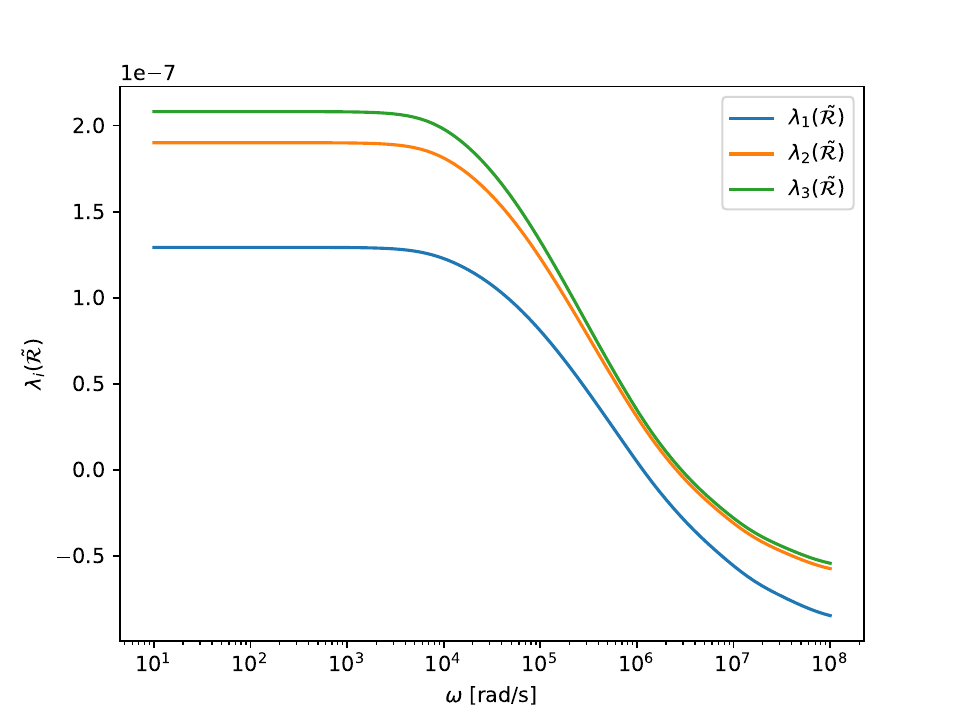} \\
(c)
\end{array}$
\caption{Homogeneous irregular polyhedron: Converged eigenvalue spectral signatures  $(a)$
$\lambda_i( {\mathcal{R}})$, $(b)$    $\lambda_i( \mathcal{I} )$
 and $(c)$ $\lambda_i( \tilde{\mathcal{R}} )$.}
\label{fig:eigenvalues_two_tetra}
\end{figure}

Furthermore, Figure~\ref{fig:comm_bounds_twotetra} shows that the computed {$\| {\mathcal Z} \| $, $ \| {\mathcal Z}^{(0) }\|$  and $ \| \tilde{\mathcal Z} \|$
}are significantly larger than the estimates 
$ \left \|  {\mathcal Z} - {\mathcal Z}^\varepsilon  \right \|\approx \Delta ( \| {\mathcal Z}\| ) $, $\left \| {\mathcal Z}^{(0)} - {\mathcal Z}^{(0),\varepsilon}  \right \| \approx \Delta ( \| {\mathcal Z}^{(0)} \| )$, and $ \left \| \tilde{\mathcal Z} - \tilde{\mathcal Z}^\varepsilon \right \| \approx \Delta ( \| \tilde{\mathcal Z}\| )$, {obtained by the error indicator in Appendix~\ref{sect:appendix}. Specifically, } (\ref{eqn:indiccommerr}), (\ref{eqn:indiccommerr0}) and (\ref{eqn:indiccommerrtil}) are applied for a discretisation with $p=3$ elements, which indicates that non-zero commutators are not just artefact caused by the regularisation and numerical discretisation and may contain additional object characterisation information.

\begin{figure}[H]
\centering
$\begin{array}{c}
\begin{array}{c c}
\includegraphics[width=0.45\textwidth]{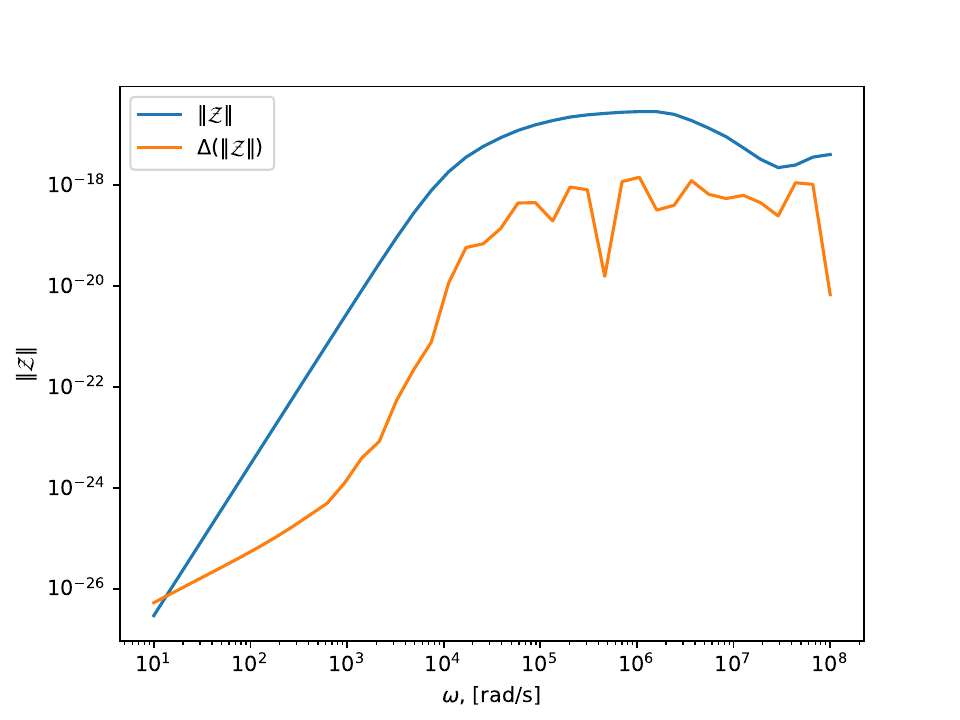} &
\includegraphics[width=0.45\textwidth]{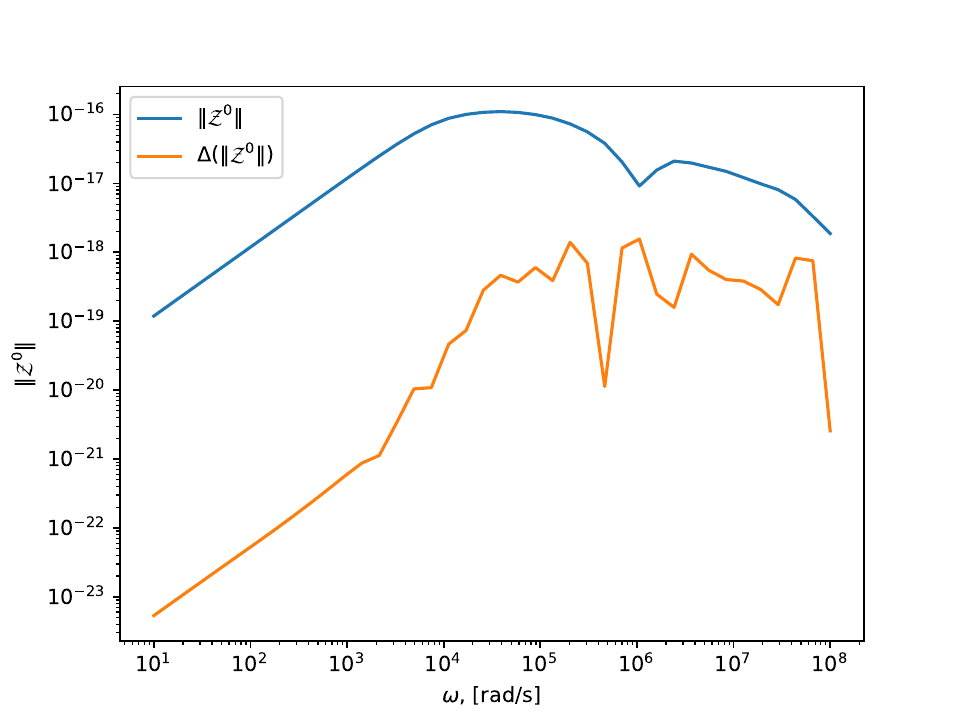} \\
(a) & (b)
\end{array}\\
\includegraphics[width=0.45\textwidth]{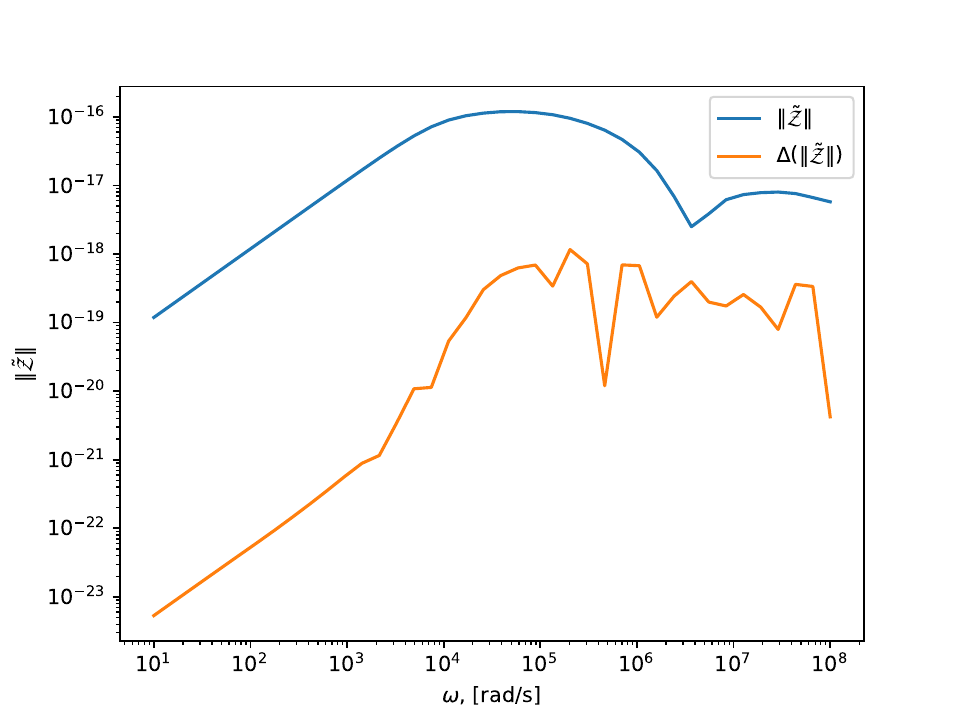}\\
(c)
\end{array}
$
\caption{Homogeneous irregular polyhedron:{ $(a)$  $\| {\mathcal Z} \|$,  $(b)$  $\| {\mathcal Z}^{(0)} \|$ and $(c)$ $\|\tilde{\mathcal Z} \|$} and the corresponding error indicators $(a)$ 
$ \left \|  {\mathcal Z} - {\mathcal Z}^\varepsilon  \right \|\approx \Delta ( \| {\mathcal Z}\| ) $, $(b)$ $\left \| {\mathcal Z}^{(0)} - {\mathcal Z}^{(0),\varepsilon}  \right \| \approx \Delta ( \| {\mathcal Z}^{(0)} \| )$, and $(c)$ $ \left \| \tilde{\mathcal Z} - \tilde{\mathcal Z}^\varepsilon \right \| \approx \Delta ( \| \tilde{\mathcal Z}\| )$.}
\label{fig:comm_bounds_twotetra}
\end{figure}

 \begin{figure}[H]
\centering
$\begin{array}{c c}
\includegraphics[width=0.45\textwidth]{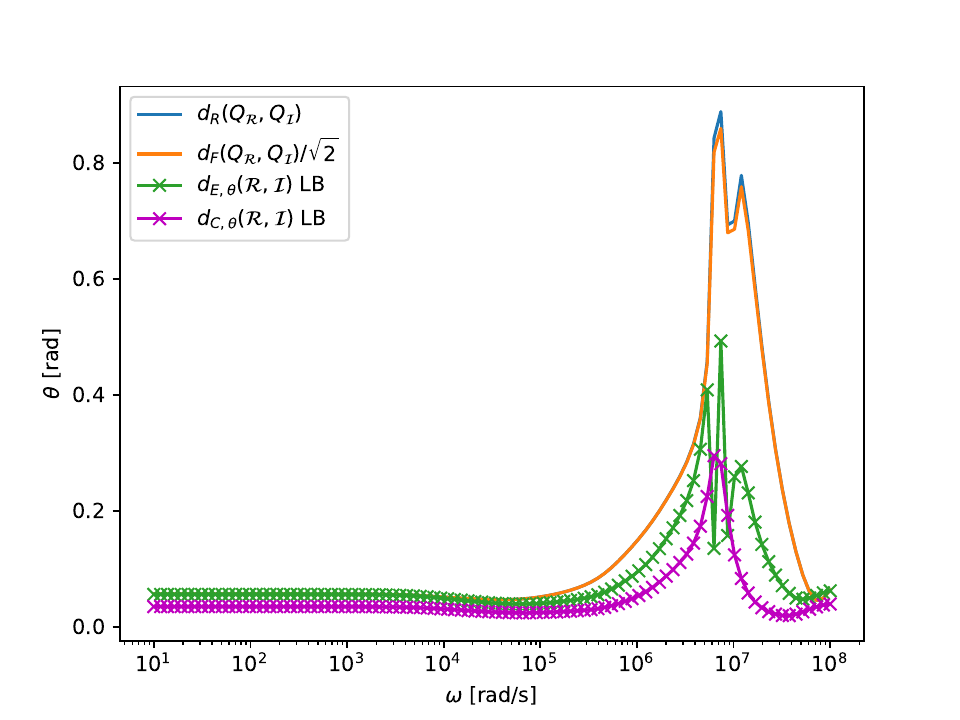} &
\includegraphics[width=0.45\textwidth]{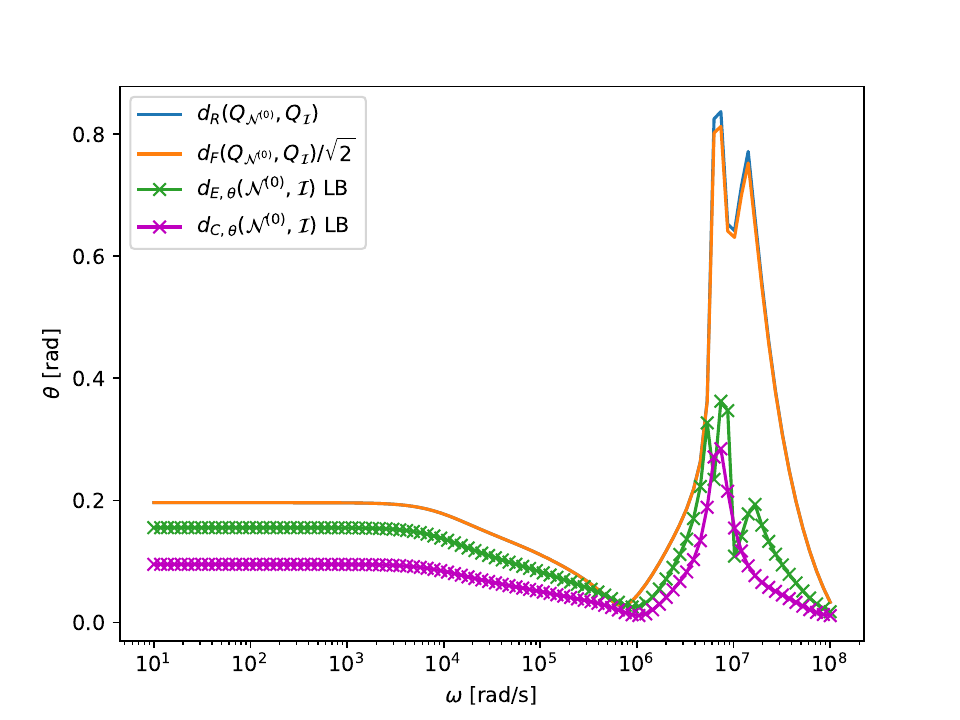}\\
(a) & (b) \end{array}$\\
$\begin{array}{c}
\includegraphics[width=0.45\textwidth]{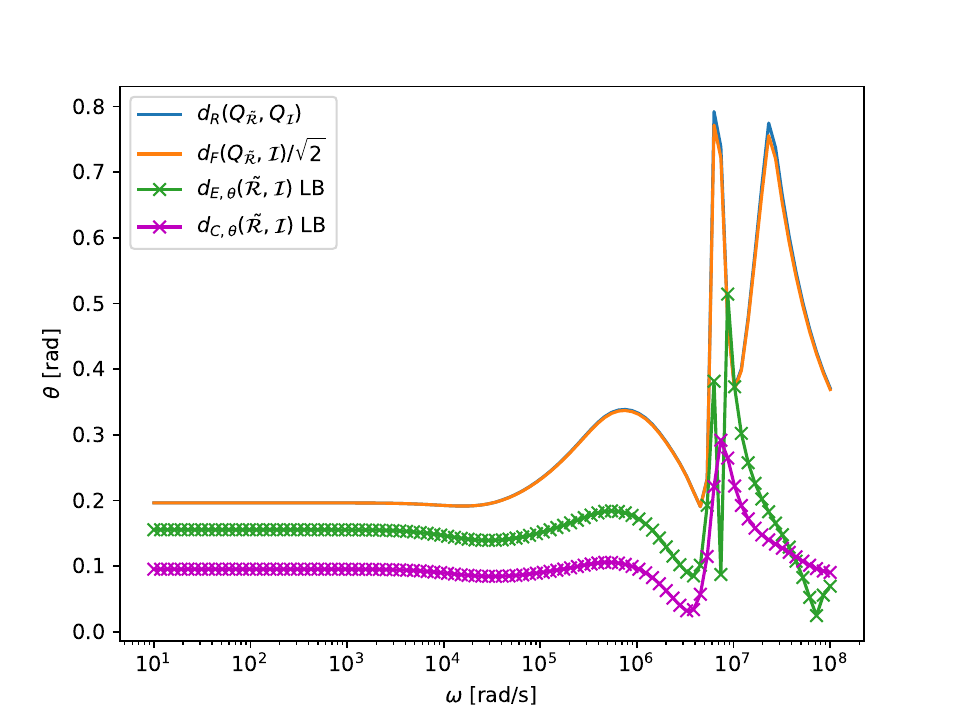} \\
(c)
\end{array}$
\caption{{Homogeneous irregular polyhedron: metric and approximate angle spectral signatures $(a)$  $d_R(Q_{\mathcal R},Q_{\mathcal I})$, $d_F(Q_{\mathcal R},Q_{\mathcal I})/\sqrt{2}$, 
$d_{E,\theta}({\mathcal R},{\mathcal I})$  and $d_{C,\theta}({\mathcal R},{\mathcal I})$,
 $(b)$ $d_R(Q_{{\mathcal N}^{(0)}},Q_{\mathcal I})$,
$d_F(Q_{{\mathcal N}^{(0)}} ,Q_{\mathcal I})$, $d_{E,\theta}({{\mathcal N}^{(0)}} ,{\mathcal I})$ and $d_{C,\theta}({{\mathcal N}^{(0)} },{\mathcal I})$, $(c) $  $d_R(Q_{\tilde{\mathcal R}},Q_{\mathcal I})$,
$d_F(Q_{\tilde{\mathcal R}},Q_{\mathcal I})$, $d_{E,\theta}(\tilde{\mathcal R} ,{\mathcal I})$ and $d_{C,\theta}(\tilde{\mathcal R} ,{\mathcal I})$.}}
\label{fig:estimated_metric_homo_TwoTetra}
\end{figure}

Having established that  variations in the computed $\| {\mathcal Z} \| $,  $ \| {\mathcal Z}^{(0)}\|$ and $ \| \tilde{\mathcal Z} \|$ are not purely due to effects or regularisation or discretisation,
Figure~\ref{fig:estimated_metric_homo_TwoTetra}~$(a)$ shows that the spectral signatures of the distances  { measured using the metrics $d_R  (Q_{\mathcal R},Q_{\mathcal I})$, $d_F(Q_{\mathcal R},Q_{\mathcal I}) /\sqrt{2}$  and the spectral signatures of the approximate distances  $d_{E,\theta} ({\mathcal R},{\mathcal I})$
and  $d_{C,\theta} ({\mathcal R},{\mathcal I})$  at different frequencies (excluding $\omega= 6.2 \times 10^6$ rad/s), which predict  non-zero distances between $Q_{\mathcal R}$ and $Q_{\mathcal I}$. Figures~\ref{fig:estimated_metric_homo_TwoTetra}~$(b)$ and ~\ref{fig:estimated_metric_homo_TwoTetra}~$(c)$ consider the spectral signatures of the same metrics for the pairs $(Q_{\mathcal N}^{(0)},Q_{\mathcal I})$  and $(Q_{\tilde{\mathcal R}},Q_{\mathcal I})$ and the corresponding approximate distances for the pairs
$({\mathcal N}^{(0)},{\mathcal I})$ and $({\mathcal R},{\mathcal I})$. These}
similarly predict non-zero distances between $Q_{{\mathcal N}^{(0) }}$ and $Q_{\mathcal I}$, and $Q_{\tilde{\mathcal R}}$ and $Q_{\mathcal I}$, respectively.
Each exhibit peaks in the predicted angles close to $\omega= 6.2 \times 10^6$ rad/s with the peaks over a broader  frequency range for  the $d_R$ and $d_F/\sqrt{2}$ metrics while those of $d_{E,\theta}$ and  $d_{C,\theta}$ are typically sharper.
 As already indicated in Figure~\ref{fig:eigenvalues_two_tetra}, the curves of $\lambda_i({\mathcal I})$ cross each other at  $\omega= 6.2 \times 10^6$ rad/s  and remain close for frequencies in this vicinity.

 {As motivated in Section~\ref{sect:basicexamples}, there will be issues in the numerical computation of $Q_{\mathcal R}$, $Q_{\tilde{\mathcal R}}$,  $Q_{\mathcal I}$ and $Q_{{\mathcal N}^0}$ when the corresponding eigenvalues become close and this  will impact on the accuracy of the predictions made by $d_R$ and $d_F$.  The earlier examples illustrated how inaccurate eigenvalue computations can lead to predictions of angles by $d_R$ and $d_F$ that are far from the smallest distance measure.  
 While it is again possible to use the  cross product approach to compute $Q_{\mathcal R}$, $Q_{\tilde{\mathcal R}}$,  $Q_{\mathcal I}$ and $Q_{{\mathcal N}^0}$  (instead of \texttt{numpy.linalg.eigh}) this comes with the  challenge of choosing suitable tolerances to switch between approaches once the eigenvalues become close. Instead, we prefer to use $d_{E,\theta}$ and  $d_{C,\theta}$, which do not experience these challenges and perform better when the eigenvalues become close.  
  }
 
  To further explain the behaviour of the predicted angles, the ellipsoids corresponding to $-x^T {\mathcal R} x=c_{\mathcal R}$ and $x^T {\mathcal I}x =c_{\mathcal I}$ at $\omega =22$  and $ 6.2\times 10^6$ rad/s, where $c_{\mathcal R}, c_{\mathcal I}$ are chosen such that the major axis of each of the ellipsoids is of unit length for each frequency,  are shown in Figure~\ref{fig:ellipsoids_two_tetra}. As well as showing both ellipsoids being superimposed on the same plot, the
  axis of rotation ${\vec k}$ (obtained from the non-zero coefficients of $K$) for the rotation angle is also shown. {For $\omega=22$ rad/s and the pair $( Q_{\mathcal R} ,Q_{\mathcal I} ) $, $d_R$ and $d_F/\sqrt{2}$ are small and  this implies  a small rotation about the indicated axis ${\bm k}$, similarly  $d_{E,\theta}$ and  $d_{C,\theta}$ for the pair $( {\mathcal R} ,{\mathcal I} ) $ are also small.
For   $\omega=6.2\times 10^6\pm \epsilon$ rad/s and small $\epsilon$, the imaginary eigenvalues are distinct, but close, and so the ellipsoid representing  $x^T {\mathcal I}x =c_{\mathcal I}$ reduces to being close to a spheroid. In the limiting case, when   the representation of $x^T {\mathcal I}x =c_{\mathcal I}$ is a spheroid, it }can be rotated through any angle about the axis ${\vec k}$ without changing the configuration relative to the ellipsoidal representation of $-x^T {\mathcal R}x =c_{\mathcal R}$. {The smallest angle, and hence the smallest distance measure (using  $d_R $ , $d_F/\sqrt{2}$, $d_{E,\theta}$ for the pair $( Q_{\mathcal R} ,Q_{\mathcal I} )$ and  $d_{C,\theta}$ for the pair $( {\mathcal R} ,{\mathcal I} )$),
{  should be close to zero for frequencies close to  $\omega=6.2\times 10^6  $ rad/s.} 
Furthermore, as equality of  the eigenvalues is approached, it is possible that the minimum distance  may increase, then dip to zero at eigenvalue equality and increasing again, which is what the results of Figure~\ref{fig:estimated_metric_homo_TwoTetra} with  $d_{R,\theta}$ and $d_{C,\theta}$ shows.}
  
  \begin{figure}[H]
\centering
$\begin{array}{cc}
\includegraphics[width=0.4\textwidth]{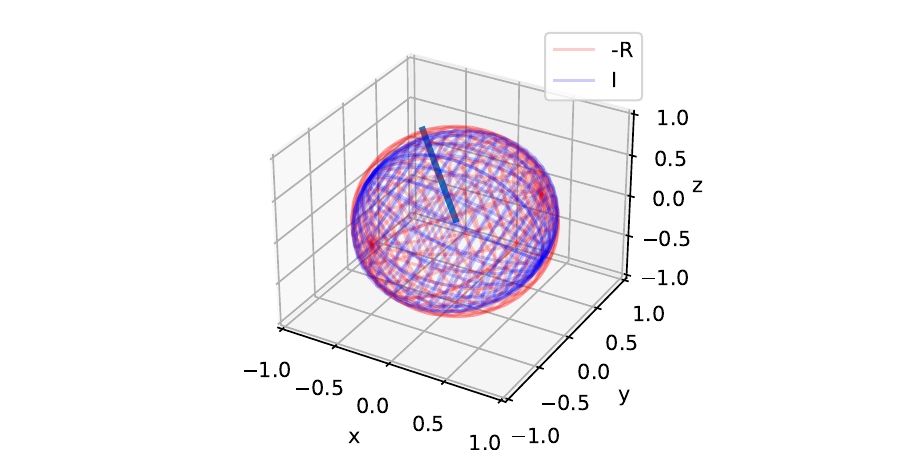} &
\includegraphics[width=0.4\textwidth]{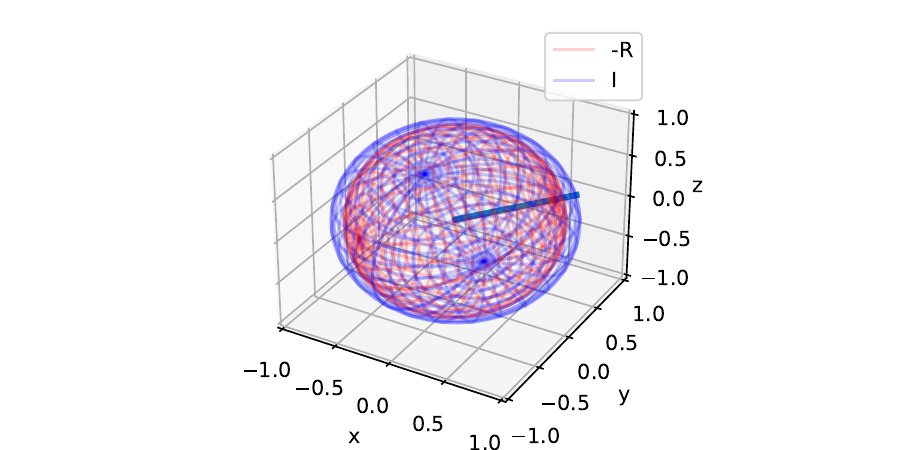}\\
(a) & (b)
\end{array}$
\caption{Homogeneous irregular polyhedron: Ellipsoidal representations of  $-x^T {\mathcal R} x=c_{\mathcal R}$ and $x^T {\mathcal I}x =c_{\mathcal I}$ at $\omega =22$ and $ 6.2 \times 10^6 $ rad/s  with $c_{\mathcal R}, c_{\mathcal I}$ chosen such that the major axis  is of unit length for both ellipsoids and each frequency.
}
\label{fig:ellipsoids_two_tetra}
  \end{figure}

\subsection{Toy gun} \label{sect:toygun}

A simplified model of a toy gun is considered as shown in Figure~\ref{fig:toygungeom}. The barrel is hollow with a cap at one end.  The length of the barrel is 0.2 m, outer and inner radii of the barrel are 0.02 m and 0.01m, respectively. The box representing the receiver is 0.08 m $\times$ 0.01 m $\times$ 0.15 m. The barrel is chosen to have  $\sigma_* = 1.45\times 10^6$ S/m and  $\mu_r=20, 100 $ is considered in turn  (note that  $\sigma_* = 1.45\times 10^6$ S/m, $\mu_r=100$ corresponds to carbon steel) and the receiver has $4.5 \times 10^6$ S/m and $\mu_r=1$, which represents a stainless steel material.  A mesh of 21,132 unstructured tetrahedra, 5,525 prisms is generated and on this mesh elements of order $p=0,1,\ldots,5$ are uniformly increased until mesh convergence. The prisms are chosen to be used as boundary layers in the barrel and the thickness of the prismatic layer is varied according to the skin depth for the barrel and a target frequency of $1 \times 10^6$ rad/s~\cite{Elgy2024_preprint}.

\begin{figure}[h]
\begin{center}
\includegraphics[width=0.25\textwidth]{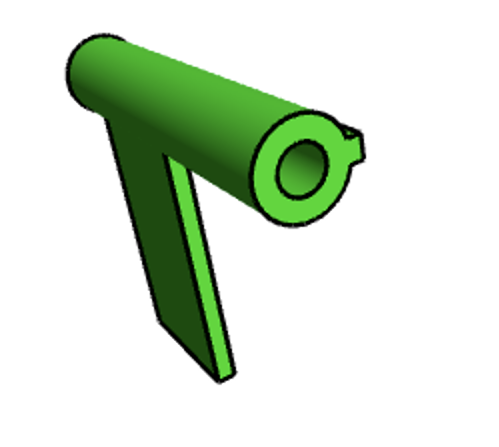}
\end{center}
\caption{Illustration of the toy gun: Barrel has $\sigma_* = 1.45\times 10^6$ S/m and $\mu_r=20,100$ and the box representing the receiver is a non-magnetic stainless steel with $4.5 \times 10^6$ S/m and $\mu_r=1$.} \label{fig:toygungeom}
\end{figure}

\subsubsection{Case $\mu_r=20$}
Figure~\ref{fig:toygunmur20} shows the converged eigenvalue spectral signatures $\lambda_i({\mathcal R})$, $\lambda_i({\mathcal I})$ and $\lambda_i(\tilde{\mathcal R})$ for the toy gun example when the barrel has a relative magnetic permeability of $\mu_r=20$. Notice that $\lambda_2({\mathcal I})= \lambda_3({\mathcal I})$ at $\omega=3.7 \times 10^3$ rad/s,  $\lambda_1({\mathcal I})=\lambda_2({\mathcal I})$ at $\omega=5.8 \times10^4$ rad/s
while $\lambda_1(\tilde{\mathcal R}) = \lambda_2(\tilde{\mathcal R})$ at $\omega=1.5\times10^3$ rad/s, { which are subsequently excluded from the MPT spectral signature.}
 Figure~\ref{fig:toygunmur20metrics}~$(a)$ shows the  {spectral signature of the distance measured using the  metrics $d_R$, $d_F/\sqrt{2}$ for the pairs  $(Q_{\mathcal R},Q_{\mathcal I})$ 
  and the approximate distances  $d_{E,\theta}$ and  $d_{C,\theta}$ 
   for the pairs $({\mathcal R},{\mathcal I})$ and Figure~\ref{fig:toygunmur20metrics}~$(b)$ the corresponding results for the pairs
    $(Q_{\tilde{\mathcal R}},Q_{\mathcal I})$ and  $(\tilde{\mathcal R},{\mathcal I})$, respectively. }
 
The sharp peaks observed in Figure~\ref{fig:toygunmur20metrics}~$(a)$ near to $\omega=3.7 \times 10^3$ and $ 5.8\times 10^4$ rad/s for {$d_R$,  $d_F/\sqrt{2}$  for the pair $(Q_{\mathcal R},Q_{\mathcal I})$  (the curve for $d_F/\sqrt{2}$ is indistinguishable from $d_R$ on this scale)
 contrast to the much better behaved $d_{E,\theta}$ for the pair $({\mathcal R},{\mathcal I})$. These frequencies correspond to those at which the curves for $\lambda_i({\mathcal I})$ cross each other (the $\lambda_i({\mathcal R})$ curves do not cross).  In the case of $\omega=3.7 \times 10^3$ rad/s, where the $\lambda_i({\mathcal I})$ curves cross only momentarily and do not remain close,  a sharp peak, with smaller magnitude, is exhibited for $d_R$,  $d_F/\sqrt{2}$, but in the neighbourhood of $\omega= 5.8\times 10^4$ rad/s the curves remain close and $d_R$,  $d_F/ \sqrt{2} $ exhibit wider peaks with larger magnitude. Note that by choosing to evaluate these signatures over a finer set of discrete frequencies would increase the magnitude of $d_R$,  $d_F/\sqrt{2}$ close to $\omega=3.7 \times 10^3$ rad/s.  The additional peaks observed in Figure~\ref{fig:toygunmur20metrics}~$(b)$ for $d_R$,  $d_F/\sqrt{2}$ 
for the pair $(Q_{\tilde{\mathcal R}},Q_{\mathcal I})$
 and  $d_{E,\theta}$ for the pair $(\tilde{\mathcal R},{\mathcal I})$  close to $\omega=1.5 \times 10^3 $ rad/s, corresponds to the location where $\lambda_i(\tilde{\mathcal R})$ cross}. On this scale, there are no noticeable peaks in $d_{C,\theta} $ for either pairs  $({\mathcal R},{\mathcal I})$ or  $(\tilde{\mathcal R},{\mathcal I})$.
\begin{figure}[h]
\begin{center}
$\begin{array}{cc}
\includegraphics[width=0.45\textwidth]{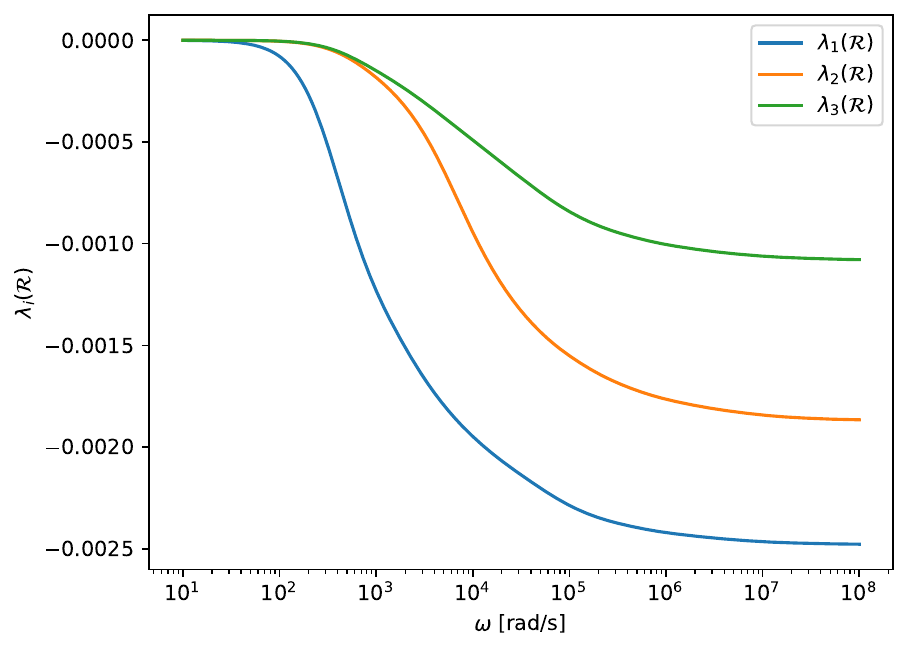} &
\includegraphics[width=0.45\textwidth]{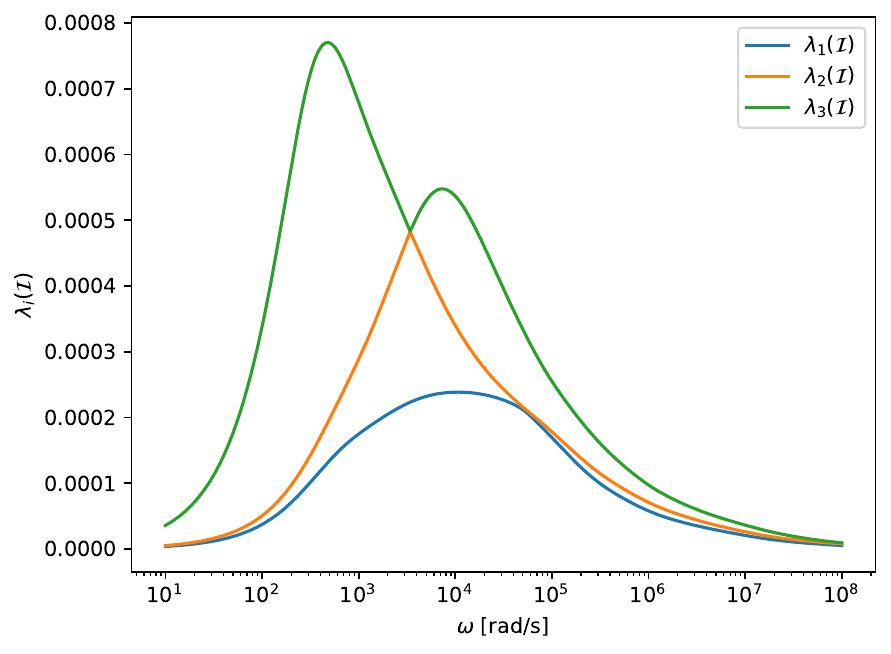} \\
(a) & (b) \\
\end{array}$\\
$\begin{array}{c}
\includegraphics[width=0.49\textwidth]{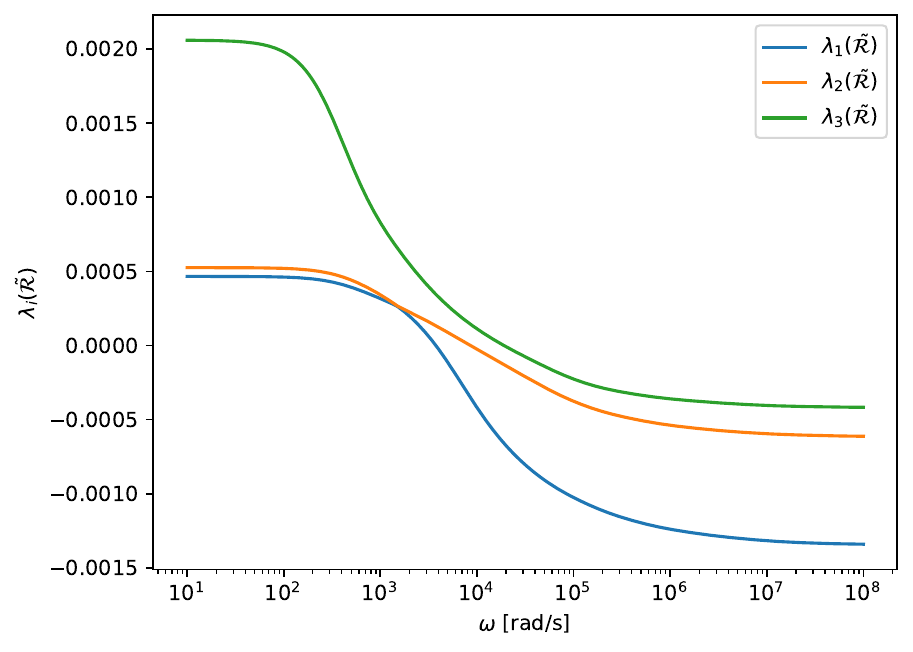} \\
(c) \end{array}$
\end{center}
\caption{Toy gun with $\mu_r=20$ in the barrel: 
Computed  eigenvalue spectral signatures $(a)$ $\lambda_i({\mathcal R}) $, $(b)$ $\lambda_i({\mathcal I}) $ and $(c)$ $\lambda_i(\tilde{\mathcal R})$.} \label{fig:toygunmur20}
\end{figure}

\begin{figure}[h]
\begin{center}
$\begin{array}{cc}
\includegraphics[width=0.5\textwidth]{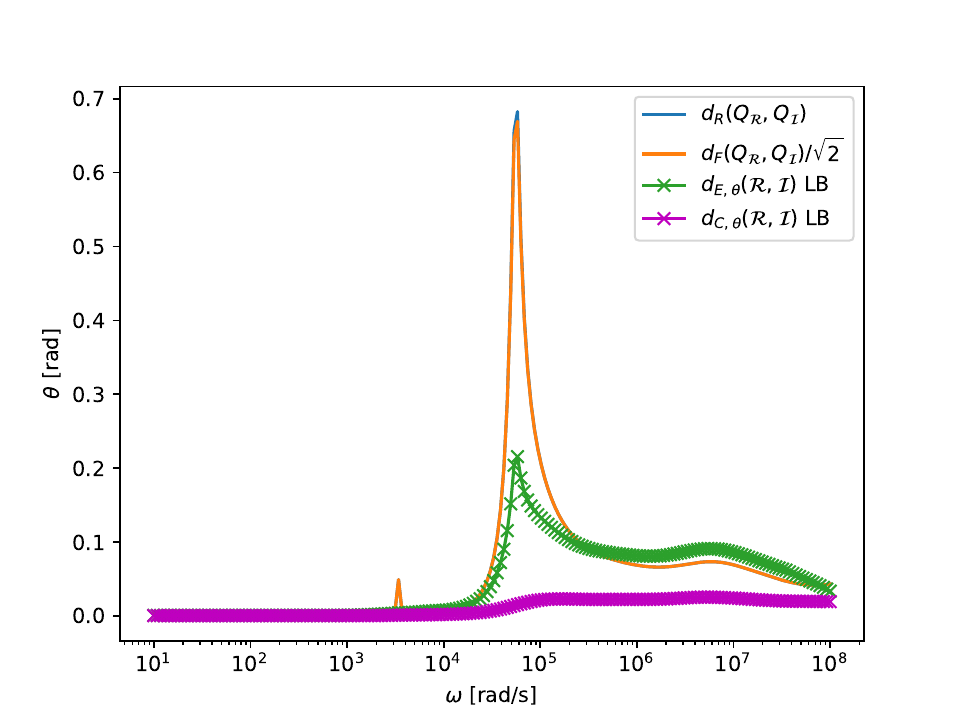}  &
\includegraphics[width=0.5\textwidth]{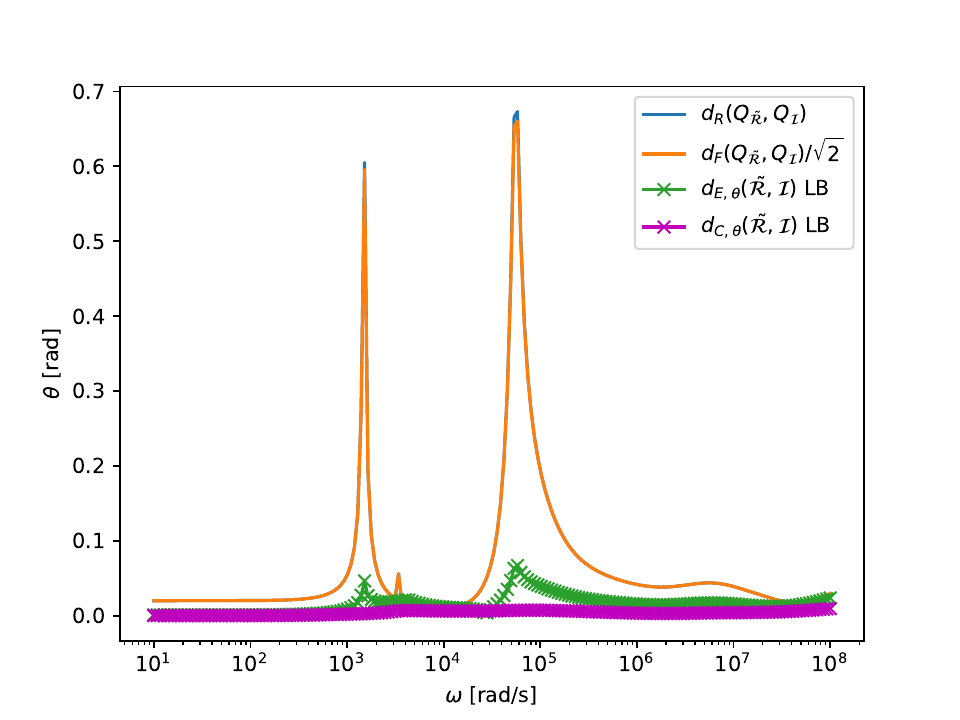} \\ (a) &(b) 
\end{array}$
\end{center}
\caption{{Toy gun with $\mu_r=20$ in the barrel: metric and approximate angle spectral signatures
$(a)$ $d_R( Q_{\mathcal R}, Q_{\mathcal I})$, $d_F(  Q_{\mathcal R}, Q_{\mathcal I})/\sqrt{2}$, $d_{E,\theta}( {\mathcal R}, {\mathcal I}) $ and $d_{C,\theta}( {\mathcal R} , {\mathcal I} ) $ and $(b)$  $d_R(Q_{ \tilde{\mathcal R}}, Q_{\mathcal I})$, $d_F( Q_{\tilde{\mathcal R}}, Q_{\mathcal I})/\sqrt{2}$, $d_{E,\theta}( \tilde{\mathcal R}, {\mathcal I}) $ and $d_{C,\theta}( \tilde{\mathcal R} , {\mathcal I} ) $}. }\label{fig:toygunmur20metrics}
\end{figure}

\subsubsection{Case $\mu_r=100$}

Figure~\ref{fig:toygunmur100} shows the converged eigenvalue spectral signatures $\lambda_i({\mathcal R})$, $\lambda_i({\mathcal I})$ and $\lambda_i(\tilde{\mathcal R})$ for the toy gun example when the barrel has a relative magnetic permeability of $\mu_r=100$. Notice that for this case, $\lambda_2({\mathcal I})= \lambda_3({\mathcal I})$ at $\omega=2.4\times 10^4$ rad/s,  $\lambda_1({\mathcal I})=\lambda_2({\mathcal I})$ at $\omega=3.3\times10^5$ rad/s
while $\lambda_1(\tilde{\mathcal R}) = \lambda_2(\tilde{\mathcal R})$ at $\omega=1.83 \times 10^3$ rad/s, {which are excluded.}
 Figure~\ref{fig:toygunmur100metrics}~$(a)$ shows the {spectral signatures of distances measured by the metrics $d_R$, $d_F/\sqrt{2}$ for the pair $(Q_{\mathcal R},Q_{\mathcal I})$
 and the approximate distances $d_{E,\theta}$ and  $d_{C,\theta}$  for the pair
  $({\mathcal R},{\mathcal I})$ and Figure~\ref{fig:toygunmur100metrics}~$(b)$ the corresponding results for the pairs $(Q_{\tilde{\mathcal R}},Q_{\mathcal I})$ and $(\tilde{\mathcal R},{\mathcal I})$, respectively. }

\begin{figure}[h]
\begin{center}
$\begin{array}{cc}
\includegraphics[width=0.5\textwidth]{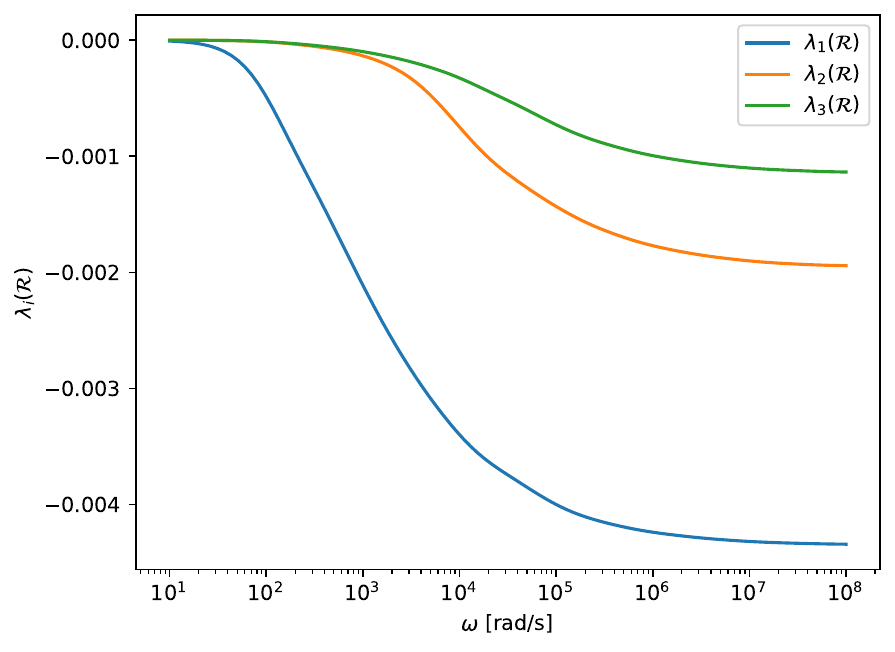} &
\includegraphics[width=0.5\textwidth]{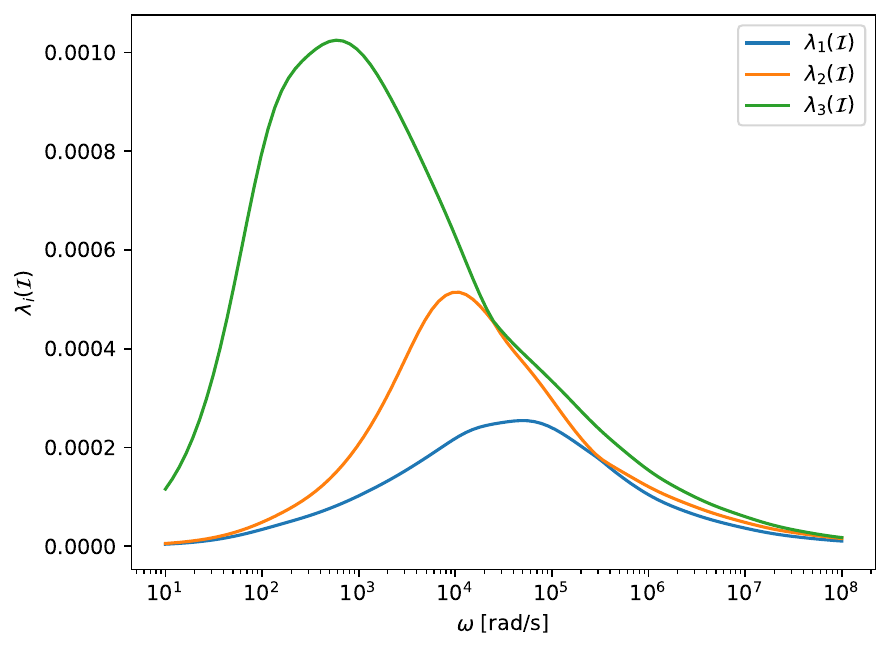} \\
(a) & (b) \\
\end{array}$\\
$\begin{array}{c}
\includegraphics[width=0.5\textwidth]{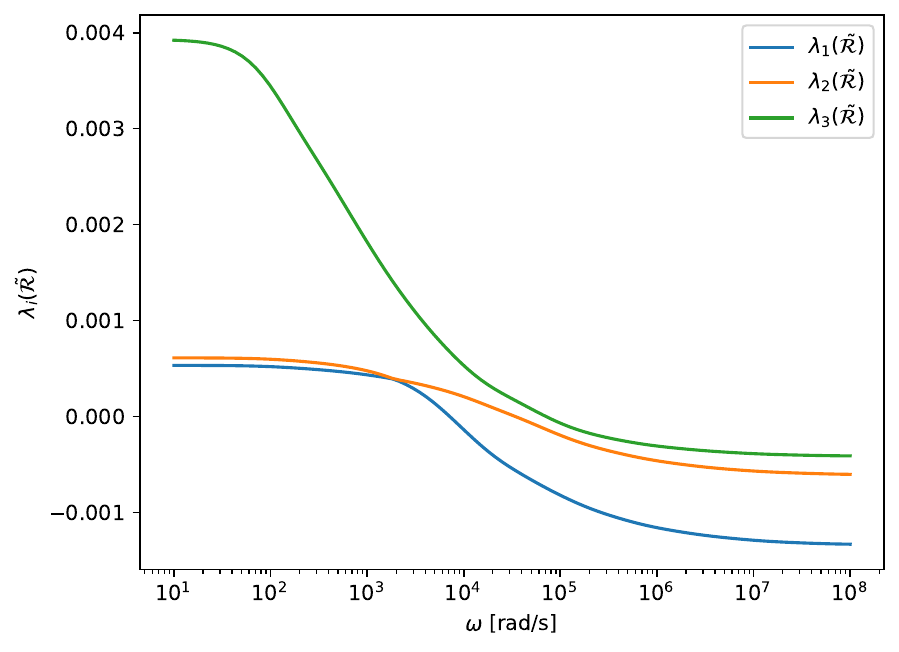} \\
(c) \end{array}$
\end{center}
\caption{Toy gun with $\mu_r=100$ in the barrel: 
Computed  eigenvalue spectral signatures $(a)$ $\lambda_i({\mathcal R}) $, $(b)$ $\lambda_i({\mathcal I}) $ and $(c)$ $\lambda_i(\tilde{\mathcal R})$.} \label{fig:toygunmur100}
\end{figure}

\begin{figure}[h]
\begin{center}
$\begin{array}{cc}
\includegraphics[width=0.5\textwidth]{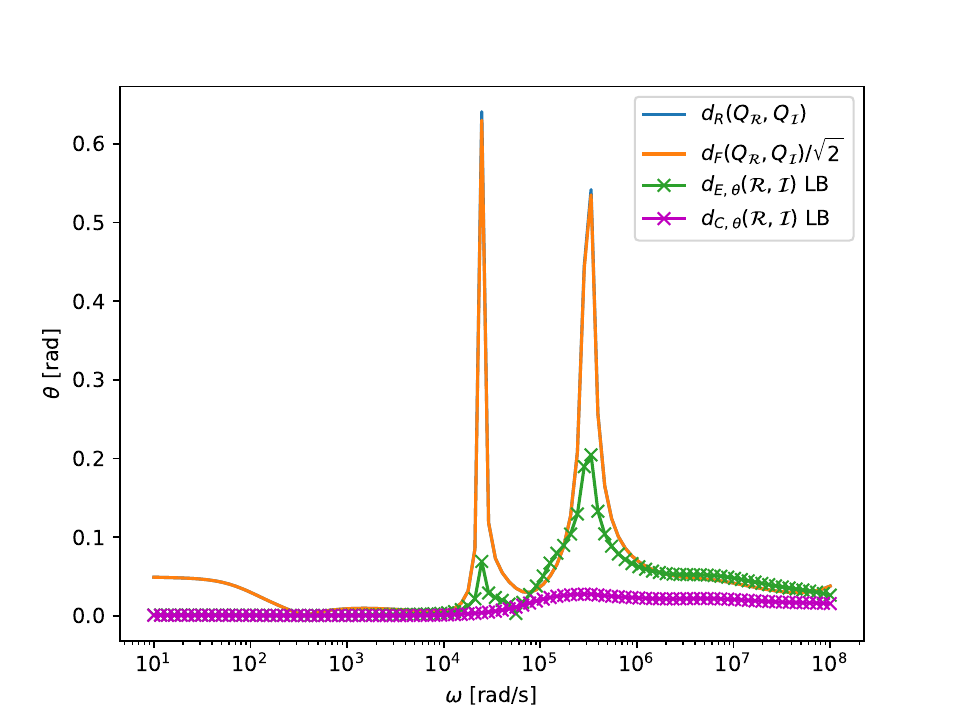}  &
\includegraphics[width=0.5\textwidth]{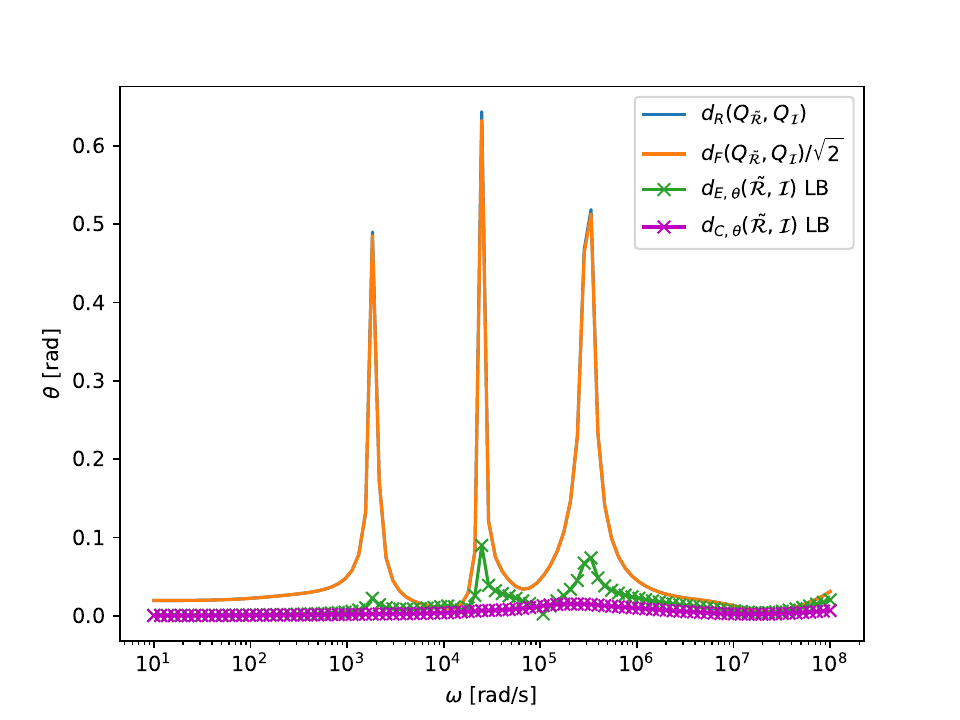} \\ (a) &(b) 
\end{array}$
\end{center}
\caption{{Toy gun with $\mu_r=100$ in the barrel: metric and approximate angle spectral signatures
$(a)$ $d_R( Q_{\mathcal R}, Q_{\mathcal I})$, $d_F(  Q_{\mathcal R},Q_{\mathcal I})/\sqrt{2}$, $d_{E,\theta}( {\mathcal R}, {\mathcal I}) $ and $d_{C,\theta}( {\mathcal R} , {\mathcal I} ) $ and $(b)$  $d_R(Q_{ \tilde{\mathcal R}}, Q_{\mathcal I})$, $d_F( Q_{\tilde{\mathcal R}}, Q_{\mathcal I})/\sqrt{2}$, $d_{E,\theta}( \tilde{\mathcal R}, {\mathcal I}) $ and $d_{C,\theta}( \tilde{\mathcal R} , {\mathcal I} ) $.}} \label{fig:toygunmur100metrics}
\end{figure}

For practical measurements, which are subject to noise, the new measures  $d_{E,\theta}$ and  $d_{C,\theta}$  we propose offer additional advantages. From the results of Section~\ref{sect:sensmetrics}, we know that the computation of eigenvectors will not only deteriorate with the noise, but also with the proximity of the eigenvalues, which are also polluted in the case of noise. The new measures, on the other hand, do not suffer from this drawback and hence offer greater potential in object classification using measured data.

\subsection{An application to object classification}

The final example considers an application  {of our  approximate measures of  distance applied to the MPT spectral signature at discrete frequencies}  as additional features for object classification.

\subsubsection{Dataset}
In our previous work, we have described an approach to object classification based on tensor invariants of the $\tilde{\mathcal R}$ and ${\mathcal I}$ evaluated at discrete frequencies of the spectral signature~\cite{ledgerwilsonlion2022}. The {spectral signature of the approximate distance $d_{E,\theta}(\tilde{\mathcal R},{\mathcal I})$, evaluated at discrete frequencies, } are proposed as  additional features that can augment these tensor invariants.
We follow the same notation as~\cite{ledgerwilsonlion2022}, where dictionaries of the form
\begin{align}
{\mathcal D} = \{ ({\rm x}_1, {\rm t}_1), ({\rm x}_2, {\rm t}_2, \ldots, ({\rm x}_P, {\rm t}_P)\} ,
\end{align}
with $P$ data pairs are constructed. In the above, the pairs $({\rm x},{\rm t})$ correspond to the feature vector ${\rm x}\in {\mathbb R}^F$ and the class encoding ${\rm t}\in{\mathbb R}^K$, where $F$ is the number of features and $K$ the number of classes~\footnote{{The definitions of $K$, $C$ and $M$ in this section are different to those in Section~\ref{sect:metrics}, but, given the context, no confusion should arise.}}. Given $K$ discrete classes, $C_1, C_2., \ldots, C_K$, then if the correct class is $C_k$, the class encoding take the form
\begin{align}
({\rm t})_i = \left \{ \begin{array}{ll}
1 & \text{if $i=k$}, \\
0 & \text{otherwise} \end{array} \right . .
\end{align}
However, we now consider extended feature vectors of length $F=7M$ of the form
\begin{align}
({\rm x})_i = \left \{ \begin{array}{ll}
I_j ( \tilde{\mathcal R}(\alpha B, \omega_m, \sigma_*, \mu_r)), & i =j+(m-1)M, \\
I_j ( {\mathcal I}(\alpha B, \omega_m, \sigma_*, \mu_r)), & i =j+(m+2)M, \\
d_{E,\theta}(\tilde{\mathcal R}(\alpha B, \omega_m, \sigma_*, \mu_r)), {\mathcal I}(\alpha B, \omega_m, \sigma_*, \mu_r)))), & i =1 + (m+5)M,
\end{array} \right .
\end{align}
where $\omega_m$, $m=1,\ldots,M$ denote the discrete frequencies for which the MPT signature is evaluated/measured and $I_j=1,2,3$ denote the principal tensor invariants, as defined in~\cite{ledgerwilsonlion2022}.

An improved, and extended, dataset of object characterisations (which importantly now includes magnetic objects) has been established through our latest open release~\cite{mpt-library}.  Taken from this dataset, Figure~\ref{fig:objectmeshes} illustrates surface meshes for a range of different object geometries including: British coins, pendants, knives, screwdrivers and toy gun models (which including vairations in shape, size and number of barrels of that presented in Section~\ref{sect:toygun}). The British coins consist of a mixture of homogeneous objects and non-homogenous objects. Notably, the recent one pence (1p), two pence (2p), one pound (\pounds 1) and two pound (\pounds 2) are magnetic and inhomogeneous with the newer  denominations
of the 1p and 2p coins being introduced due to the cost of the original solid copper coin exceeding its face value~\cite{Elgy2022_preprintB}. While not all coins are circular, they do have mirror symmetries so that their MPT characterisation of the objects in their canonical configuration is diagonal.  
The homogeneous pendants are all constructed of non-magnetic materials (including gold, silver and platinum) and each have two or more mirror symmetries so that their MPT characterisation in the canonical configuration is diagonal. 
Inhomogeneous kitchen knife models that have blades made of both magnetic and non-magnetic materials are considered and, if appropriate, include copper rivets. These models have only one mirror symmetry and their MPT characterisation is not diagonal in the canonical configuration. Finally, inhomogeneous toy gun models with a single or a double barrel made of a magnetic steel and receivers with non-magnetic material are considered. These models have been constructed so as to not have mirror symmetries. Given the smaller skin depths associated with magnetic materials, prismatic boundary layer elements are included for these objects. Full details of each of the objects is provided in~\cite{mpt-library}, with additional variations in object size and object conductivity incorporated using scaling results~\cite{Wilson2021} and the addition of 40 dB noise, which  require only negligible
 computation. 

\begin{figure}
\vspace{-0.1in}
\begin{center}
$\begin{array}{ccccccccc}
\includegraphics[width=0.0325\textwidth, keepaspectratio]{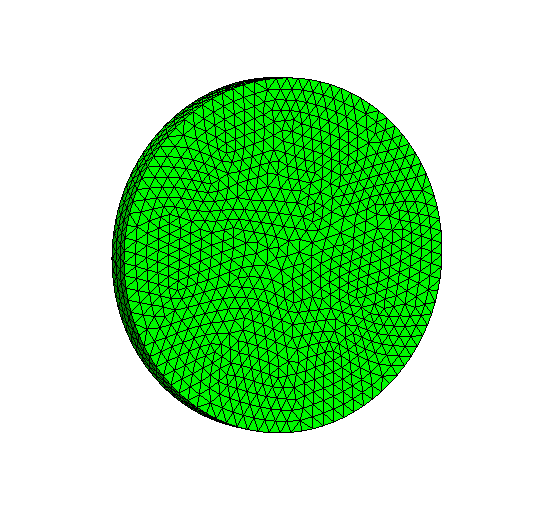} &
\includegraphics[width=0.045\textwidth, keepaspectratio]{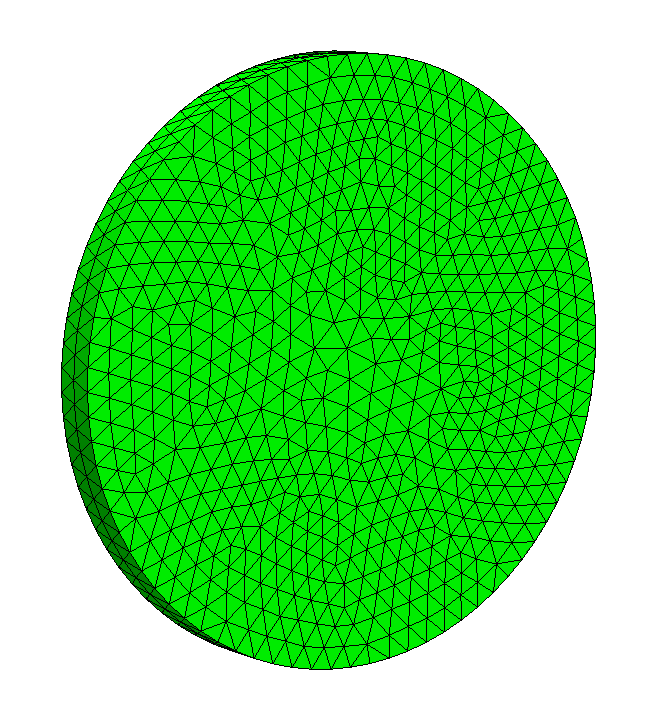} &
\includegraphics[width=0.05\textwidth, keepaspectratio]{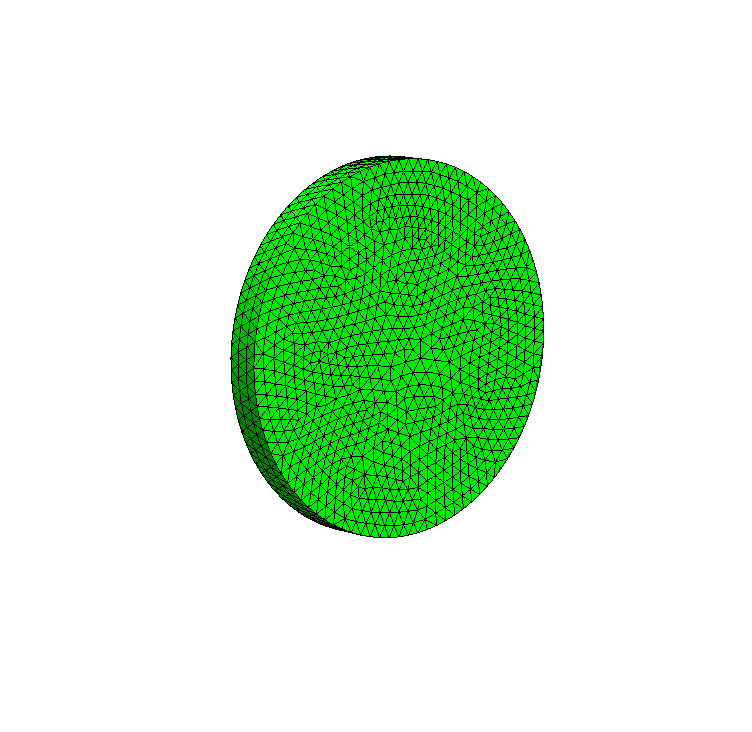} &
\includegraphics[width=0.045\textwidth, keepaspectratio]{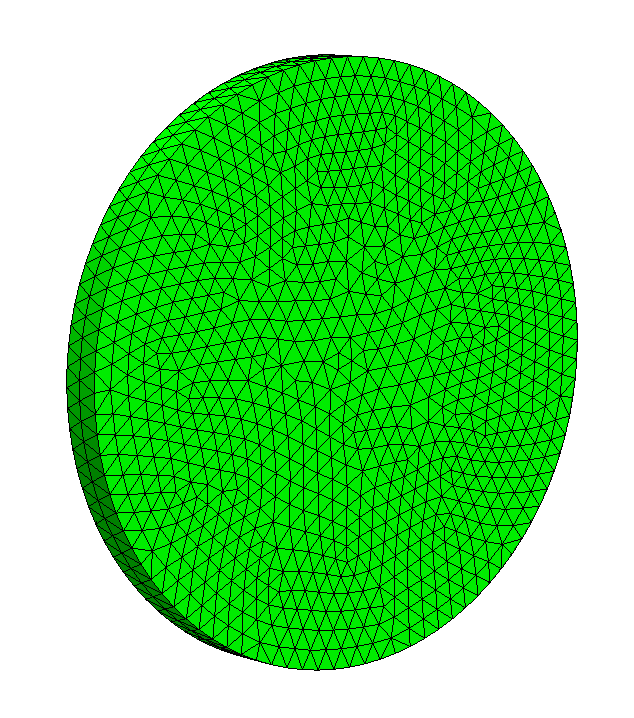} &
\includegraphics[width=0.0325\textwidth, keepaspectratio]{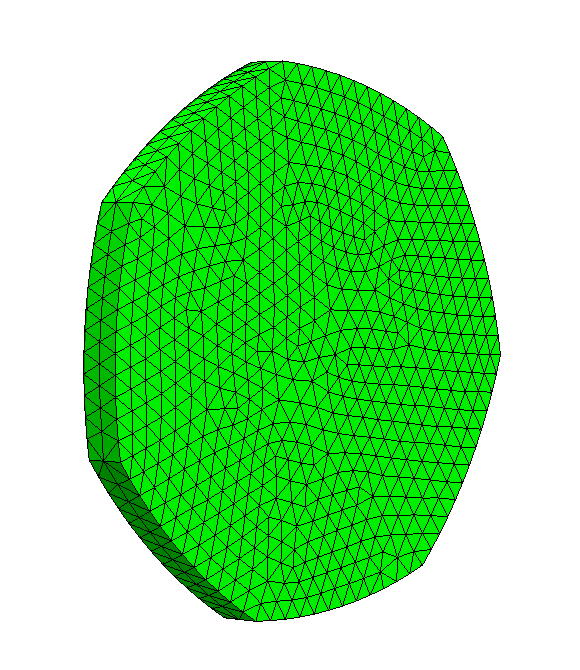} &
\includegraphics[width=0.075\textwidth, keepaspectratio]{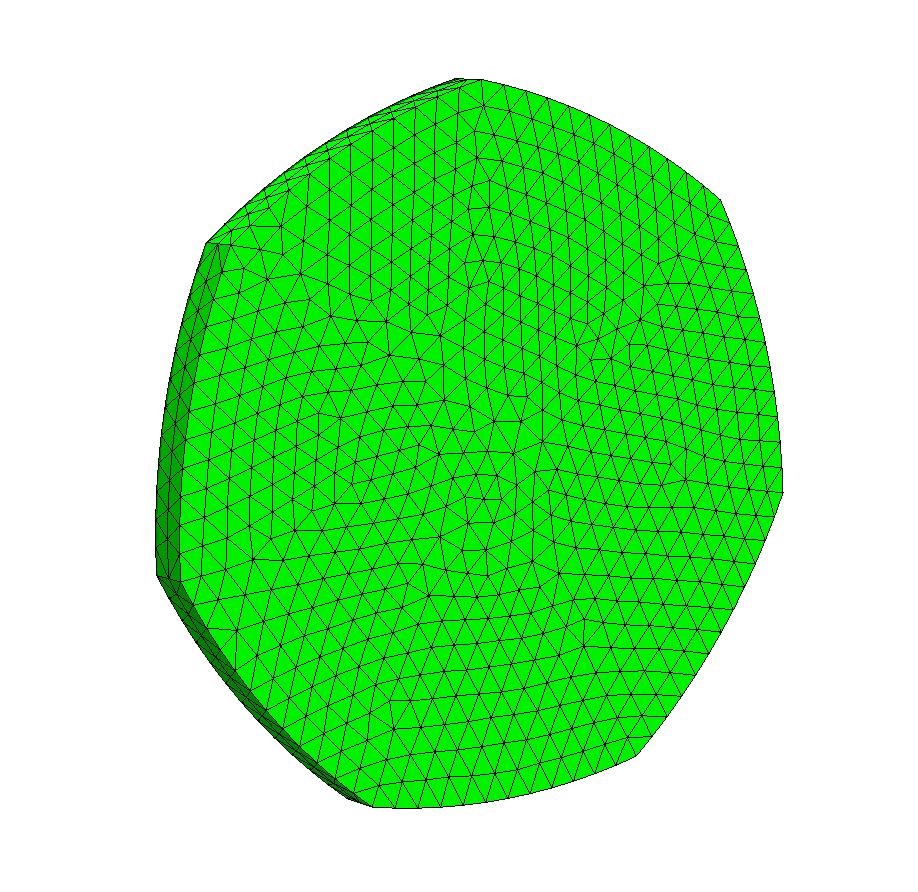} &
\includegraphics[width=0.075\textwidth, keepaspectratio]{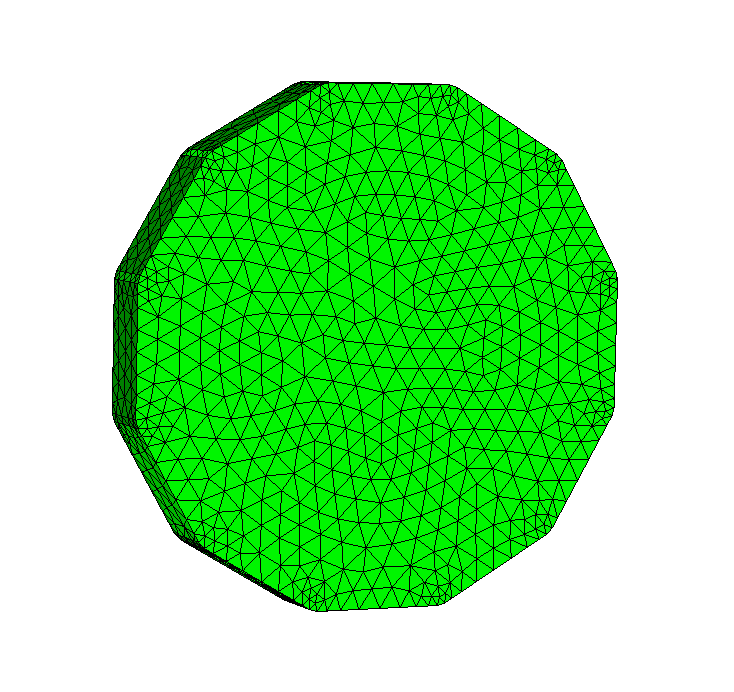} &
\includegraphics[width=0.075\textwidth, keepaspectratio]{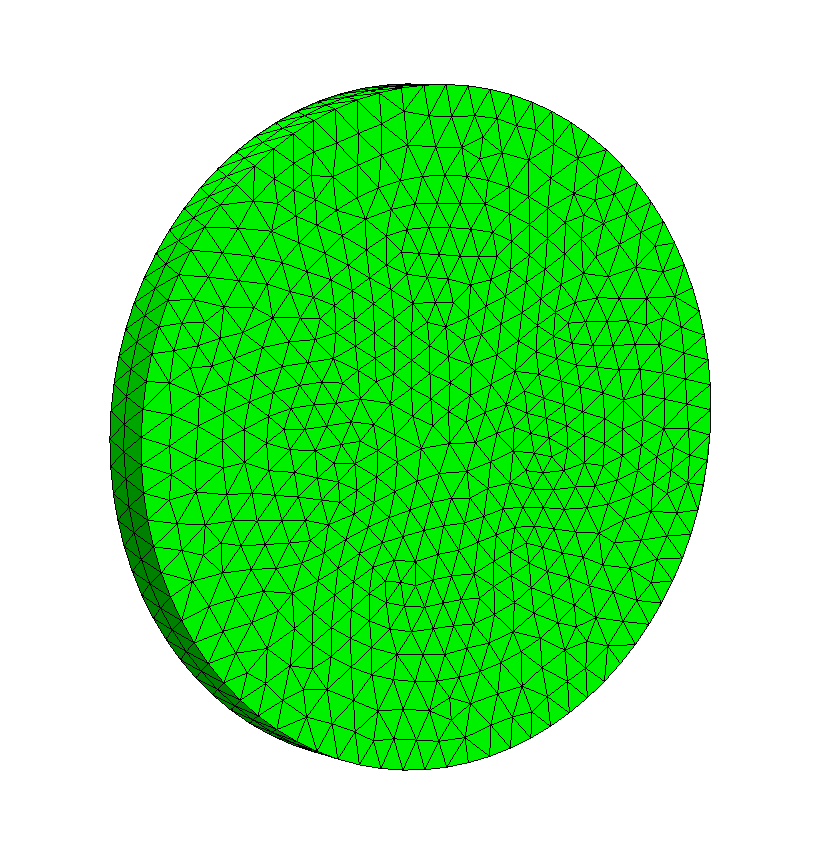} &
\includegraphics[width=0.075\textwidth, keepaspectratio]{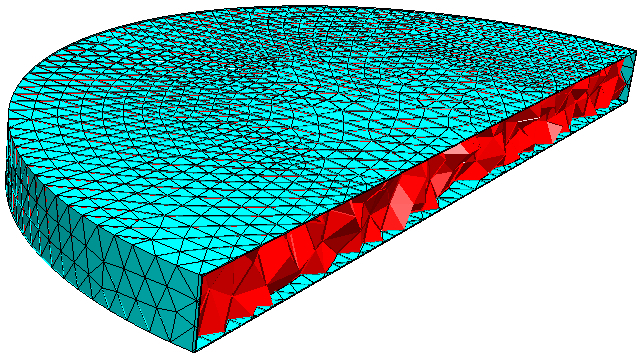} 
\end{array}$
\end{center}
\begin{center}
$(a)$ Selection of British coins (with layered materials modelled using prismatic layers).
\end{center}
\begin{center}
$\begin{array}{ccccccc}
\includegraphics[width=0.04\textwidth, keepaspectratio]{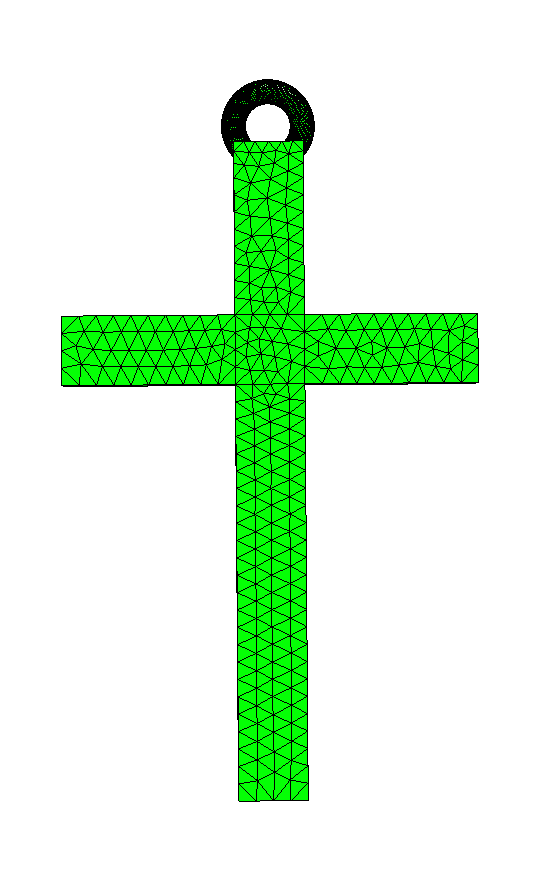}&
\includegraphics[width=0.04\textwidth, keepaspectratio]{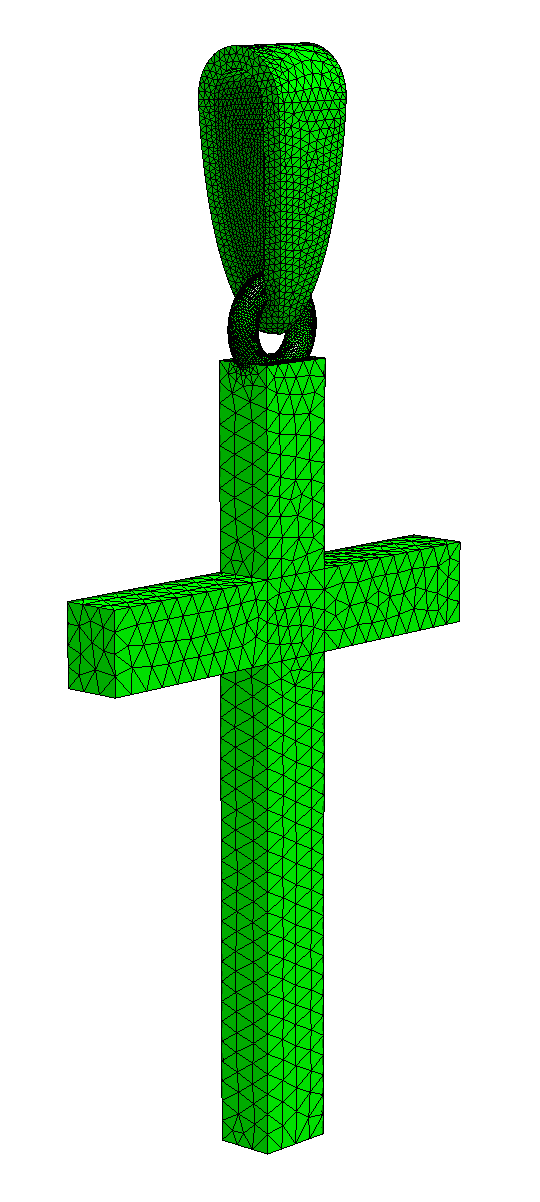} &
\includegraphics[width=0.04\textwidth, keepaspectratio]{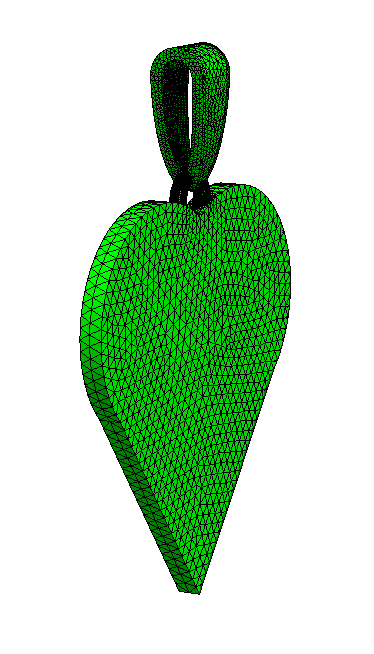} &
\includegraphics[width=0.04\textwidth, keepaspectratio]{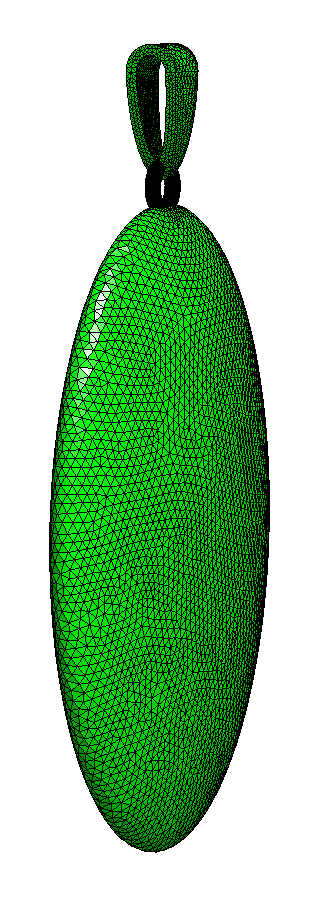} &
\includegraphics[width=0.04\textwidth, keepaspectratio]{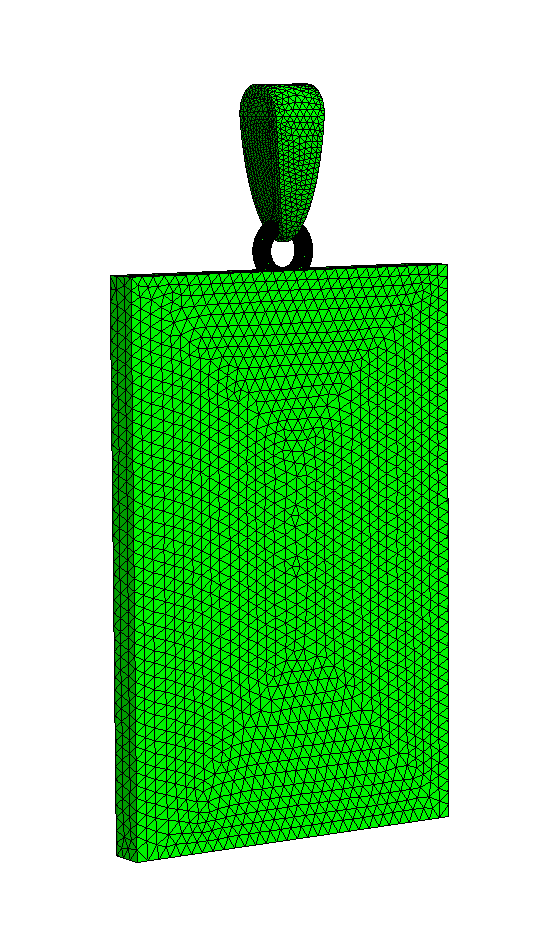} &
\includegraphics[width=0.04\textwidth, keepaspectratio]{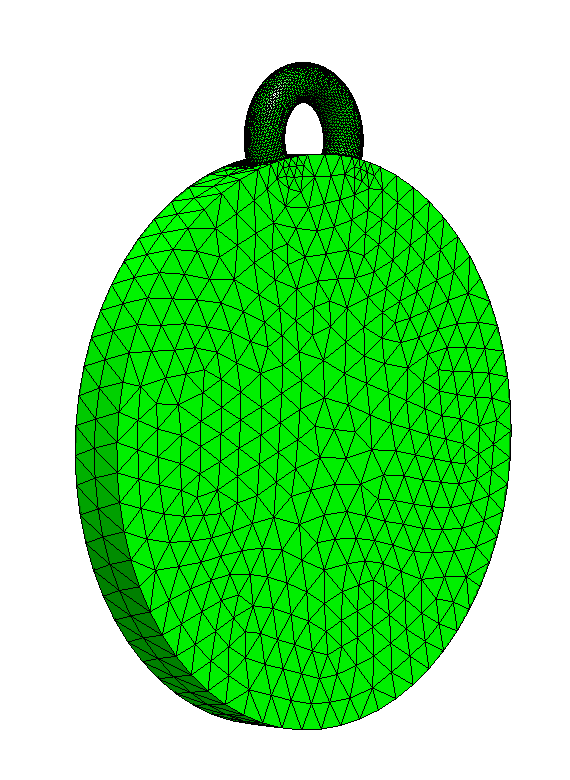} &
\includegraphics[width=0.04\textwidth, keepaspectratio]{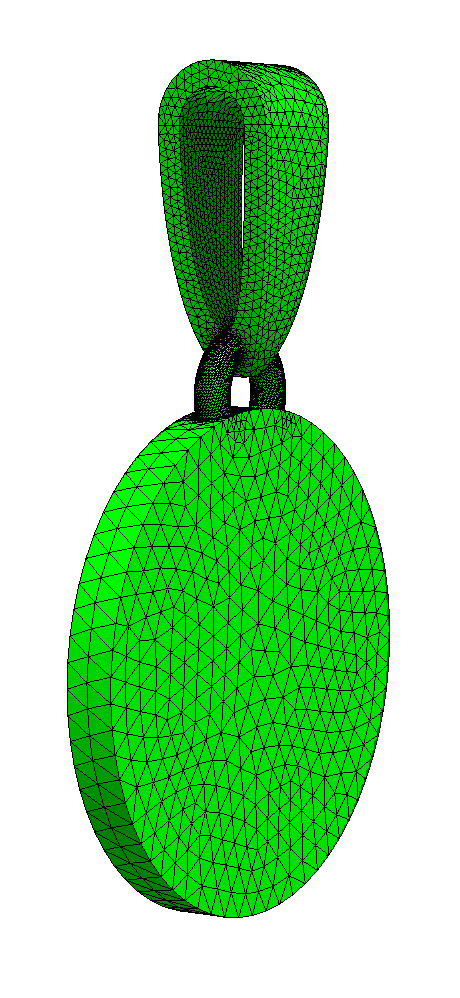}
\end{array}$ 
\end{center}
\begin{center}
$(b)$ Selection of different pendants.
\end{center}
\begin{center}
$\begin{array}{cccccc}
\includegraphics[width=0.1\textwidth, keepaspectratio]{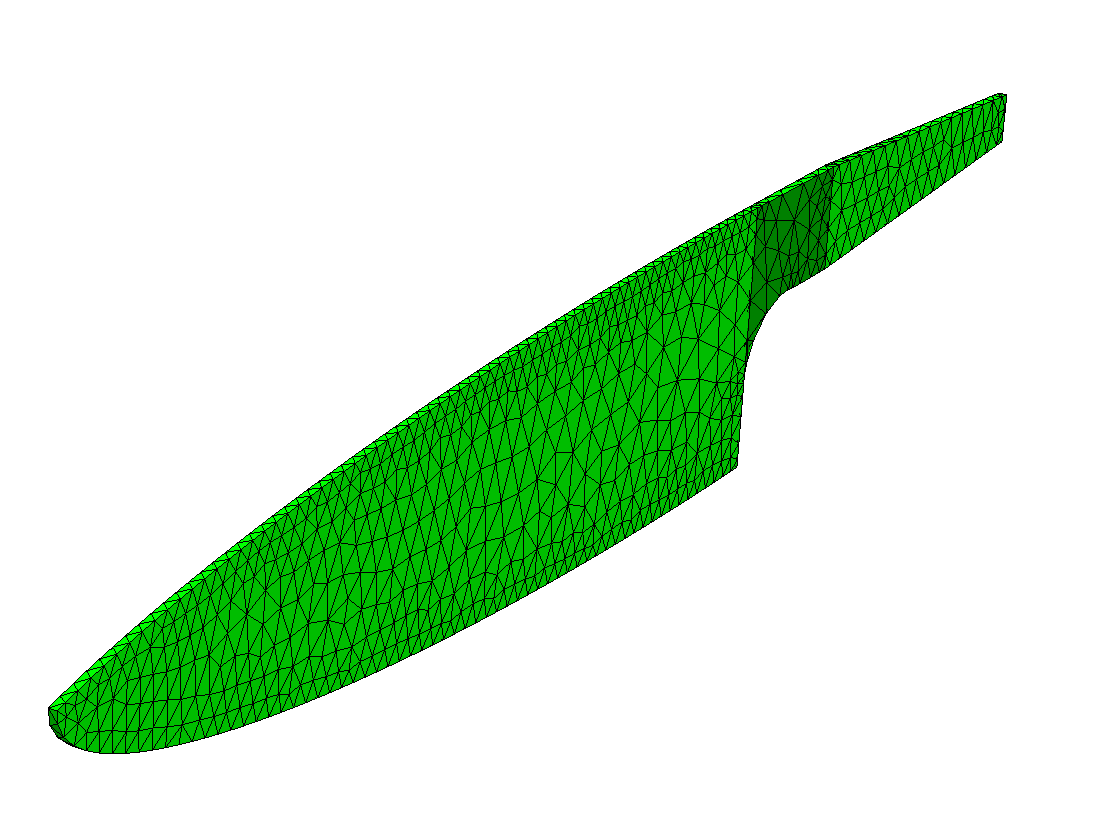} &
\includegraphics[width=0.1\textwidth, keepaspectratio]{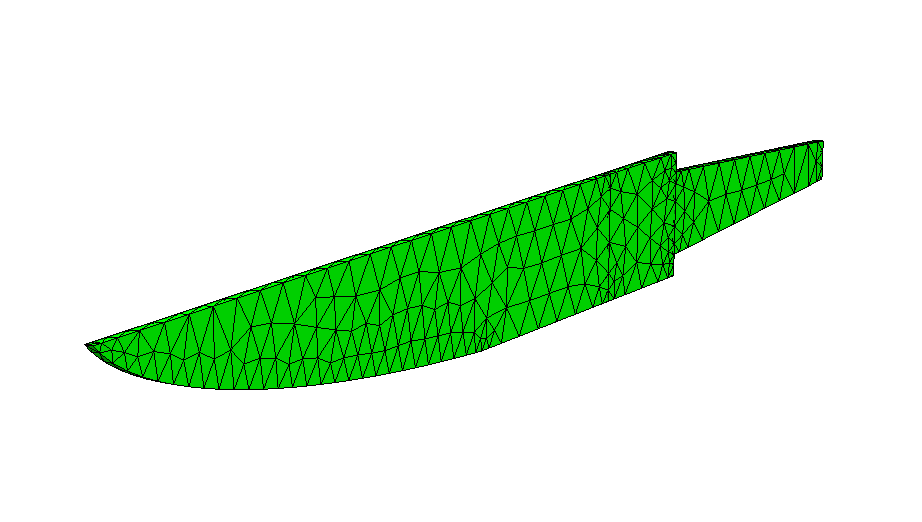} &
\includegraphics[width=0.1\textwidth, keepaspectratio]{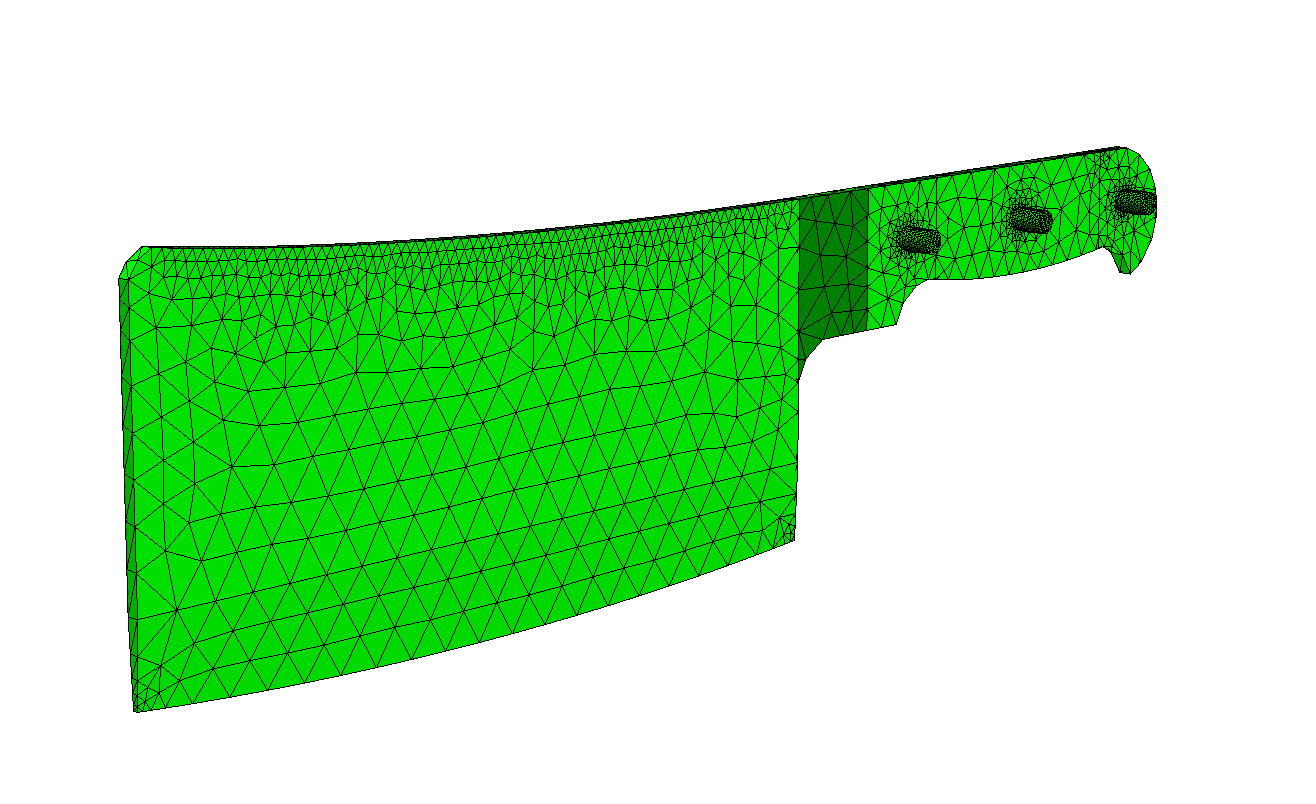} &
\includegraphics[width=0.1\textwidth, keepaspectratio]{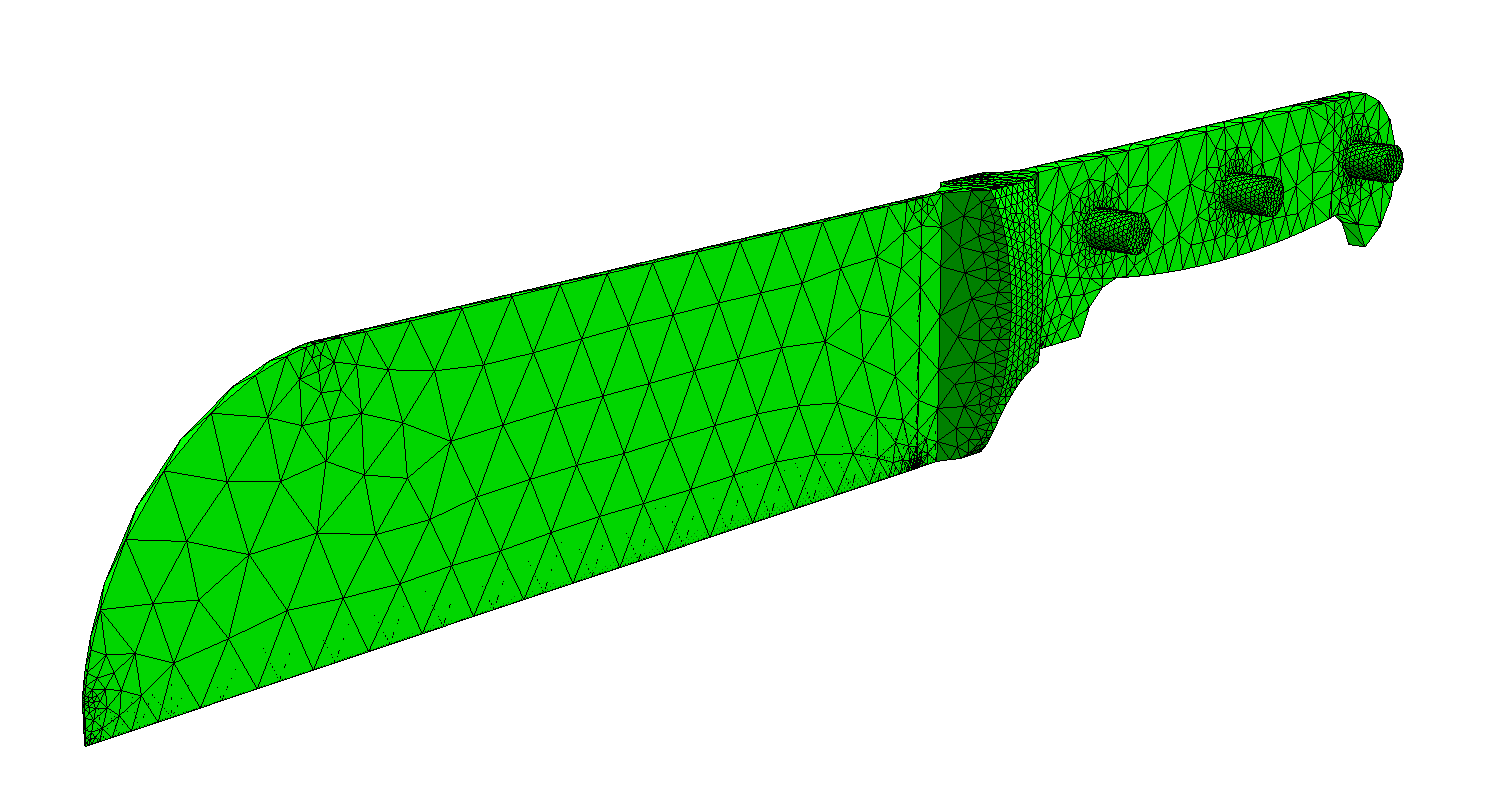} &
\includegraphics[width=0.1\textwidth, keepaspectratio]{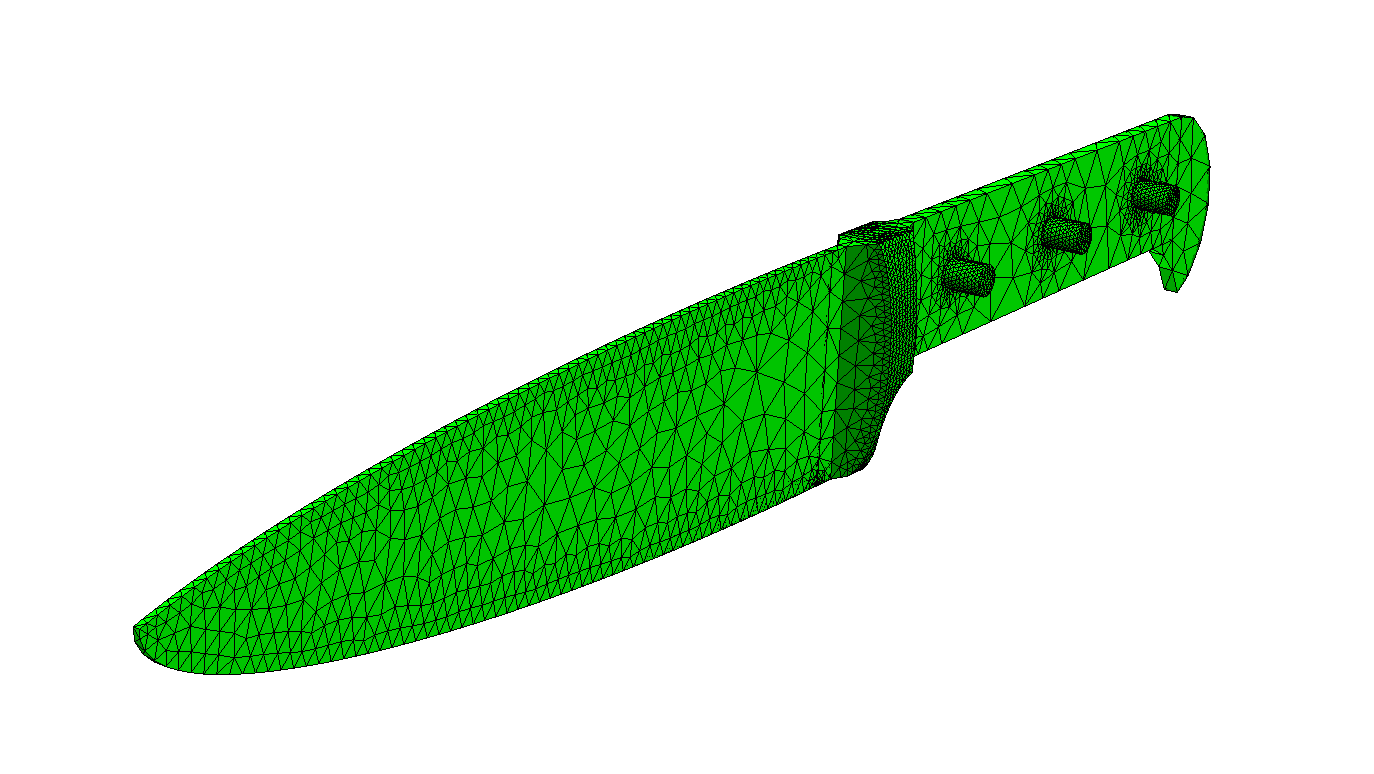} &
\includegraphics[width=0.075\textwidth, keepaspectratio]{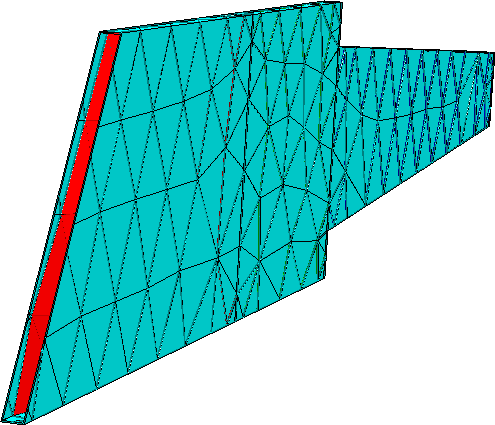}
\end{array}$
\end{center}
\begin{center}
$(c)$ Kitchen knives with magnetic and non-magnetic blades (modelled with prismatic layers).
\end{center}
\begin{center}
$\begin{array}{ccccccc}
\includegraphics[width=0.1\textwidth, keepaspectratio]{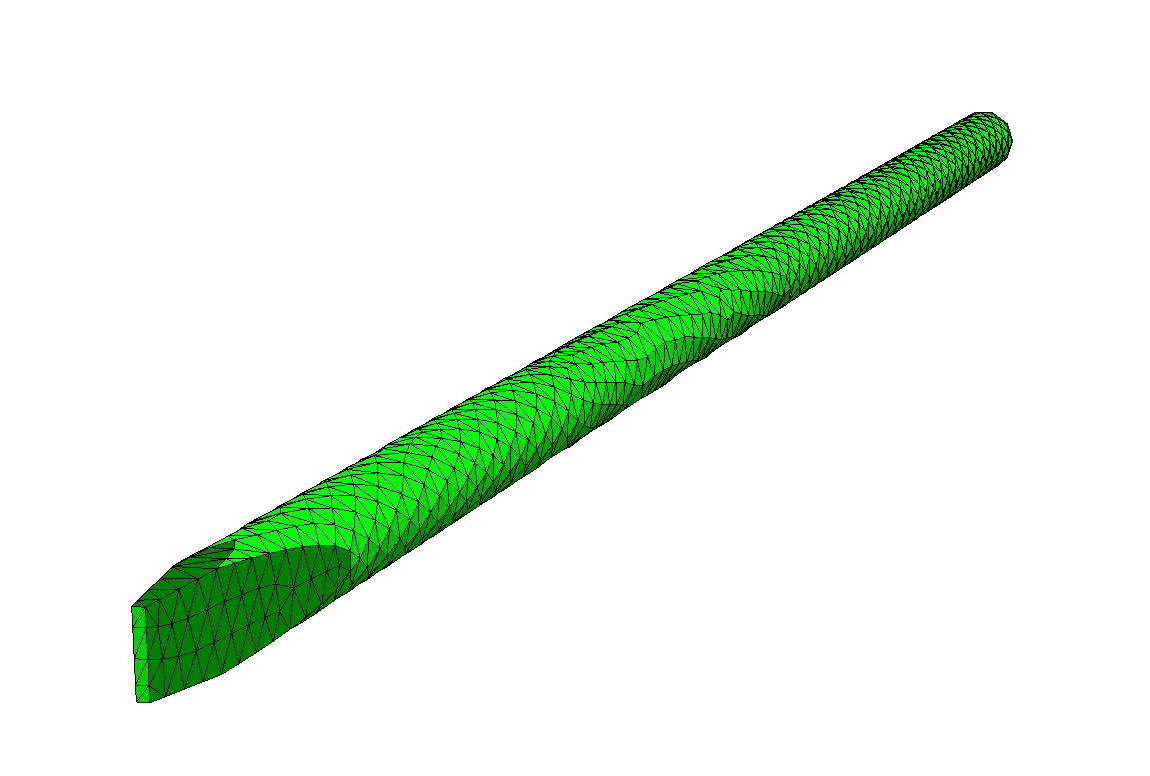} & 
\includegraphics[width=0.1\textwidth, keepaspectratio]{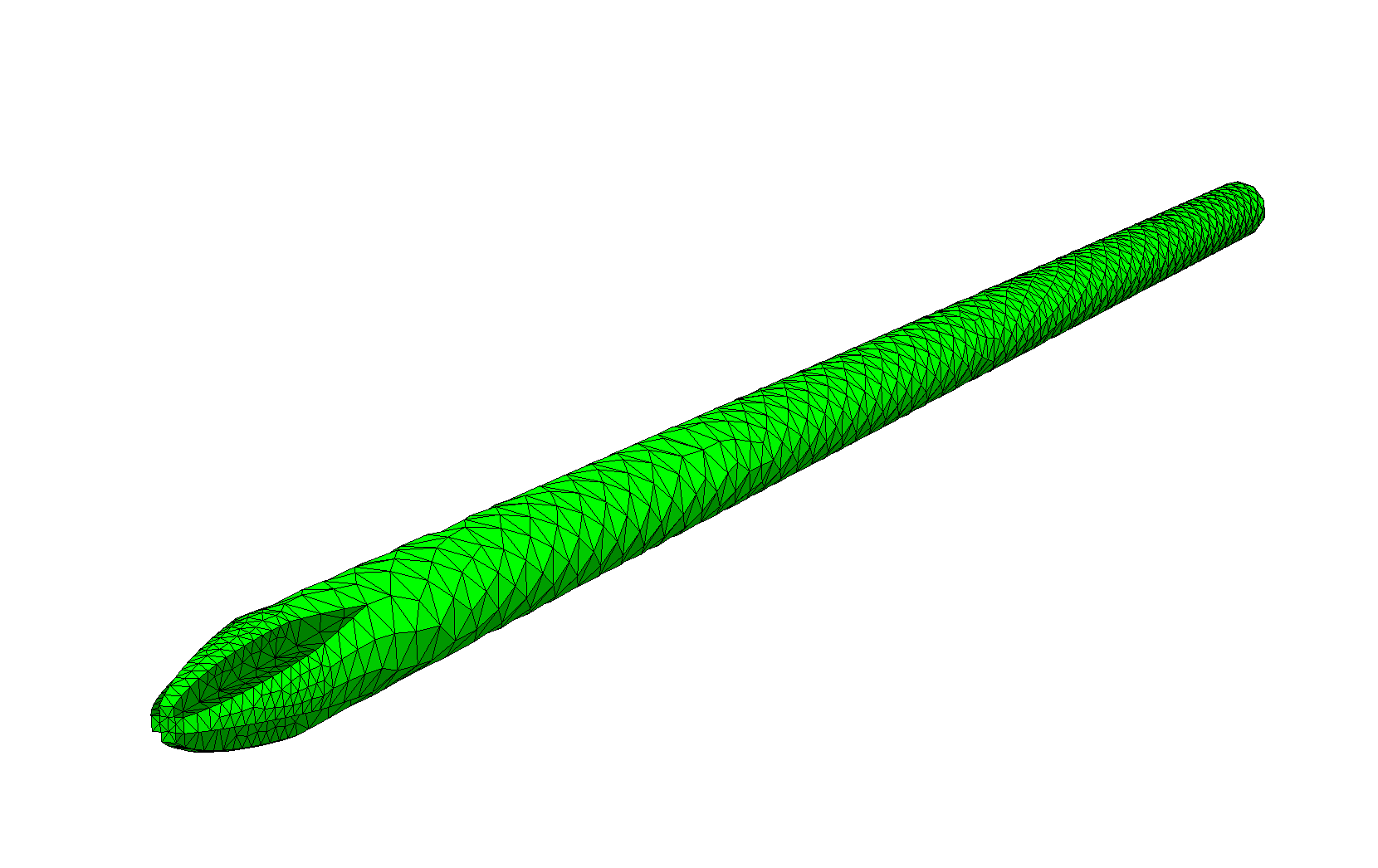} &
\includegraphics[width=0.075\textwidth, keepaspectratio]{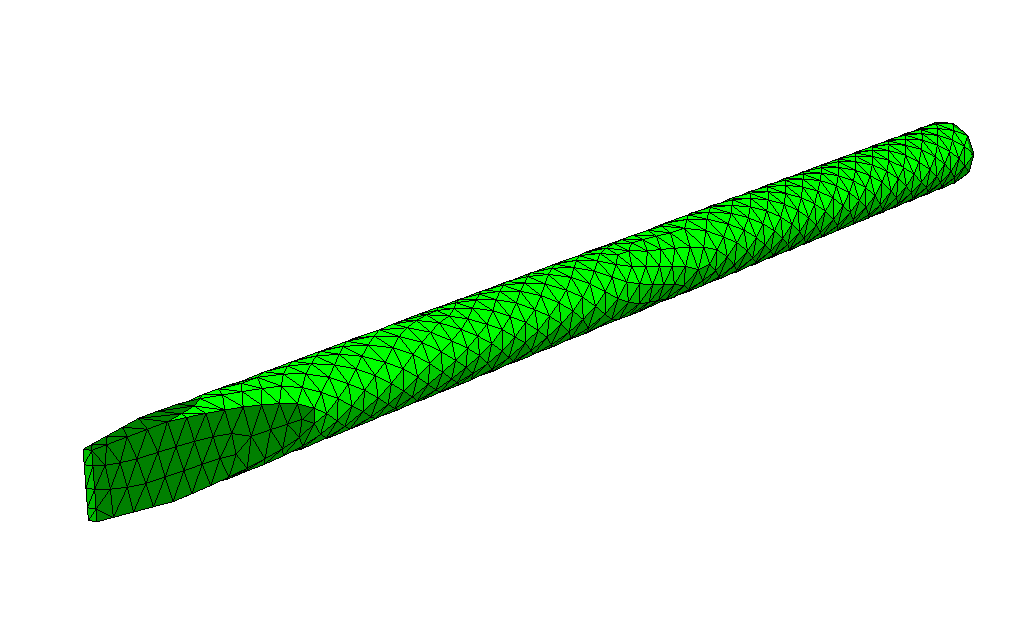}&
\includegraphics[width=0.075\textwidth, keepaspectratio]{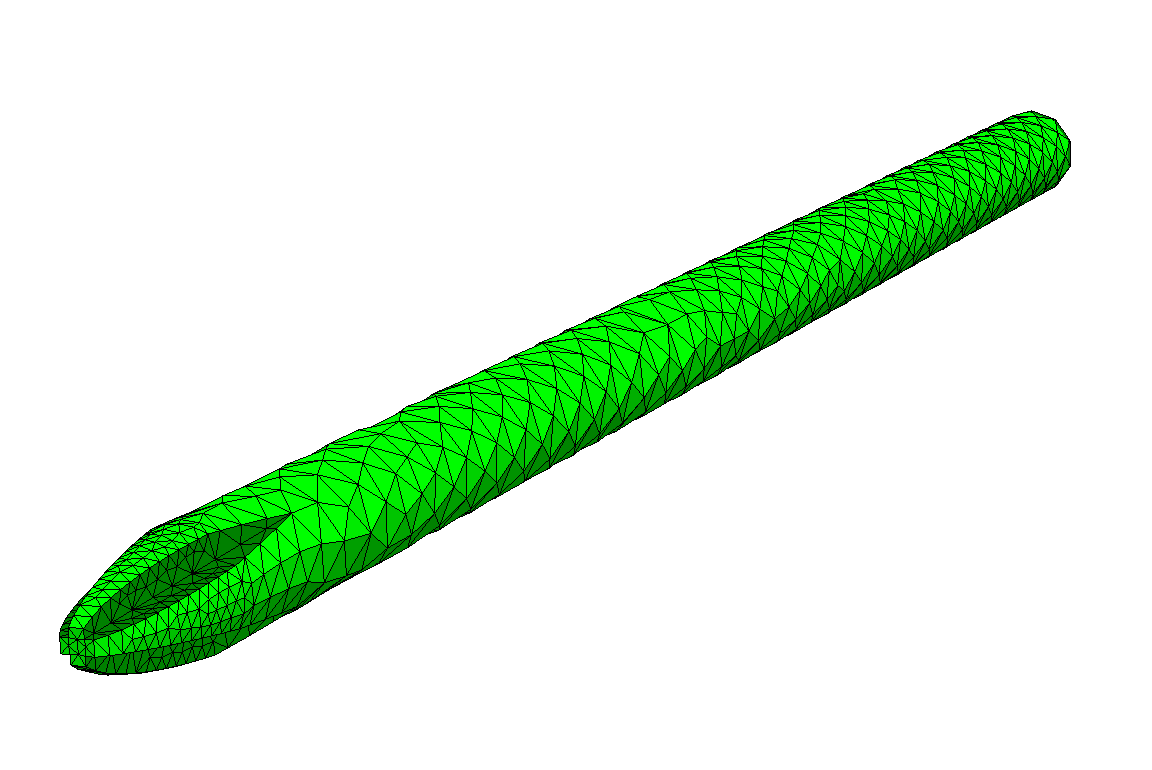}&
\includegraphics[width=0.05\textwidth, keepaspectratio]{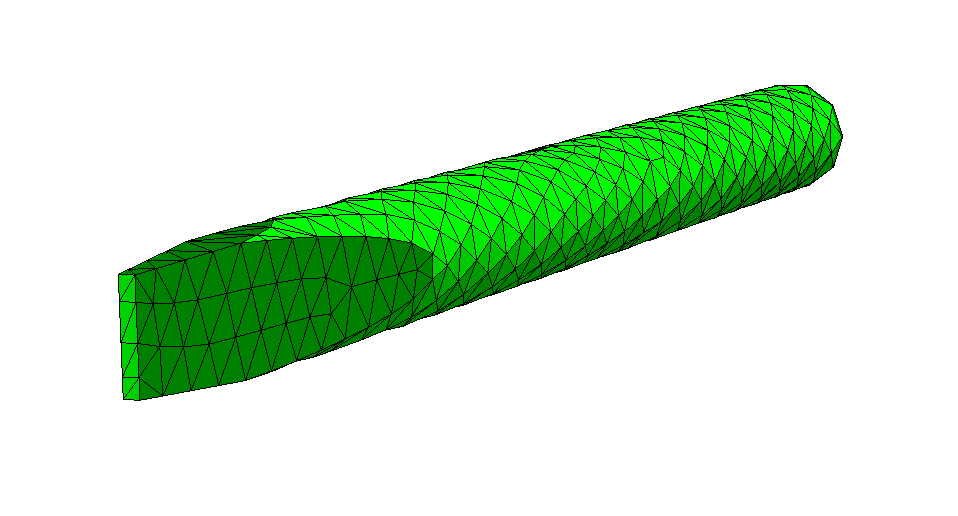} &
\includegraphics[width=0.05\textwidth, keepaspectratio]{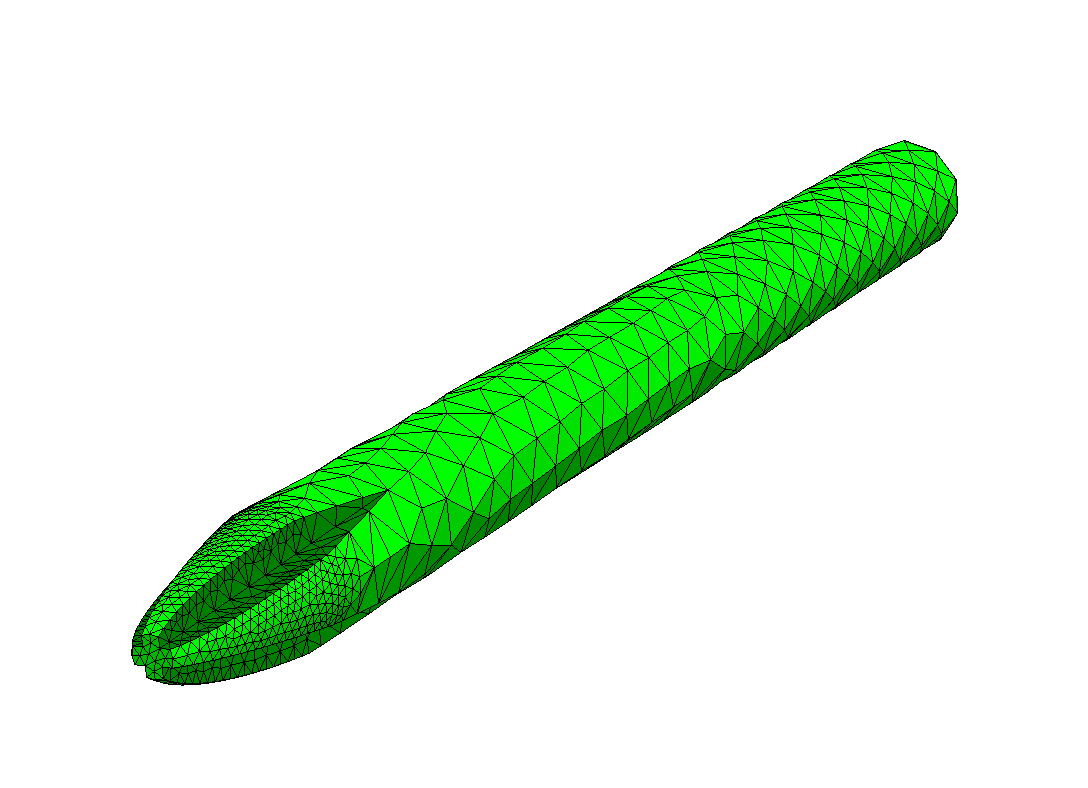} &
\includegraphics[width=0.075\textwidth, keepaspectratio]{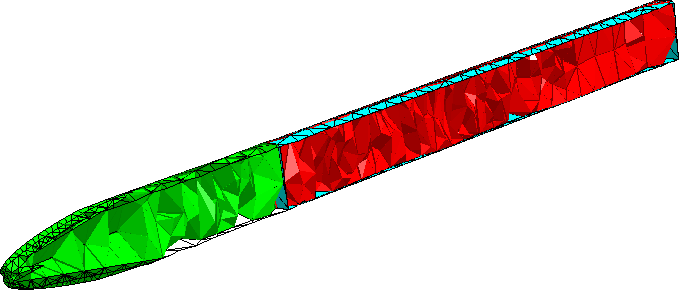}
\end{array}$
\end{center}
\begin{center}
$(d)$ Magnetic screwdrivers (with magnetic tips modelled using prismatic layers).
\end{center}
\begin{center}
$\begin{array}{ccccccc}
\includegraphics[width=0.1\textwidth, keepaspectratio]{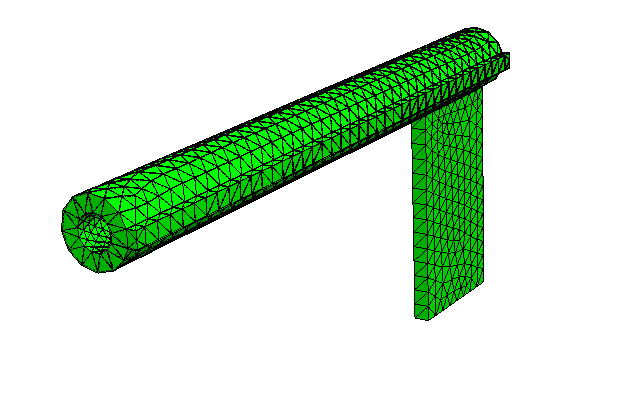} &
\includegraphics[width=0.1\textwidth, keepaspectratio]{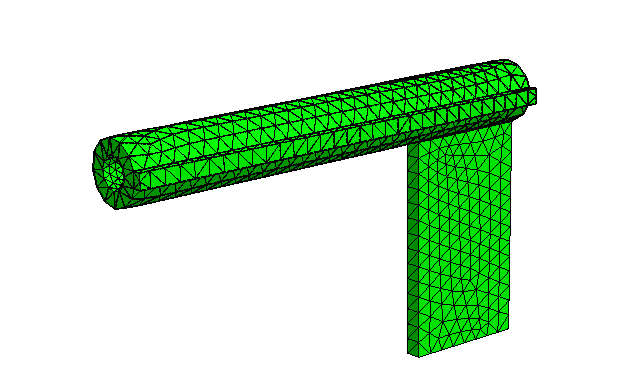} &
\includegraphics[width=0.1\textwidth, keepaspectratio]{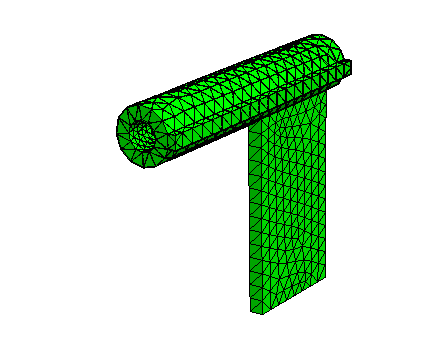} &
\includegraphics[width=0.12\textwidth, keepaspectratio]{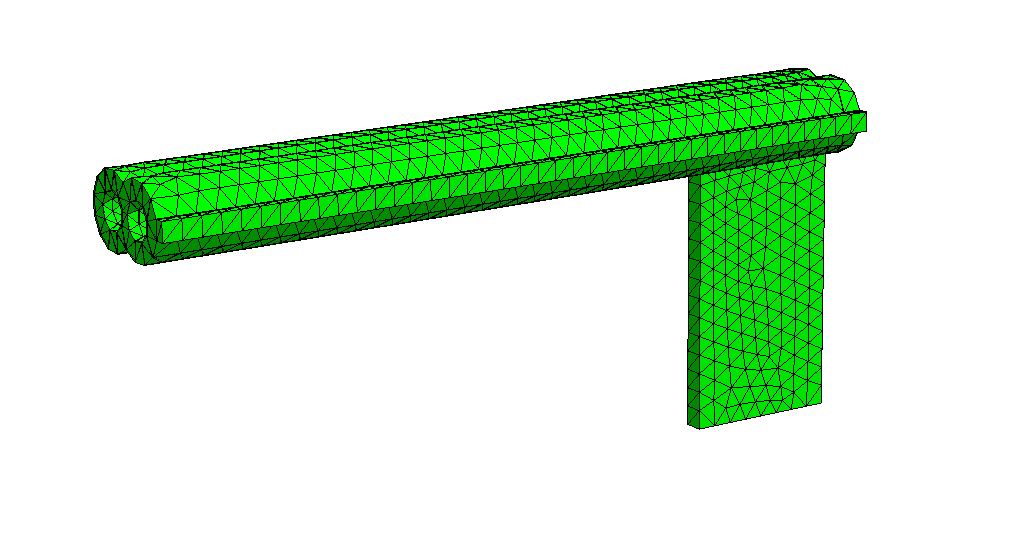} &
\includegraphics[width=0.12\textwidth, keepaspectratio]{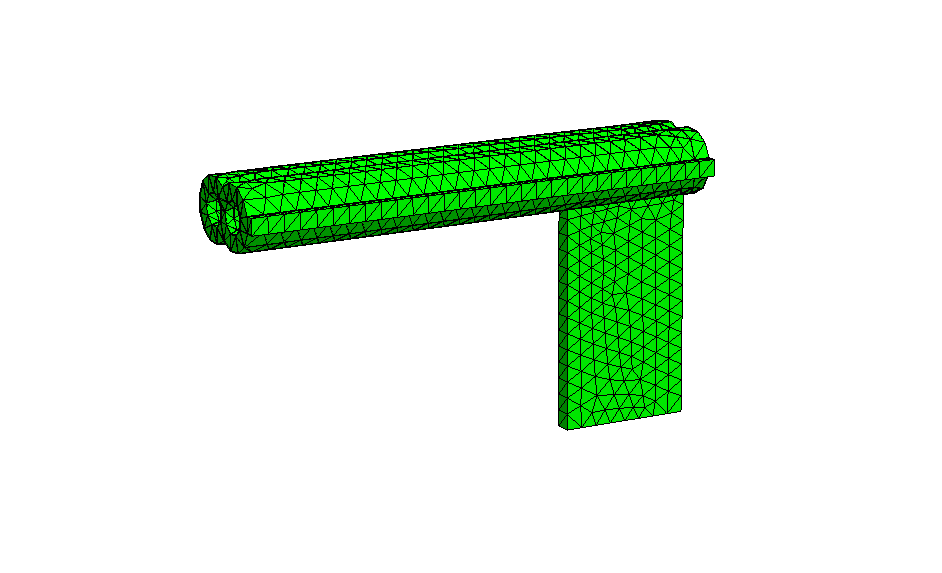} &
\includegraphics[width=0.12\textwidth, keepaspectratio]{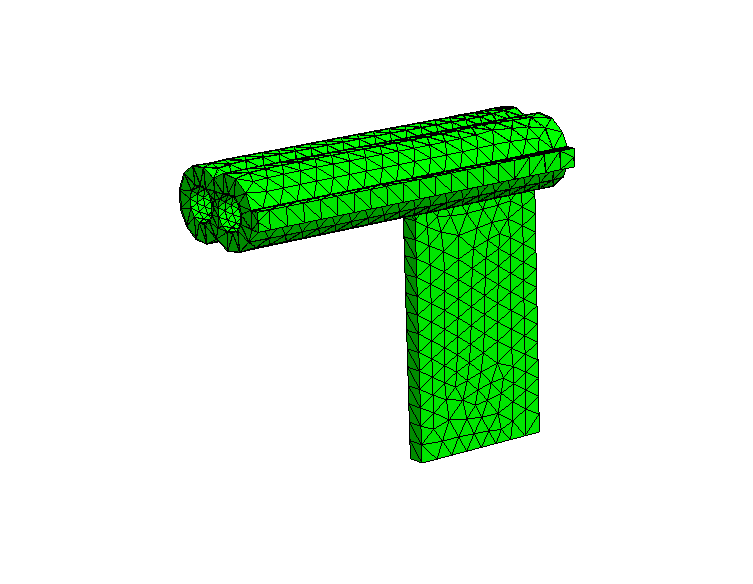} &
\includegraphics[width=0.05\textwidth, keepaspectratio]{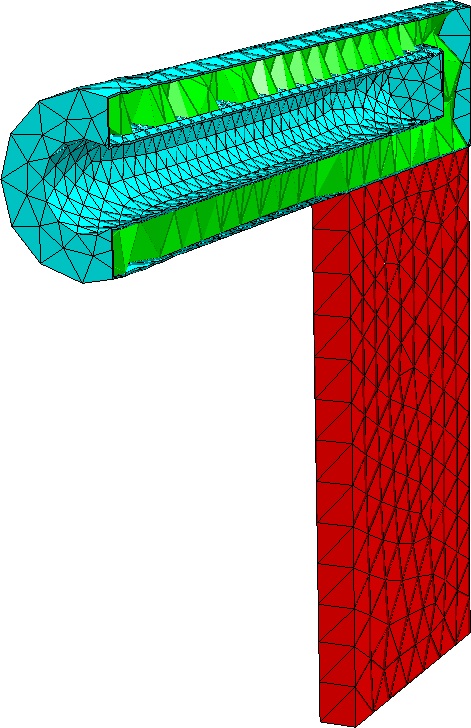}
\end{array}$
\end{center}
\begin{center}
$(e)$ Toy gun models with  non-magnetic receivers and magnetic barrels (modelled with prismatic layers).
\end{center}
\caption{Bayesian classification: Illustration of  surfaces meshes and typical cuts through objects showing prismatic layers for a  sample of non-magnetic and magnetic objects showing:  $(a)$ British coins,  $(b)$ pendants, $(c)$ kitchen knives, $(d)$ screw drivers and $(e)$ toy gun models. {In the cuts, light green is used to highlight the prismatic layers while red and green are used to represent the internal tetrahedra for different materials.}} \label{fig:objectmeshes}
\end{figure}

Using this dataset, we will consider a classification problem consisting of $K=5$ classes, which are  British coins $C_1$; pendents $C_2$;  kitchen knives $C_3$;  screwdrivers $C_4$ and toy gun models $C_5$. Of course, this is just one possible grouping of the objects in the dataset, which is based on their physical meaning and other groupings are of course possible.
 Figure~\ref{fig:dRclassesofobjects} illustrates how the additional feature information  $d_{E,\theta}(\tilde{\mathcal R}(\alpha B, \omega, \sigma_*, \mu_r), {\mathcal I}(\alpha B, \omega, \sigma_*, \mu_r))$ varies as a function of $\omega$ for the classes of objects considered, {where, from a visual inspection, there are noticeable differences in the spectral signature of $d_{E,\theta} (\tilde{\mathcal R},{\mathcal I})$ for each of the classes and similarities within the objects that have been grouped.}  The results for the British coins are not shown, but are similar to those obtained for the pendants. Indeed,   given the mirror symmetries of the pendants (and British coins), we expect $d_{E,\theta} (\tilde{\mathcal R},{\mathcal I})$ to be negligible, which is what is obtained. The kitchen knives, screwdrivers and toy gun models all exhibit different behaviour, but similarities within each class can be observed indicating that our choice of object groupings in the classes is appropriate.

\begin{figure}
\begin{center}
$\begin{array}{cc}
\includegraphics[width=0.5\textwidth]{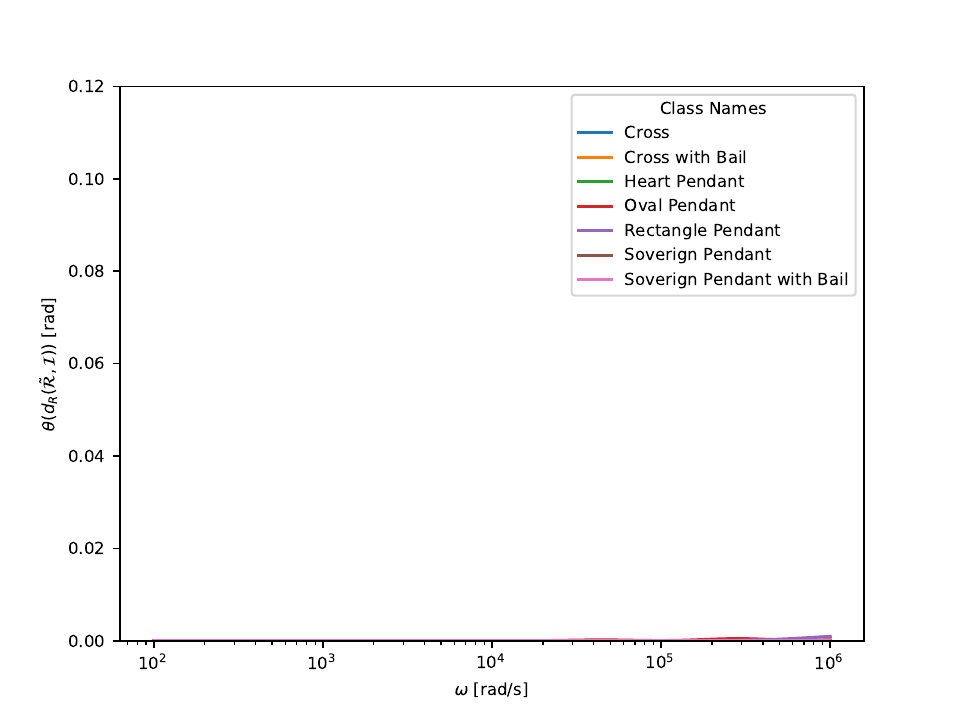} &
\includegraphics[width=0.5\textwidth]{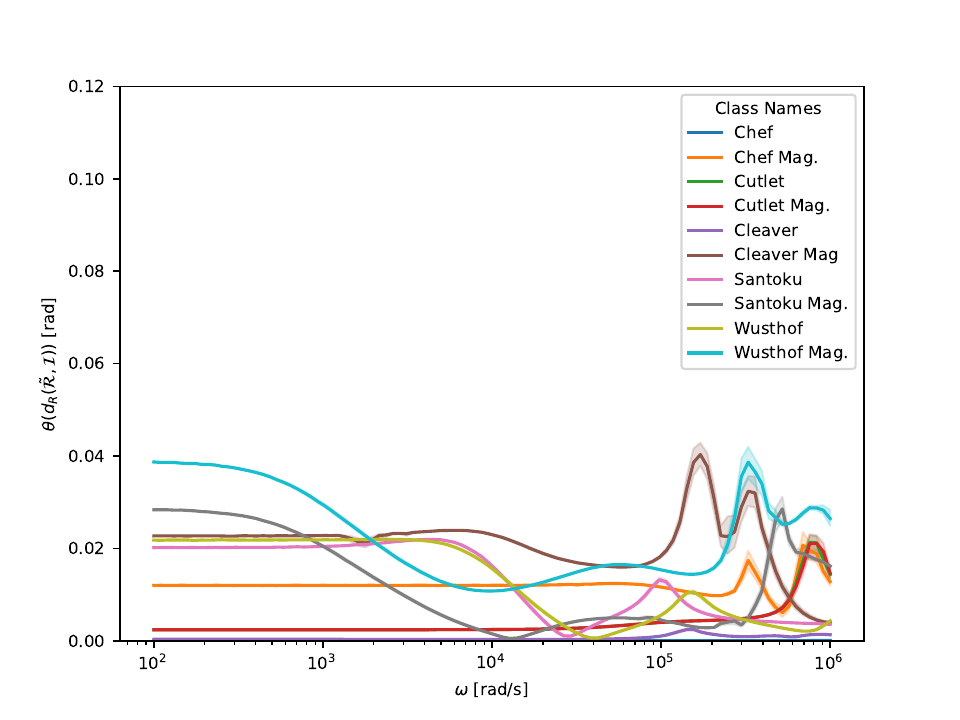} \\
\text{$(a)$ Pendants} & \text{$(b)$ Knives \& Magnetic Knives}\\
\includegraphics[width=0.5\textwidth]{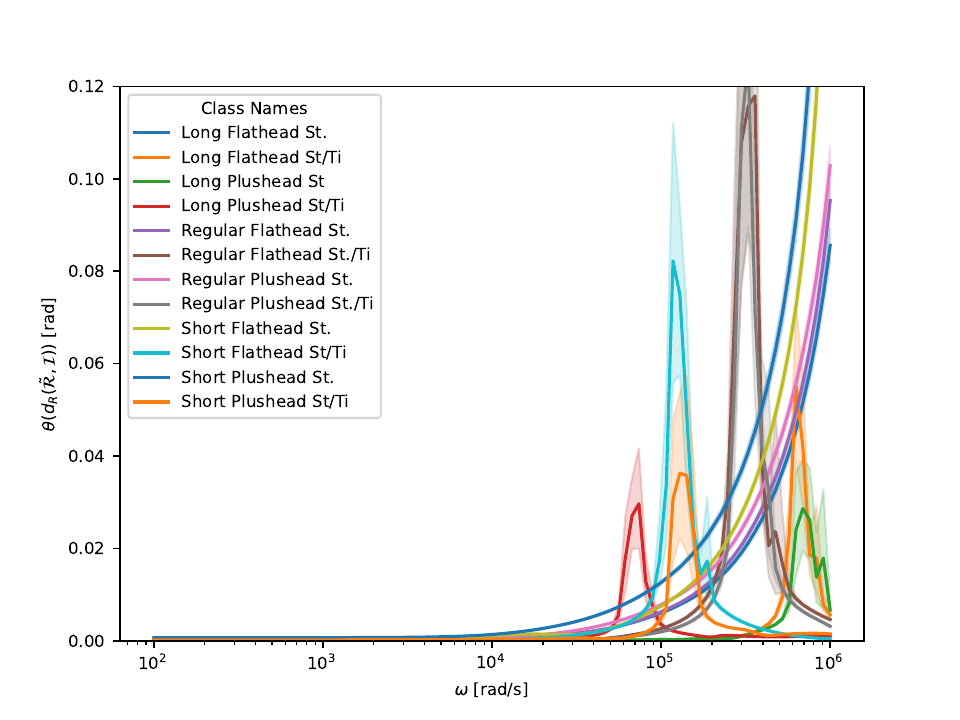} &
\includegraphics[width=0.5\textwidth]{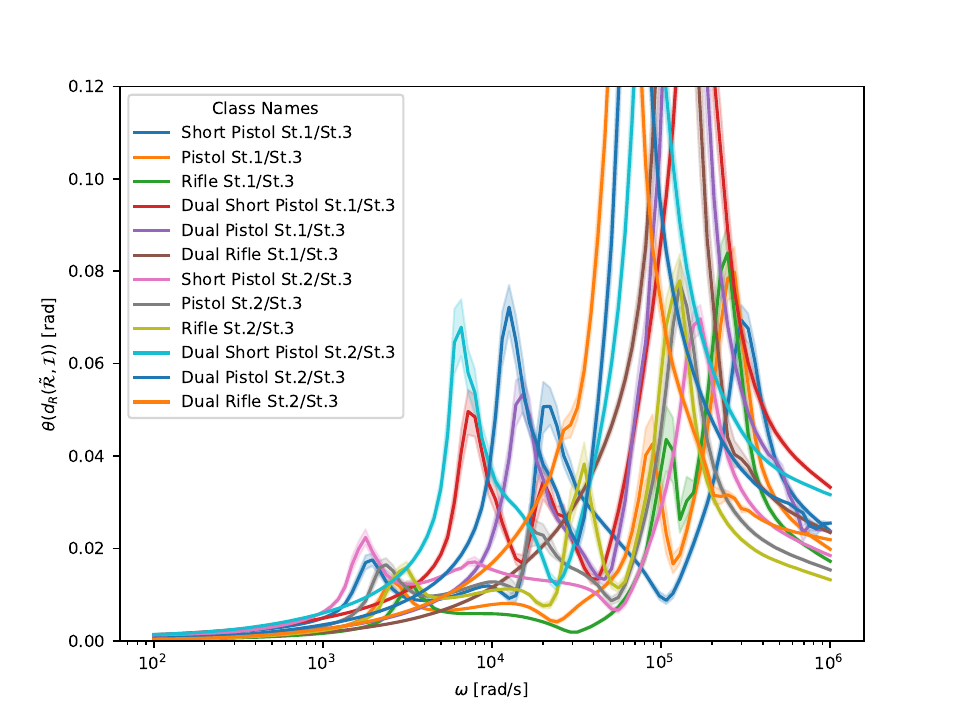} \\
 \text{$(c)$ Magnetic Screwdrivers} & \text{$(d)$ Magnetic Guns}
\end{array}$
\end{center}
\caption{Bayesian classification: Results of $d_{E,\theta} (\tilde{\mathcal R},{\mathcal I}))$ for a range of different objects, materials and sizes, showing: $(a)$ classes of pendants, $(b)$ classes of kitchen knives, $(c)$ classes of screwdrivers and $(d)$ classes of magnetic guns. } \label{fig:dRclassesofobjects}
\end{figure}

The number of features $F$ can be reduced to $G \ll F$ by using a truncated SVD. Introducing a matrix $D$ whose columns are the feature vectors, and approximating it by a truncated singular value decomposition, gives
\begin{align}
D =\left  ( {\rm x}_1 , {\rm x}_2 \ldots {\rm x}_P \right )   \approx { U}^G { \Sigma}^G ({ V}^G)^T \in {\mathbb R}^{G \times P} \nonumber ,
\end{align}
where ${ U}^G \in {\mathbb R}^{F \times G}$, ${ \Sigma}^G \in {\mathbb R}^{G \times G}$, $ { V}^G \in {\mathbb R}^{ P \times G}$ are the truncated matrices of left singular vectors, diagonal matrix of truncated singular values (arranged in descending order) and truncated matrix of right singular vectors, respectively. In a similar way to proper orthogonal decomposition~\cite{ben2020}, $ { U}^G$ can be used a projection leading to
\begin{align}
({ U}^G)^T { D} = \left (\tilde{\rm x}_1 , \tilde{\rm x}_2 , \cdots , \tilde{\rm x}_P  \right)  \in {\mathbb R}^{G \times P}, \nonumber
\end{align}
where $ \tilde{\rm x}_n = ({U}^G)^T {\rm x}_n \in {\mathbb R}^G$, $n=1,\ldots,P$ are reduced vectors of features. The updated  data set is
\begin{align}
{\mathcal D} := \left \{ (\tilde{\rm x}_1 , {\rm t}_1), (\tilde{\rm x}_2 , {\rm t}_2), \ldots, (\tilde{\rm x}_P , {\rm t}_P) \right \}, \nonumber
\end{align}
which is split in to training and testing datasets using a 3:1 split.
Given a measurement vector ${ x}\in {\mathbb R}^F$, this can be transformed as $\tilde{ x}  = ({U}^G)^T { x}\in{\mathbb R}^G$. An illustration of the decay of the relative singular values  is shown in Figure~\ref{fig:singularvaluesdecay}. For example, by choosing  truncation at $(\Sigma)_{nn}/(\Sigma)_{11} =10^{-9}$ and applying the above projection leads to a reduction from $F=170$ to $G=70$. {We justify this choice of truncation as we expect that the first two dips in the relative singular values are significant, but that the subsequent behaviour is associated with the effects of rounding in the numerical computations.}

\begin{figure}
\begin{center}
\includegraphics[width=0.5\textwidth]{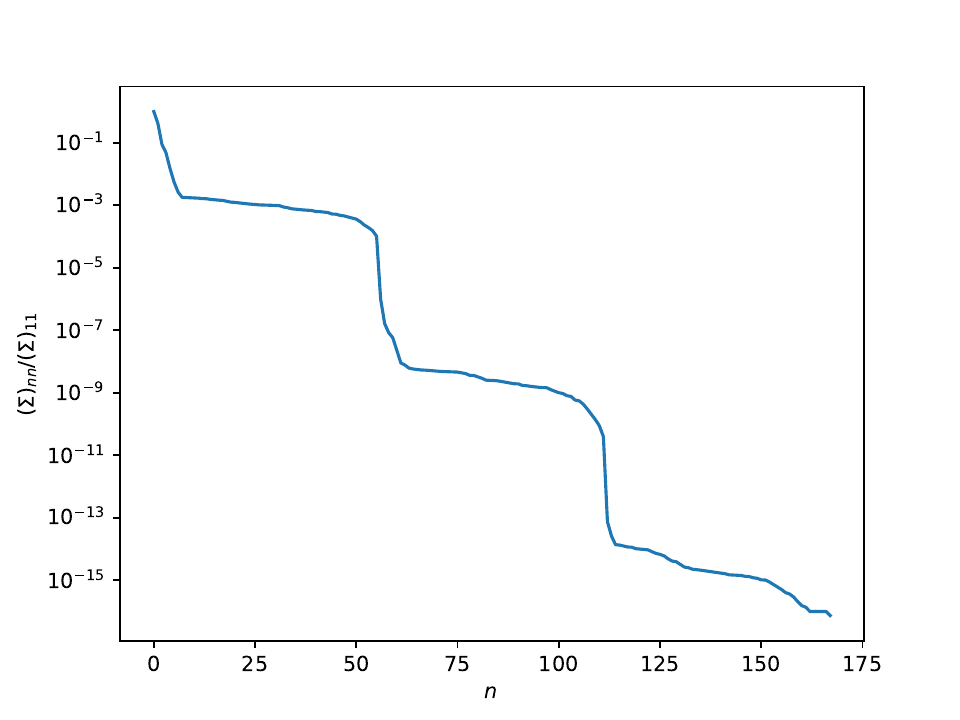} 
\end{center}
\caption{Bayesian classification: Illustration of the decay of the relative singular values of $D$. For example, choosing  truncation at $(\Sigma)_{nn}/(\Sigma)_{11} =10^{-9}$ leads to a reduction from $F=170$ to $G=70$.} \label{fig:singularvaluesdecay}
\end{figure}

\subsubsection{Object classification}

While there are many different classification approaches that could be applied~\cite{bishopbook}, we choose to focus on comparing two Bayesian based classification approaches, namely a Bayesian logistic regression classifier and a Bayesian multi-layer perceptron neural network classifier~\cite{bishopbook,pml1Book,pml2Book} using the \texttt{Probflow} library~\cite{probflow}. 
The Bayesian approach differs from standard deterministic classification approaches to logistic regression and multi-layer perceptron neural network approaches by considering the weights and bias to be random, rather than deterministic, variables. In the Bayesian approach, prior distributions are specified for these variables and the output posterior probabilities are distributions rather than deterministic values. However, even if the  priors are considered to be  Gaussian distributions with zero mean,  the variance (or covariance in general) still needs to be set. The approach followed is to specify a prior on the prior in the form of an inverse gamma distribution with suitably chosen concentration and weight.   Further details of the Bayesian approach applied to object classification using MPTs is planned for a subsequent publication.

Once the  classifier has been trained, then given  a vector of measured features ${x}\in {\mathbb R}^F$, the result of the classification is not a deterministic value, but a probability distribution  of the posterior conditional probability $p(C_k| x)$, which informs us of how likely the (hidden) object belongs to class $C_k$ given  $x$. Exploring this distribution leads to the violin plots shown in Figures~\ref{fig:Bayclasslogreg} and~\ref{fig:BayclassoptNN} for Bayesian logistic regression and a  Bayesian neural network, respectively.
Each plot in these figures corresponds to a case where the feature vector $x$ is sampled from the test dataset and corresponding to where the correct class is either British coins, pendants, kitchen knives, screwdrivers or the simplified gun model, respectively. The results have been obtained for the case where, prior to feature reduction, there are $N=1000$ objects and $M=100$ discrete frequencies.
 To determine the architecture of the neural network, a deterministic search is first performed to determine the number of layers and number of neurons in each layer prior to applying the Bayesian classification.

Comparing the results in Figures~\ref{fig:Bayclasslogreg} and~\ref{fig:BayclassoptNN}, we observe the superior performance of the Bayesian neural network approach. For instance, results for the Bayesian logistic regression classifier in Figure~\ref{fig:Bayclasslogreg} $(a)$ indicate that the distribution of $p(C_1|x)$, when $x$ corresponds to features for a British coin, to be centred around $p(C_1|x)\approx0.6$ and a distribution for  $p(C_2|x)$ centered around  $p(C_2|x)\approx0.4$; Predicting the class to be only slightly more likely to be a British coin rather than a pendant, but also illustrates the uncertainty in this prediction. The behaviour is similar in Figure~\ref{fig:Bayclasslogreg} $(b)$. 
On the other hand, Figure~\ref{fig:BayclassoptNN} $(a)$ illustrates that the Bayesian neural network classifier predicts distributions for $p(C_1|x)$, $p (C_2|x)$ centred around  
$p(C_1|x)\approx 1.0$ and $p(C_2|x)\approx 0.0$, respectively, with the distributions for $p(C_1|x)$, $p(C_2|x)$  having thin  long tails  illustrating some small uncertainty in the predictions. 
Thus, compared to Bayesian logistic regression, this classifier correctly predicts that the class is a British coin with a high degree of certainty. 
The situation is also improved in 
Figure~\ref{fig:BayclassoptNN} $(b)$ compared to ~\ref{fig:Bayclasslogreg} $(b)$.
For Bayesian logistic regression, the results in  Figure~\ref{fig:Bayclasslogreg} $(c)$, predict a distribution of $p(C_3|x)$,  when $x$ corresponds to features for a Kitchen knife, centred around $p(C_3|x)\approx1$. With similar good performance for Bayesian logistic regression in Figure~\ref{fig:Bayclasslogreg}~$(d)$ and $(e)$, where the correct prediction class are screwdriver and simplified guns, respectively.  The results for the Bayesian neural network classifier in Figure~\ref{fig:BayclassoptNN} $(c)$, $(d)$ and $(e)$ are similarly good as they predict the correct class in each case with high degree of certainty in the predictions.


\begin{figure}
\begin{center}
$\begin{array}{c}
\includegraphics[width=0.5\textwidth]{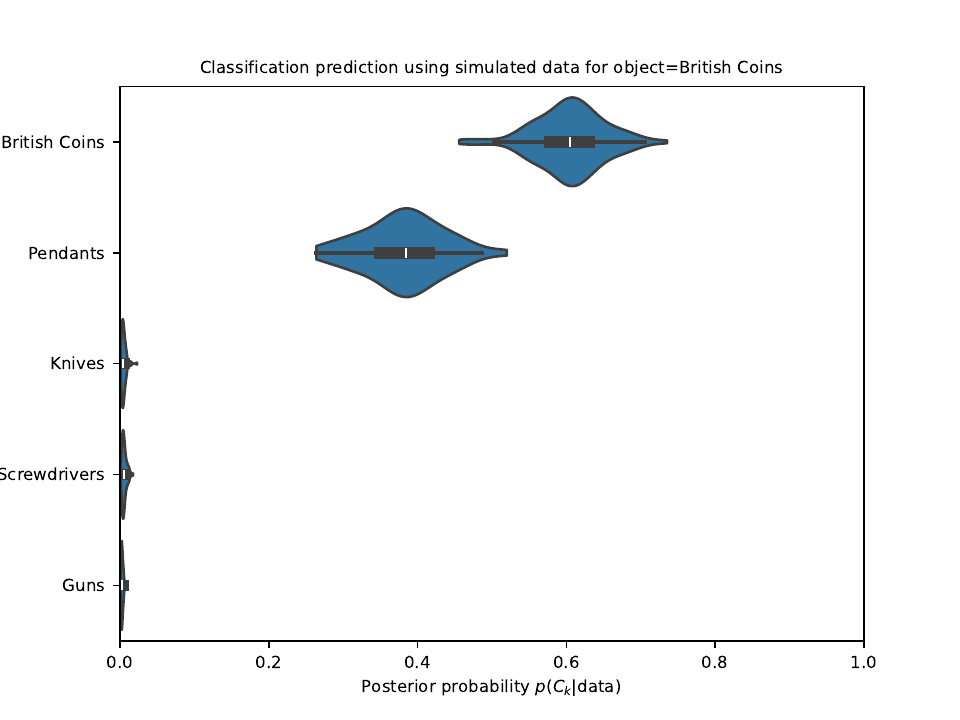} \\
(a)\\
\begin{array}{cc}
\includegraphics[width=0.5\textwidth]{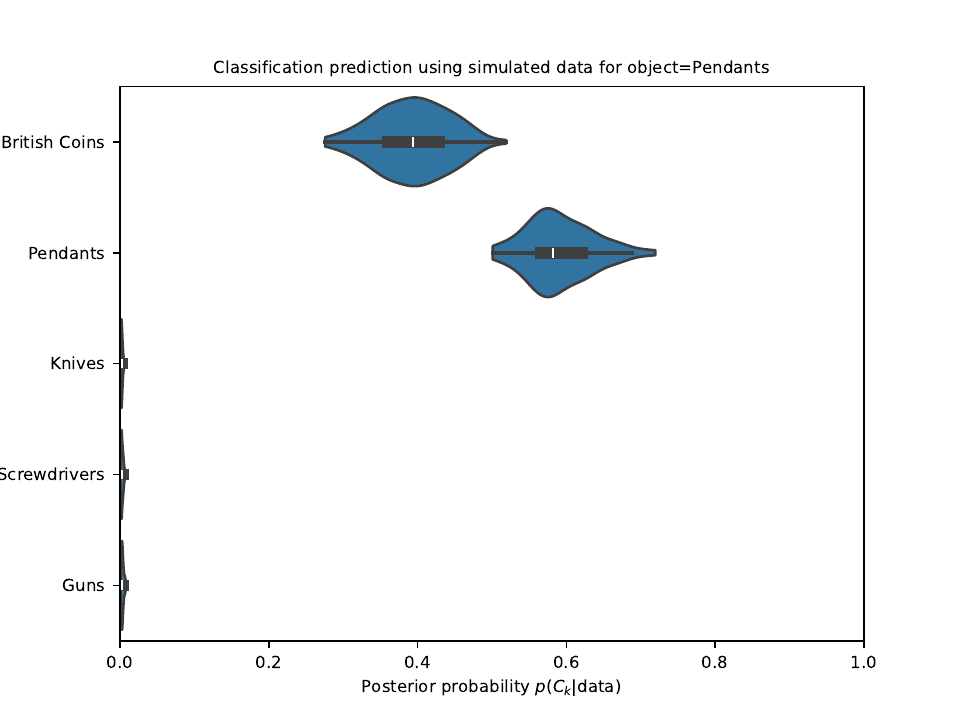} &
\includegraphics[width=0.5\textwidth]{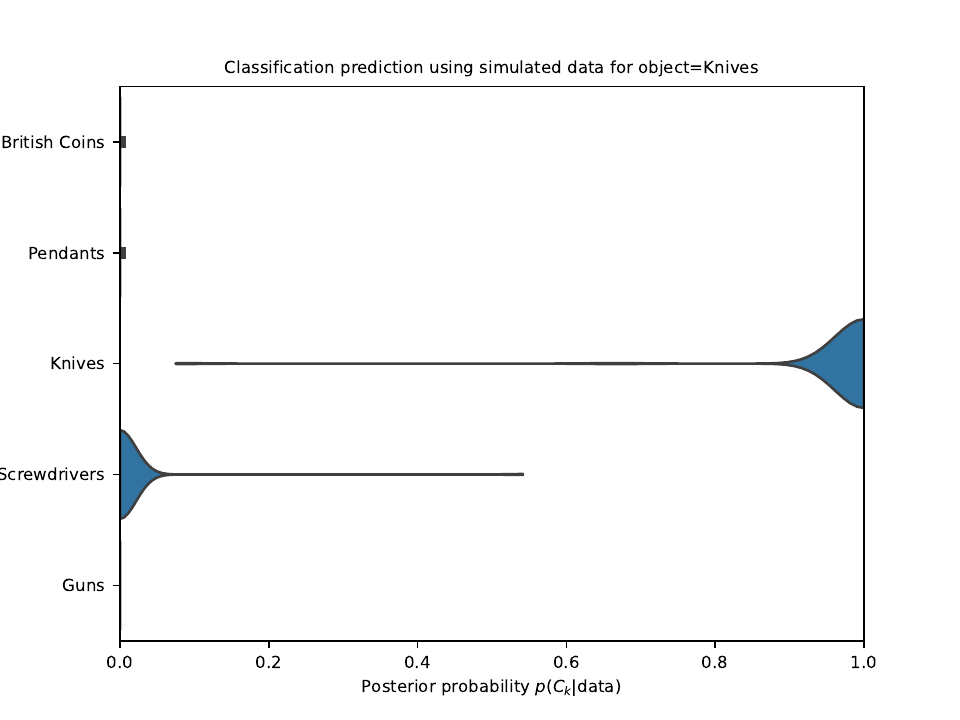} \\
(b) & (c)\\
\includegraphics[width=0.5\textwidth]{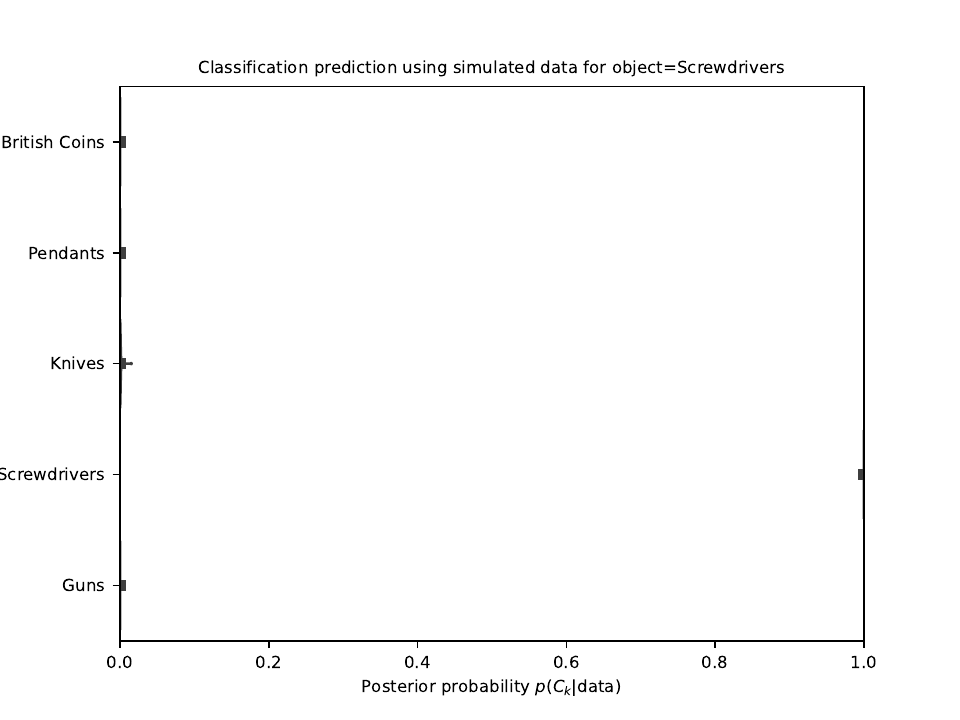} &
\includegraphics[width=0.5\textwidth]{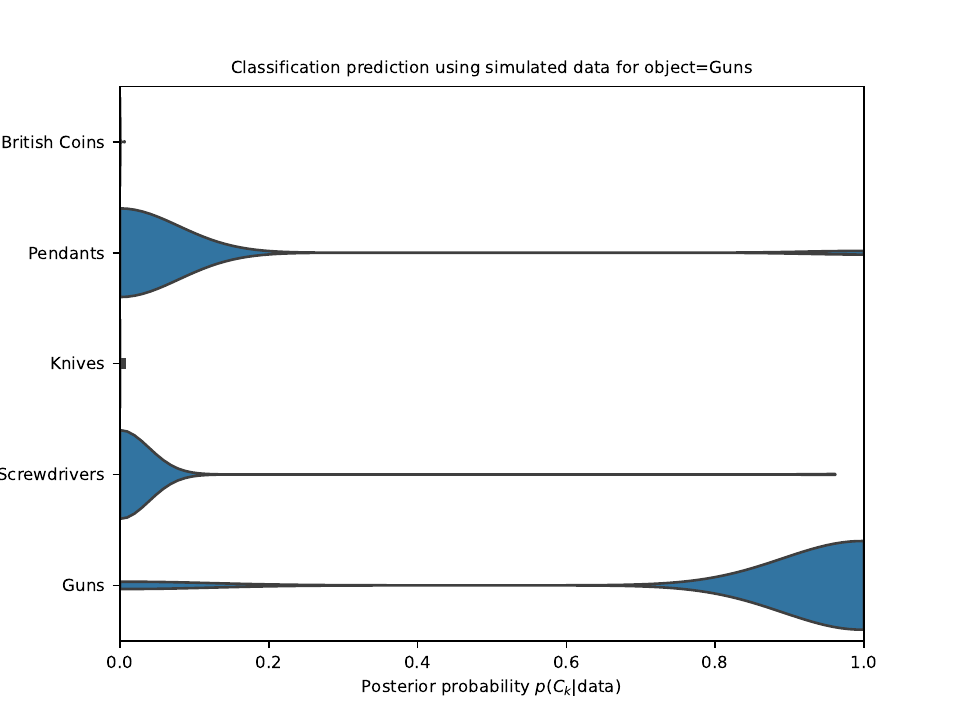} \\
(d) & (e)
\end{array}
\end{array}$
\end{center}
\caption{Bayesian logistic regression classification: Using a library with $N=1000$ objects and $M=100$ discrete frequencies where the correct classification is $(a)$ British coins, $(b)$ pendents, $(c)$ kitchen knives, $(d)$ screwdrivers and $(e)$ simplified gun models.} \label{fig:Bayclasslogreg}
\end{figure}

\begin{figure}
\begin{center}
$\begin{array}{c}
\includegraphics[width=0.5\textwidth]{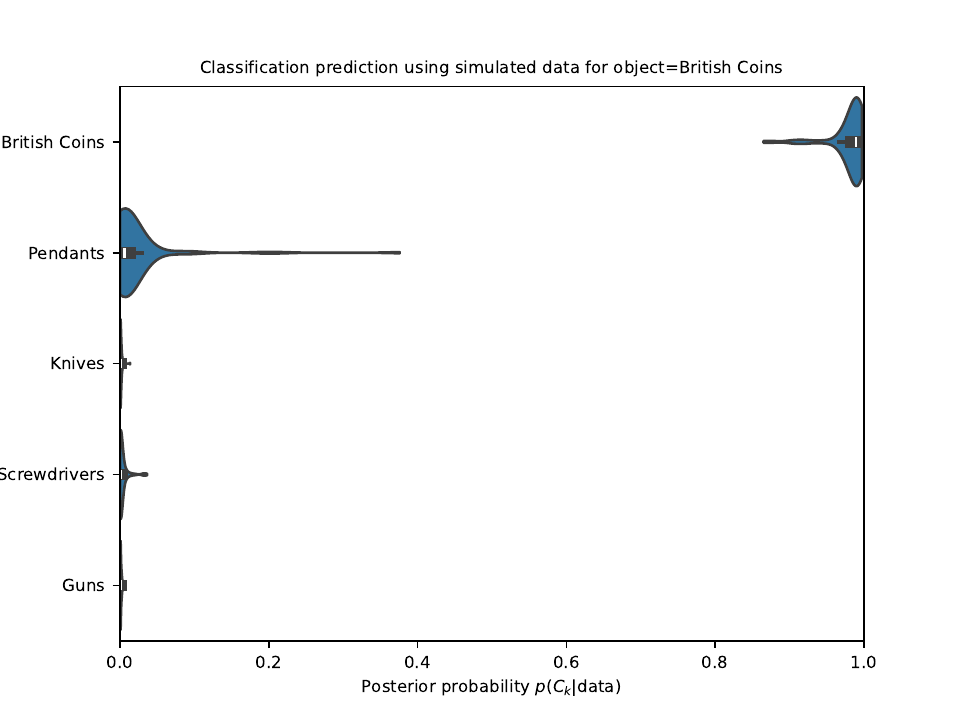} \\
(a) \\
\begin{array}{cc}
\includegraphics[width=0.5\textwidth]{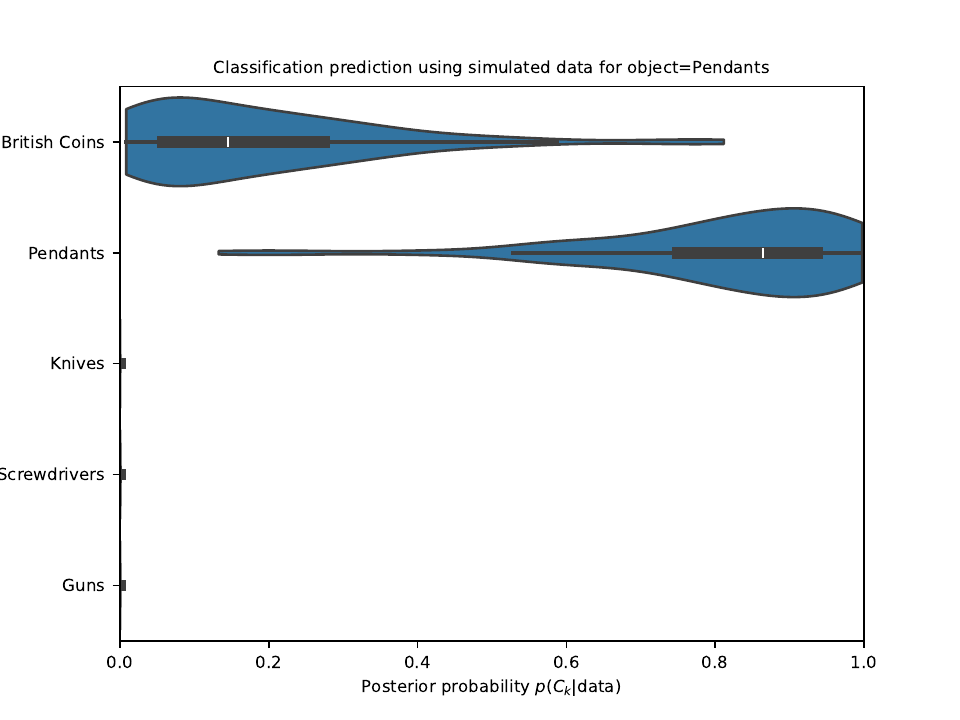} &
\includegraphics[width=0.5\textwidth]{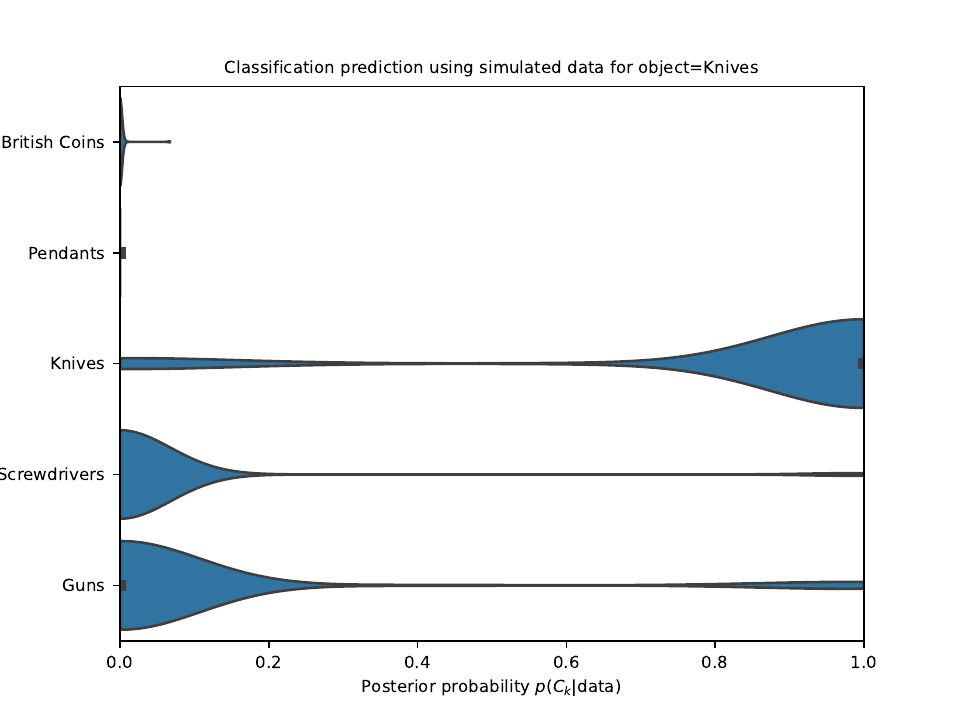} \\
(b) & (c) \\
\includegraphics[width=0.5\textwidth]{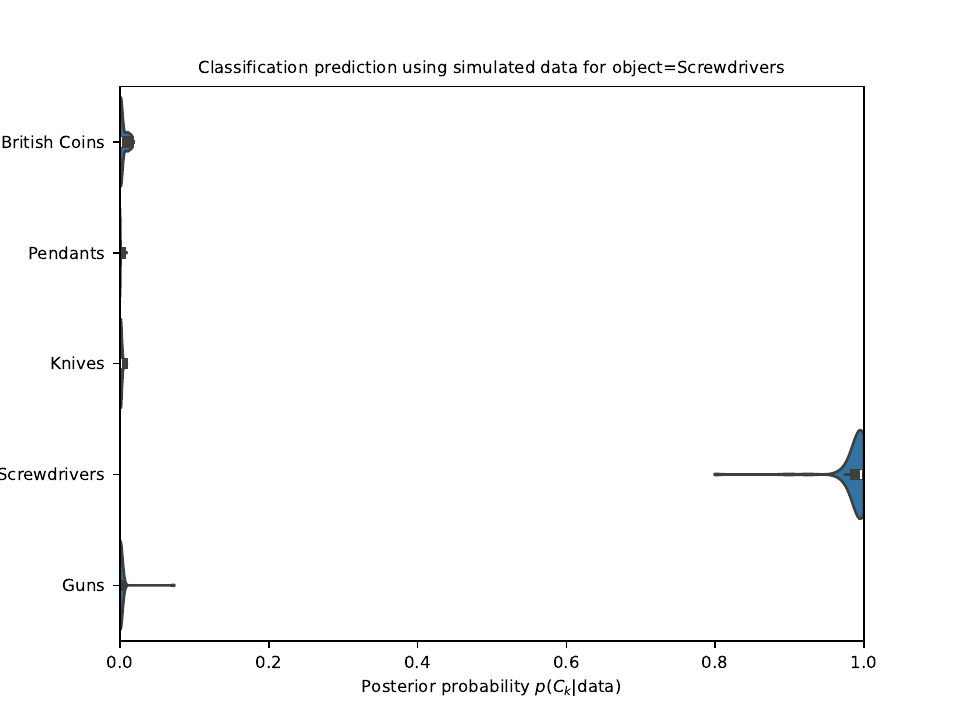} &
\includegraphics[width=0.5\textwidth]{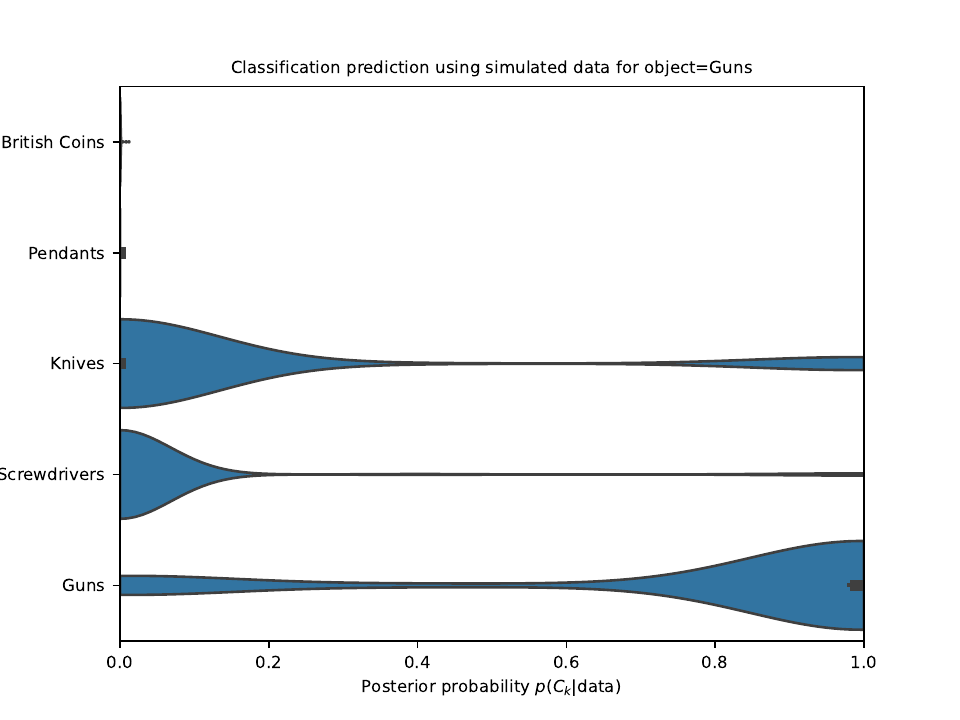}\\
(d) & (e) 
\end{array}
\end{array}$
\end{center}
\caption{Bayesian neural network classification:  Using a library with $N=1000$ objects and $M=100$ discrete frequencies where the correct classification is $(a)$ British coins, $(b)$ pendents, $(c)$ kitchen knives, $(d)$ screwdrivers and $(e)$ simplified gun models.} \label{fig:BayclassoptNN}
\end{figure}

\section{Conclusion}
{For symmetric matrices $A,B \in {\mathbb R}_s^{n\times n}$ this paper has presented new semi-metrics $|d_E^\pm(A,B)|$ and $d_C(A,B)$, and their associated approximate measures of distance $d_{E,\theta} A,B)$ and $d_{C,\theta}(A,B)$, which, for small angles, can be used to provide an approximation of the Riemannian metric $d_R(Q_A,Q_B)$ for the smallest distance between the rotational matrices $Q_A$ and $Q_B$. Our new results provide  an approximation of the extent to which $A$ and $B$  are non-commuting when the eigenvalues $\Lambda_A$ and $\Lambda_B$ are distinct, but may become close. They are beneficial in the case of  close eigenvalues, as they do not require the explicit computation of $Q_A$ and $Q_B$, whose accuracy deteriorates   when the eigenvalues become close using standard numerical solvers. Hence, in such cases, they offer advantages over traditional metrics for $SO(n)$, such as $d_R(Q_A,Q_B)$ and the  Euclidean metric $d_F(Q_A,Q_B)/\sqrt{2}$, which both require eigenvector information. The paper has shown that, in common with traditional metrics,  $|d_E^\pm(A,B)|$ and $d_C(A,B)$,  as well as $d_{E,\theta} A,B)$ and $d_{C,\theta}(A,B)$ are invariant under rotation of the dataset. }

{In addition to some illustrative examples, our main chosen application was to apply the new measures to outputs of finite element computations and we have undertaken an analysis to show that the approximations made in the numerical simulations do not pollute the new measures, provided the discretisation and regularisation are chosen to be sufficiently fine and small, respectively. }

We have applied our {new  approximate measures of  distance $d_{E,\theta}$ and $d_{C,\theta}$ to the MPT spectral signature evaluated at discrete frequencies  to aid with object characterisation in metal detection.  We have included numerical results to illustrate how peaks in the metrics and our approximate measures of distance are associated close eigenvalues and illustrated the superior performance of our approximate distance measures  in this case.
 We provide illustrative comparisons of $d_R$, $d_F/\sqrt{2}$ and $d_{E,\theta}$ and $d_{C,\theta}$ as a function of exciting frequency for several objects and include an example of Bayesian classification using the new distance measures as additional features.}

\section*{Acknowledgements}
P.D. Ledger is grateful for the financial support received from a ICMS KE Catalyst grant that is supported by EPSRC (U.K), which has helped to support part of this work.
P.D. Ledger and J. Elgy are grateful for the financial support  received from EPSRC (U.K) in the form of grant EP/V009028/1, which has helped to support part of this work. W.R.B. Lionheart is grateful to the financial support received from EPSRC in the form of grant  EP/V009109/1, which has helped to support part of this work.

\section*{Data statement}
{The open source  \texttt{MPT-Calculator} software used for computing the object characterisations can be accessed at \texttt{https://github.com/MPT-Calculator/MPT-Calculator} and version 1.51 with commit number 94a825c
 was employed in this work. The other data can be obtained on reasonable request from the corresponding author.}

\appendix
{
\section{Explicit computation of eigenvalues and eigenvectors of $A\in {\mathbb R}_s^{3 \times 3}$} \label{sect:appendixexplicteig}
We recall from~\cite{smitheigen} that Cardano's explicit computation of eigenvalues of $A\in {\mathbb R}_s^{3 \times 3}$ is given by
\begin{subequations}
\begin{align}
\lambda_1(A) = & m + \sqrt{p}\cos \phi, \\
\lambda_2(A) = & m - \sqrt{p}( \cos \phi + \sqrt{3} \sin \phi), \\
\lambda_3(A) = & m - \sqrt{p}( \cos \phi - \sqrt{3} \sin \phi) ,
\end{align}
\end{subequations}
where $3m =\text{tr}(A)$, $ 3\phi=\tan^{-1} (\sqrt{p^2-q^2}/{q})$ with $0 \le \phi \le \pi$, $2q=\text{det}(A-mI)$ and $6p$ is the sum of squares of $A-mI$.
}

{
Recall $B \, \text{adj} \,(B)  = \text{det} \, (B) I$  for any $B \in {\mathbb R}^{n \times n}$~\cite{strang}, so if $B$ is singular $\text{det} \, (B)=0 $ and the columns of $\text{adj} \,(B)$ are in the nullspace of $B$.  Moreover,  if $\text{rank}(B)=n-1$ then $\text{rank}\,(\text{adj}\,( B))) =1$~\cite{matrix_analysis}[Chapt 4.].}
{Now suppose $\lambda_i(A)$ is an eigenvalue of $A$ with algebraic multiplicity $1$ then  $B=A - \lambda_i I$ has rank $n-1$ and as $\text{rank}\, (\text{adj}\,  (B)) =1$, any non zero column of $\text{adj}\, (B) $ is an eigenvector $q_i$ of $A$~\cite{matrix_analysis}[Chapt 4.]. For the case of $n=3$, this can also be related to vector cross products. Writing
\begin{align}
B=A - \lambda_i I = \left ( \begin{array}{ccc}
(B)_{11} & (B)_{12} & (B)_{13} \\
(B)_{21} & (B)_{22} & (B)_{23} \\
(B)_{31} & (B)_{32} & (B)_{33} \end{array} \right)
\end{align}
then, forming the vectors
\begin{align}
{u} = \left ( \begin{array}{rrr}  (B)_{11} \\ (B)_{21}\\ (B)_{31} \end{array}
\right ), \qquad 
{v} = \left ( \begin{array}{rrr}  (B)_{12} \\ (B)_{22}\\ (B)_{32} \end{array}
\right ), \qquad 
{w} = \left ( \begin{array}{rrr}  (B)_{13} \\ (B)_{23}\\ (B)_{33} \end{array}
\right ), 
\end{align}
and  provided ${u}$ and ${v}$ are linearly dependent, the coefficients of ${ q}_i = {u} \times {v}$ (or equivalently up to a sign ${u} \times {w}$ or ${ v} \times { w}$ since ${ w}$ is a linear combination of ${u}$ and ${v}$ as $B$ has at most rank $2$) are also those of the eigenvector $q_i$. This can be seen by comparing the coefficients of $\text{adj}\,(B)$ to those of these cross products. If ${u}$ and ${v}$ are not linearly dependent then $ {u} \times {v} = {0}$ and is not an eigenvector.}

{Unfortunately, in finite precision arithmetic, using $\text{adj}\, (B) $ (or the cross product approach) can suffer from cancelation errors in the computation of $B$ and in $\text{adj}\,(B)$ and requires special treatment based for closely spaced eigenvalues, which additionally requires the detection of eigenvalue proximity~\cite{kopp,3x3techreport}.}

\section{Simple constant free a-posteriori error indicator} \label{sect:appendix}
This appendix presents a simple constant free error indicator to assess the impact of numerical regularisation. To do this, the alternative MPT tensor splitting $({\mathcal M})_{ij} = ({\mathcal N}^{(0)})_{ij} + ({\mathcal N}^{\sigma_*})_{ij}  - ({\mathcal C}^{\sigma_*})_{ij}$~\cite{LedgerLionheart2016} is used where
\begin{align}
({\mathcal N}^0)_{ij} & = \frac{\alpha^3}{2} \int_B ( 1 - \mu_r^{-1} ) {\bm e}_i \cdot \nabla \times {\bm \theta}_j^{(0)} \dif {\bm \xi},\\
({\mathcal C}^{\sigma_*})_{ij} & - \frac{\im \nu \alpha^3}{4} \int_B {\bm e}_i \cdot {\bm \xi} \times ( {\bm \theta}_j^{(0)} + {\bm \theta}_j^{(1)} ) \dif {\bm \xi},\\
({\mathcal N}^{\sigma_*})_{ij} & -\frac{\alpha^3}{4} \int_B (1- \mu_r^{-1})  {\bm e}_i \cdot \nabla \times {\bm \theta}_j^{(1)}  \dif {\bm \xi}.
\end{align}
Consider  a regularised formulation, where ${\bm \theta}_i^{(0)}$ is replace by ${\bm \theta}_i^{(0), \varepsilon} = \tilde{\bm \theta}_i^{(0), \varepsilon} +{\bm e}_i \times {\bm \xi} $ and ${\bm \theta}_i^{(1)}$ is replace by ${\bm \theta}_i^{(1), \varepsilon}$, which have strong forms: Find $ \tilde{\bm \theta}_i^{(0), \varepsilon}$ such that
\begin{subequations}
\begin{align}
\nabla \times {\mu}_r^{-1} \nabla \times \tilde{\bm \theta}_i^{(0), \varepsilon} + \varepsilon 
\tilde{\bm \theta}_i^{(0), \varepsilon} & = {\bm 0} && \text{in $B$} , \\
\nabla \times  \nabla \times \tilde{\bm \theta}_i^{(0), \varepsilon} + \varepsilon 
\tilde{\bm \theta}_i^{(0), \varepsilon} & = {\bm 0} && \text{in $\Omega \setminus B$},  \\
[ {\bm n} \times \tilde{\bm \theta}_i^{(0), \varepsilon}  ] & ={\bm 0} && \text{on $\Gamma$}, \\
[ {\bm n} \times \tilde{\mu}_r^{-1} \nabla \times  \tilde{\bm \theta}_i^{(0), \varepsilon}  ] & =
 -2 [\tilde{\mu_r}^{-1}] {\bm n }\times {\bm e}_i && \text{on $\Gamma$} ,\\
 {\bm n } \times  \tilde{\bm \theta}_i^{(0), \varepsilon}  & ={\bm 0} && \text{on $\partial \Omega$},
\end{align}
\end{subequations}
and  find: $ \tilde{\bm \theta}_i^{(1), \varepsilon}$ such that
\begin{subequations}
\begin{align}
\nabla \times {\mu}_r^{-1} \nabla \times {\bm \theta}_i^{(1), \varepsilon} - \im \nu {\bm \theta}_i^{(1), \varepsilon}  & =  \im \nu {\bm \theta}_i^{(0), \varepsilon} && \text{in $B$},\\
\nabla \times  \nabla \times {\bm \theta}_i^{(1), \varepsilon}   + \varepsilon {\bm \theta}_i^{(1), \varepsilon}& =  {\bm 0}&& \text{in $\Omega \setminus \overline{B}$},\\
[ {\bm n} \times {\bm \theta}_i^{(1), \varepsilon}] & = {\bm 0} && \text{on $\Gamma$}, \\
[ {\bm n} \times \mu_r^{-1} \nabla \times {\bm \theta}_i^{(1), \varepsilon}] & = {\bm 0} && \text{on $\Gamma$}, \\
{\bm n} \times {\bm \theta}_i^{(1), \varepsilon}   & = {\bm 0}  && \text{on $\partial \Omega$}.
\end{align}
\end{subequations}
If one chooses to replace  $   {\bm \theta}_i^{(0)} $ by ${\bm \theta}_i^{(0),\varepsilon}  =    \tilde{\bm \theta}_i^{(0), \varepsilon}+ {\bm e}_i \times {\bm \xi}$  and $   {\bm \theta}_i^{(1)} $ by ${\bm \theta}_i^{(1),\varepsilon} $ and performing integration by parts in a similar way to Theorem 3.2 and Theorem 5.1 of~\cite{LedgerLionheart2020spect}  then $({\mathcal N}^0)_{ij}$ and $({\mathcal N}^{\sigma_*})_{ij} -({\mathcal C}^{\sigma_*})_{ij}$ are replaced by
\begin{align}
( {\mathcal N}^{0})_{ij} =&  \alpha^3 \int_B [\mu_r] \delta_{ij} \dif {\bm \xi} +
\frac{\alpha^3}{4} \int_\Omega \tilde{\mu}_r^{-1} \nabla \times {\bm \theta}_i^{(0), \varepsilon} \cdot \nabla \times {\bm \theta}_j^{(0), \varepsilon} \dif {\bm \xi} +
\frac{\alpha^3\varepsilon }{4} \int_\Omega{\bm \theta}_i^{(0), \varepsilon} \cdot {\bm \theta}_j^{(0), \varepsilon} \dif {\bm \xi} ,\nonumber \\
 = & ({\mathcal N}^{0,\epsilon})_{ij} + \frac{\alpha^3\varepsilon }{4} \int_\Omega{\bm \theta}_i^{(0), \varepsilon} \cdot {\bm \theta}_j^{(0), \varepsilon} \dif {\bm \xi} ,\\
({\mathcal N}^{\sigma_*} )_{ij}-
( {\mathcal C}^{\sigma_* })_{ij}  = &  \frac{\im \alpha^3}{4} \int_B
\nabla \times \mu_r^{-1} \nabla \times {\bm \theta}_j^{(1),\varepsilon} \cdot
\nabla \times \mu_r^{-1} \nabla \times \overline{{\bm \theta}_i^{(1),\varepsilon} }\dif {\bm \xi}
\nonumber \\
&- \frac{\alpha^3}{4} \int_\Omega \tilde{\mu}_r^{-1} \nabla \times {\bm \theta}_j^{(1),\varepsilon} \cdot\nabla \times
 \overline{{\bm \theta}_i^{(1),\varepsilon} }\dif {\bm \xi} 
  +\frac{ \varepsilon \alpha^3}{4} \int_{\Omega \setminus \overline{B}} {\bm \theta}_j^{(1),\varepsilon} \cdot
 \overline{{\bm \theta}_i^{(1),\varepsilon} }\dif {\bm \xi} 
 \nonumber\\
&  -\frac{ \varepsilon \alpha^3}{4} \int_{\Omega \setminus \overline{B}} {\bm \theta}_j^{(1),\varepsilon} \cdot
{{\bm \theta}_i^{(0),\varepsilon} }\dif {\bm \xi} 
-\frac{\varepsilon \alpha^3}{4} \int_{\Omega \setminus \overline{B}} {\bm \theta}_j^{(0),\varepsilon} \cdot {\bm e}_i \times {\bm \xi}
 \dif {\bm \xi} \nonumber \\
  = &({\mathcal R}^\varepsilon )_{ij}  + \im({\mathcal I}^\varepsilon )_{ij}  \nonumber\\
  & +\frac{ \varepsilon \alpha^3}{4} \int_{\Omega \setminus \overline{B}} {\bm \theta}_j^{(1),\varepsilon} \cdot
 \overline{{\bm \theta}_i^{(1),\varepsilon} }\dif {\bm \xi}  -\frac{ \varepsilon \alpha^3}{4} \int_{\Omega \setminus \overline{B}} {\bm \theta}_j^{(1),\varepsilon} \cdot
{{\bm \theta}_i^{(0),\varepsilon} }\dif {\bm \xi} 
 -\frac{\varepsilon \alpha^3}{4} \int_{\Omega \setminus \overline{B}} {\bm \theta}_j^{(0),\varepsilon} \cdot {\bm e}_i \times {\bm \xi}
 \dif {\bm \xi} ,
\end{align}
respectively. Making approximations gives the error indicators
\begin{subequations}
\begin{align}
| ({\mathcal N}^{(0)})_{ij} - ({\mathcal N}^{(0),\varepsilon})_{ij} | &  =\left |  \frac{\varepsilon \alpha^3}{4} \int_{\Omega }  {\bm \theta}_j^{(0),\varepsilon} \cdot {\bm \theta}_j^{(0),\varepsilon}  \dif {\bm \xi} \right |
\approx
\left |  \frac{\varepsilon \alpha^3}{4} \int_{\Omega }  {\bm \theta}_j^{(0),hp} \cdot {\bm \theta}_j^{(0),hp}  \dif {\bm \xi} \right |
 , \\
| ({\mathcal R})_{ij} - ({\mathcal R}^{\varepsilon})_{ij} |  &=
\left | \text{Re} \left (
  \frac{ \varepsilon \alpha^3}{4}   \int_{\Omega \setminus \overline{B}} {\bm \theta}_j^{(1),\varepsilon} \cdot
 \overline{{\bm \theta}_i^{(1),\varepsilon} }\dif {\bm \xi}  +\frac{ \varepsilon \alpha^3}{4}  \int_{\Omega \setminus \overline{B}} {\bm \theta}_j^{(1),\varepsilon} \cdot
{{\bm \theta}_i^{(0),\varepsilon} }\dif {\bm \xi}  \right) \right . \nonumber \\
&\qquad \qquad \qquad \qquad \qquad \left . -\frac{\varepsilon \alpha^3}{4} \int_{\Omega \setminus \overline{B}} {\bm \theta}_j^{(0),\varepsilon} \cdot {\bm e}_i \times {\bm \xi} \dif {\bm \xi}  \right |   , \nonumber\\
  & \approx 
\left | \text{Re} \left (
  \frac{ \varepsilon \alpha^3}{4}   \int_{\Omega \setminus \overline{B}} {\bm \theta}_j^{(1),hp } \cdot
 \overline{{\bm \theta}_i^{(1), hp } }\dif {\bm \xi}  +\frac{ \varepsilon \alpha^3}{4}  \int_{\Omega \setminus \overline{B}} {\bm \theta}_j^{(1), hp } \cdot
{{\bm \theta}_i^{(0), hp } }\dif {\bm \xi}  \right) \right . \nonumber \\
&\qquad \qquad \qquad \qquad \qquad \left . -\frac{\varepsilon \alpha^3}{4} \int_{\Omega \setminus \overline{B}} {\bm \theta}_j^{(0), hp} \cdot {\bm e}_i \times {\bm \xi} \dif {\bm \xi}  \right |   , \\
| ({\mathcal I})_{ij} - ({\mathcal I}^{\varepsilon})_{ij} |  &  =\left |  \text{Im} \left (
  \frac{ \varepsilon \alpha^3}{4} \int_{\Omega \setminus \overline{B}} {\bm \theta}_j^{(1),\varepsilon} \cdot
 \overline{{\bm \theta}_i^{(1),\varepsilon} }\dif {\bm \xi}  -\frac{ \varepsilon \alpha^3}{4} \int_{\Omega \setminus \overline{B}} {\bm \theta}_j^{(1),\varepsilon} \cdot
{{\bm \theta}_i^{(0),\varepsilon} }\dif {\bm \xi}  \right) \right | , \nonumber \\
& \approx  \left |  \text{Im} \left (
  \frac{ \varepsilon \alpha^3}{4} \int_{\Omega \setminus \overline{B}} {\bm \theta}_j^{(1), hp } \cdot
 \overline{{\bm \theta}_i^{(1), hp} }\dif {\bm \xi}  -\frac{ \varepsilon \alpha^3}{4} \int_{\Omega \setminus \overline{B}} {\bm \theta}_j^{(1), hp} \cdot
{{\bm \theta}_i^{(0), hp} }\dif {\bm \xi}  \right) \right | .
\end{align}
\end{subequations}
From which
\begin{subequations}
\begin{align} 
 \left \|  {\mathcal Z} - {\mathcal Z}^\varepsilon  \right \| \approx & \Delta ( \| {\mathcal Z} \|) :=
2 \|    {\mathcal R} - {\mathcal R}^\varepsilon \| \|  {\mathcal I}^\varepsilon \| +
2 \|   {\mathcal I} - {\mathcal I}^\varepsilon \| \|  {\mathcal R}^\varepsilon \| , \label{eqn:indiccommerr} \\
\left \| {\mathcal Z}^{(0)} - {\mathcal Z}^{(0),\varepsilon}  \right \| \approx & \Delta ( \| {\mathcal Z}^{(0)} \|) :=  2\|   {\mathcal N}^{(0)} - {\mathcal N}^{(0),\varepsilon}  \| \|  {\mathcal I}^\varepsilon  \|
+2\|  {\mathcal N}^{(0),\varepsilon} \| \|   {\mathcal I}- {\mathcal I}^\varepsilon \|, \label{eqn:indiccommerr0}  \\
 \left \| \tilde{\mathcal Z} - \tilde{\mathcal Z}^\varepsilon \right \| \approx &  \Delta ( \| \tilde{\mathcal Z} \|) := 
2 \|   {\mathcal R} - {\mathcal R}^\varepsilon \| \| {\mathcal I}^\varepsilon \| +
2 \|   {\mathcal I} - {\mathcal I}^\varepsilon \| \|  {\mathcal R}^\varepsilon \| + 2 \|   {\mathcal N} ^0- {\mathcal N}^{0,\varepsilon} \| \| {\mathcal I}^\varepsilon \| , \label{eqn:indiccommerrtil}
\end{align}
\end{subequations}
can be approximated.

\bibliographystyle{acm}
\bibliography{paperbib}
\end{document}